\documentclass[11pt]{amsart}
\usepackage{amsmath}
\usepackage{amssymb}
\usepackage{amscd}

\usepackage{url}
\usepackage[all]{xypic}

\usepackage{pstricks}
\usepackage{pstricks-add}
\usepackage{multido}
\usepackage{pst-solides3d}

\usepackage{graphicx}

\newcommand{\modfrac}[2]{\genfrac{|}{|}{}{1}{#1}{#2}}
\newcommand{\bp}{\begin{pmatrix}}
\newcommand{\ep}{\end{pmatrix}}
\newcommand{\be}{\begin{equation}}
\newcommand{\ee}{\end{equation}}
\newcommand{\ol}[1]{\overline{#1}}
\newcommand{\e}{\varepsilon}

\numberwithin{equation}{section}

\theoremstyle{plain}
\newtheorem{theorem}[equation]{Theorem}
\newtheorem{lemma}[equation]{Lemma}
\newtheorem{proposition}[equation]{Proposition}

\newtheorem{corollary}[equation]{Corollary}

\theoremstyle{definition}
\newtheorem{example}[equation]{Example}
\newtheorem{remark}[equation]{Remark}

\newtheorem{definition}[equation]{Definition}

\numberwithin{equation}{section}

\newtheorem*{acknowledgements}{Acknowledgements}

\def\Z{\mathbb Z}
\def\R{\mathbb R}

\def\O{\Omega}
\def\wt#1{\widetilde{#1}}

\def\p{\partial}

\def\vu{\vec{u}}

\def\S{\Sigma}
\def\DU{U_{21}}
\def\Hr{H_\textit{\!right}}
\def\Hl{H_\textit{\!left}}
\DeclareMathOperator\Int{Int}

\DeclareMathOperator\ind{ind}
\newrgbcolor{lightgreen}{0.95 1.0 0.95}
\newrgbcolor{lightgreen}{0.95 1.0 0.95}
\newrgbcolor{lightgreenbutnotsolight}{0.85 1.0 0.85}
\newrgbcolor{darkgreen}{0.2 0.6 0.2}
\newrgbcolor{lightblue}{0.8 0.8 1.0}
\newrgbcolor{moderateblue}{0.5 0.5 1.0}
\newrgbcolor{orange}{1.0 0.7 0.0}
\def\purple{\color{purple}}
\title[Morse theory for manifolds with boundary]{Morse theory for manifolds with boundary}
\author{Maciej Borodzik}
\address{Institute of Mathematics, University of Warsaw, ul. Banacha 2,
02-097 Warsaw, Poland}
\email{mcboro@mimuw.edu.pl}
\thanks{The first author is supported by Polish MNiSzW Grant No N201 397937.
The second author is partially supported by OTKA Grant K100796. }
\author{Andr\'as N\'emethi}
\address{A. R\'enyi Institute of Mathematics, 1053 Budapest,  Re\'altanoda u. 13-15,  Hungary.}
\email{nemethi@renyi.hu}
\author{Andrew Ranicki}
\address{School of Mathematics, University of Edinburgh, Edinburgh EH9 3JZ, Scotland UK.}
\email{a.ranicki@ed.ac.uk}

\date{\today}
\makeatletter
\@namedef{subjclassname@2010}{\textup{2010} Mathematics Subject Classification}
\makeatother

\begin{document}
\begin{abstract}
We develop Morse theory for manifolds with boundary. Beside standard and expected facts like the handle cancellation theorem and
the  Morse lemma for manifolds with boundary, we prove that under suitable connectedness assumptions
a critical point in the interior of a Morse function can be moved to the boundary, where it splits into a pair of boundary critical points. As an application,
we prove that every cobordism of connected manifolds with boundary splits as a union of left product cobordisms and right product cobordisms.
\end{abstract}
\maketitle

\section{Introduction}

For some time now, Morse theory has been a very fruitful tool in
the topology of manifolds.  One of the milestones was the
$h$-cobordism theorem of Smale \cite{Sm2}, and its Morse-theoretic
exposition by Milnor \cite{Mi-morse,Mi-hcob}. Recently, Morse
theory has become even more popular, for two reasons.  In the
first instance, on account of its connections with Floer homology,
 see e.g. \cite{Sa,Wi,Nic,KM}. Secondly, the stratified Morse theory
developed by Goresky and MacPherson \cite{GMcP}. In the last 20
years Morse theory has also had an enormous impact on the singularity
theory of complex algebraic and analytic varieties.

Morse theory for manifolds with boundary was studied in the seventies by \cite{Br, JR, Haj}.
Recently it was developed by Kronheimer and Mrowka in \cite{KM}.
Since then the theory has experienced a very fast development, as witnessed by the papers of Bloom \cite{Blo} and Laudenbach \cite{Lau}. Our paper is another
contribution.

In this paper we prove some new results in the Morse theory for
manifolds with boundary.  Besides some standard and expected
results, like the boundary handle cancellation theorem
(Theorem~\ref{thm:elementarycanc}) and the topological description
of passing critical points on the boundary (using the notions of
right and left half-handles introduced in Section~\ref{s:hh-sec})
we describe another phenomenon; see Theorem~\ref{thm:handsplit}.
An interior critical point can be
moved to the boundary and there split into two boundary critical
points. A related result was stated in \cite[Theorem 5]{Haj}; we provide
a rigorous proof under much weaker assumptions.

In particular, if we have a cobordism of manifolds with
boundary, then under a natural topological assumption we can find a
Morse function which has only boundary critical points.  We use
this result to prove a structure theorem for connected cobordisms
of connected manifolds with connected non-empty boundary: such a
cobordism splits as a union of left and right product cobordisms.
This is a topological counterpart to the
algebraic splitting of cobordisms obtained in \cite[Main Theorem 1]{BNR2}: an
algebraic splitting of the chain complex cobordism of a geometric
cobordism can be realized topologically by a geometric splitting.
This algebraic splitting is used to study the algebraic properties of
the Seifert matrices of isotopic non-spherical
$2n-1$ dimensional links in $S^{2n+1}$. This will provide the algebraic background
to our proof that the semicontinuity of mod 2 spectra of hypersurface singularities is a purely topological phenomenon (see \cite{BNR3}, especially the paragraph before proof of Theorem 2.1.8).

The structure of the paper is as follows.  After preliminaries
in Section~\ref{ss:prelim} we study in Section~\ref{s:hh-sec} the changes
in the topology of the level sets when crossing a boundary critical point.
Theorem~\ref{thm:morsehalfhandle} is
the main result of this section: passing a boundary stable (unstable) critical
point produces a left (right) half-handle attachment. In Section~\ref{S:sih} we
prove Theorem~\ref{thm:handsplit}, which
moves interior critical points to the boundary. This is the first main result of
this article. Then we pass to
some more standard results, namely rearrangements of critical
points in Section~\ref{S:rear}. We finish the section with our
most important --- up to now --- application,
Theorem~\ref{thm:cobsplit}, about the splitting of a cobordism into
left product and right product cobordisms.
Finally, in Section~\ref{S:bhc} we discuss the possibility of cancelling a pair
of critical points. We include this part for completeness, the main result of this
section was proved in \cite[Theorem 1]{Haj}.

\begin{acknowledgements}
The first author wishes to thank R\'enyi Institute for hospitality, the first two authors
are grateful for Edinburgh Mathematical Society for a travel grant to Edinburgh and Glasgow in March 2012.
The authors thank Andr\'as Juh\'asz, Mark Powell and Andr\'as Stipsicz for fruitful discussions and to Rob Kirby for drawing their
attention to \cite{KM}. We are grateful to Bogus\l{}aw Hajduk for his comments on the first version of the paper and for
drawing our attention to \cite{JR,Haj}.
\end{acknowledgements}

\subsection{Notes on gradient vector fields}\label{ss:prelim}

To fix the notation, let us recall what a cobordism of manifolds with boundary is.

\begin{definition}
Let $\S_0$ and $\S_1$ be compact oriented, $n$-dimensional
manifolds with non-empty boundary $M_0$ and $M_1$. We shall say
that $(\O,Y)$ is a \emph{cobordism} between $(\S_0,M_0)$ and
$(\S_1,M_1)$, if $\O$ is a compact, oriented $(n+1)$-dimensional
manifold with boundary $\p \O=Y\cup \S_0\cup \S_1$, where $Y$ is
nonempty, $\S_0\cap\S_1=\emptyset$, and
$Y\cap \S_0=M_0$, $Y\cap\S_1=M_1$.
\end{definition}
\begin{remark}\label{rem:corners}
Strictly speaking, $\O$ is a manifold with corners, so around a point $x\in M_0\cup M_1$
it is locally modelled by $\R^{n-1}\times\R_{\geqslant 0}^2$.
Accordingly, sometimes
we write that $\S_0$, $\S_1$ and $Y$, as manifolds with boundary, have tubular neighbourhoods in $\O$ of the form $\S_0\times[0,1)$,
$\S_1\times[0,1)$, or $Y\times[0,1)$, respectively. Nevertheless,
in most cases it is safe (and more convenient) to assume that
$\O$ is a manifold with boundary, i.e. that the corners are smoothed along $M_0$ and $M_1$. Whenever possible we
make this simplification in order to avoid unnecessary technicalities.
\end{remark}
\begin{example}
Given a manifold with boundary $(\S,M)$, we call
$(\S,M)\times[0,1]$ a trivial cobordism, with $\O=\S\times[0,1]$,
$Y=M\times[0,1]$, $\S_{i}=\S\times\{i\}$, $M_i=M\times\{i\}$ for
$i=0,1$.
\end{example}

We recall the notion of a Morse function (in \cite{Haj} they are called $m$-functions). For this it is convenient to
fix a Riemannian metric $g$ on $\Omega$.

\begin{definition}\label{def:Morse}
Let $F\colon\O\to[0,1]$ be a smooth function. A critical point $z$ of $F$ is called \emph{Morse}, if
the Hessian of $F$ at $z$ is non--degenerate.
The function $F\colon \O\to[0,1]$ is called a \emph{Morse function}
on the cobordism $(\O,Y)$, if $F(\S_0)=0$, $F(\S_1)=1$, $F$
has only Morse critical points, the  critical points are not situated on $\S_0\cup\S_1$,
and $\nabla F$ is everywhere tangent to $Y$.
\end{definition}

There are two ways of doing Morse theory on manifolds. One can
either consider the gradient flow of $\nabla F$ associated with $F$ and
the Riemannian metric  (in the Floer
theory, one often uses $-\nabla F$), or, the so-called
gradient-like vector field.

\begin{definition}(See \expandafter{\cite[Definition 3.1]{Mi-hcob}}.)
Let $F$ be a Morse function on a cobordism $(\O,Y)$. Let
$\xi$ be a vector field on $\O$. We shall say that $\xi$ is
\emph{gradient-like with respect to $F$}, if the following
conditions are satisfied:
\begin{itemize}
\item[(a)] $\xi\cdot F>0$ away from the set of critical points
    of $F$;
\item[(b)] if $p$ is a critical point of $F$ of index $k$, then
    there exist local coordinates $x_1,\dots,x_{n+1}$ in a
    neighbourhood of $p$, such that
\[F(x_1,\dots,x_{n+1})=F(p)-(x_1^2+\dots+x_k^2)+(x_{k+1}^2+\dots+x_{n+1}^2)\]
and
\[\xi=(-x_1,\dots,-x_k,x_{k+1},\dots,x_{n+1})\text{ in $U$};\]
\item[(b')] furthermore, if $p$ is a boundary critical point,
    then the above coordinate system can be chosen so that
    $Y=\{x_j=0\}$ and $U=\{x_j\geqslant 0\}$  for some $j\in\{1,\dots,n+1\}$.
\item[(c)] $\xi$ is everywhere tangent to $Y$;
\end{itemize}
\end{definition}

The conditions (a) and (b) are the same as in the classical case.
Condition (b') is an analogue of condition (b) in the
boundary case, compare Lemma~\ref{lem:boundarymorselemma}.

Smale in \cite{Sm1} noticed that for any gradient-like vector
field $\xi$ for a function $F$ there exists a choice of the
Riemannian metric such that $\xi=\nabla F$ in that metric.  The
situation is identical in the boundary case. This is stated explicitly
in the following lemma, whose proof is straightforward and will be omitted.

\begin{lemma}\label{lem:linktwo}
Let $U$ be a paracompact $k$-dimensional manifold and $F\colon
U\to\R$ a Morse function without critical points. Assume that
$\xi$ is a gradient--like vector field on $U$. Then there exists
a metric $g$ on $U$ such that $\xi=\nabla F$ in that metric.

A similar  statement holds if $U$ has boundary and $\xi$ is everywhere tangent to the boundary.
\end{lemma}

Hence the two approaches -- gradients and gradient-like vector fields --
 are equivalent. However, we shall need both approaches. In
Section~\ref{S:sih} we use gradients of functions and a specific
choice of a metric, because the argument becomes slightly simpler.
In Section~\ref{S:bhc} we follow \cite{Mi-hcob} very closely; as he
uses gradient-like vector fields, we use them as well.

The next result shows that the condition from Definition~\ref{def:Morse} that $\nabla F$ is
everywhere tangent to $Y$ can be relaxed. We shall use this result in
Proposition~\ref{prop:rearii}.

\begin{lemma}\label{lem:altermet}
Let $\O$ be a compact, Riemannian manifold of dimension $(n+1)$
and let $Y\subset\p\O$ be compact as well. Let $g$ denote the metric.
Suppose that there exists a function $F\colon\O\to\R$, and a
relative open subset $U\subset Y$ such that $\nabla F$ is tangent
to $Y$ at each point $y\in U$. Suppose furthermore, that for any
$y\in Y\setminus U$ we have
\begin{equation}\label{eq:kerdfnot}
T_yY\not\subset\ker dF.
\end{equation}
Then, for
any open neighbourhood $W\subset\O$ of $Y\setminus U$, there exist a metric
$h$ on $\O$, agreeing with $g$ away from $W$, such that $\nabla_h
F$ (the gradient in the new metric) is everywhere tangent to $Y$.
\end{lemma}
\begin{proof}
Let us fix a point $y\in Y\setminus U$ and consider a small open
neighbourhood $V_y$ of $y$ in $W$, in which we choose local
coordinates $x_1,\dots,x_{n+1}$ such that $Y\cap
V_y=\{x_{n+1}=0\}$ and $V_y\subset\{x_{n+1}\geqslant 0\}$. In
these coordinates we have $dF=\sum_{i=1}^{n+1} f_i(x)\,dx_i$ for
some smooth functions $f_1,\dots,f_{n+1}$. By \eqref{eq:kerdfnot}, for
each $x\in V_y$, there exists  $i\leqslant n$ such that $f_i(x)\neq 0$.
Shrinking $V_y$ if needed, we may assume that the index $i$ is the same for each $x\in V_y$.
In the following we suppose that $i=1$, that is for every $x\in V_y$ we have
$\pm f_1(x)>0$ and the sign $\pm$ is the same for every $x$. Let us choose a symmetric positive definite
matrix $A_y=\{a_{ij}(x)\}_{i,j=1}^{n+1}$ so that $a_{11}=\pm
f_1(x)$ and for $i>1$, $a_{1i}=a_{i1}=f_i(x)$. $A_y$ defines a
metric $h_y$ on $V_y$ such that $\nabla_{h_y}F=(\pm
1,0,\dots,0)\subset TY$ in that metric.

Now let us choose an open subset $V$ of $\O\setminus (Y\setminus U)$ such that
$V\cup\bigcup_{y\in Y\setminus U} V_y$
is a covering of $\O$. Let $\{\phi_V\}\cup\{\phi_y\}_{y\in Y\setminus U}$
be a partition of unity subordinate to this covering. Define
\[h=\phi_V\cdot g+\sum_{y\in Y\setminus U}\phi_yh_y.\]
Then $h$ is a metric, which agrees with $g$ away from $W$. Moreover, as for each
metric $h_y$, and $x\in V_y\cap Y$ we have $\nabla_{h_y}F(x)\in T_xY$
by construction, the same holds for a convex linear combination of metrics.
\end{proof}

\section{Boundary stable and unstable critical points}\label{s:hh-sec}

Most of the results of this section appeared previously in \cite{Br,JR,Haj}. We provide them for completeness of exposition
and for the convenience of the reader. In what follows, we use notation and terminology of \cite{KM}.

\subsection{Morse function for manifolds with boundary}

The first question concerns the existence of Morse functions. While the condition that the function has
only critical points of Morse type is open--dense, it requires a little argument to show that
there are many functions generic in the interior such that their gradient, when restricted to $Y$ is tangent to $Y$.

\begin{lemma}\label{lem:morseexist}
Morse functions exist. In fact, for any Morse function $f\colon Y\to[0,1]$ with
$f(M_0)=0$, $f(M_1)=1$ there exists a Morse function $F\colon \O\to[0,1]$ whose restriction to $Y$
is $f$.
\end{lemma}

\begin{proof}
Let $f\colon Y\to[0,1]$ be a Morse function on the boundary,
such that $f(M_0)=0$ and $f(M_1)=1$. We want to extend $f$ to a Morse function on $\O$.

First, let us choose a small tubular neighbourhood $U$ of $Y$ and a
diffeomorphism $U\cong Y\times[0,\varepsilon)$ for some $\varepsilon>0$.
Let $\tilde{F}\colon U\to[0,1]$ be given by the formula
\begin{equation}\label{eq:choiceofsign}
U\cong Y\times[0,\varepsilon)\ni (x,t)\to \tilde{F}(x,t)=f(x)-f(x)(1-f(x))t^2.
\end{equation}
The factor $f(x)(1-f(x))$ ensures that $\tilde{F}$ attains values in the interval
$[0,1]$ and $\tilde{F}^{-1}(i)\subset \S_i$ for $i\in\{0,1\}$.
It is obvious that there exists a smooth function $F\colon \O\to[0,1]$, which agrees on $Y\times[0,\varepsilon/2)$ with $\tilde{F}$, and it
satisfies the Morse condition on the whole of $\O$. The gradient $\nabla F$ is everywhere tangent to $Y$.

\end{proof}
\begin{remark}\label{rem:allarestab}
The above construction yields a function with the property that  all its
boundary critical points are boundary stable (see
Definition~\ref{def:bstab} below). This is due to the choice of
sign $-1$ in front of $f(x)(1-f(x))t^2$ in
\eqref{eq:choiceofsign}. If we change the sign to $+1$, we obtain
a function with all boundary critical points boundary unstable.
\end{remark}

We fix a Morse function $F\colon \O\to [0,1]$ and we start to
analyze its critical points. Let $z$ be such a point. If $z\in \O\setminus Y$,
we shall call it an \emph{interior} critical point. If $z\in Y$, it will be called a \emph{boundary} critical point.
There are two types of  boundary critical points.

\begin{definition}\label{def:bstab}
Let $z$ be a boundary critical point. We shall call it
\emph{boundary stable}, if the tangent space to the unstable
manifold of $z$ lies entirely in $T_zY$, otherwise it is called
\emph{boundary unstable}.

\end{definition}

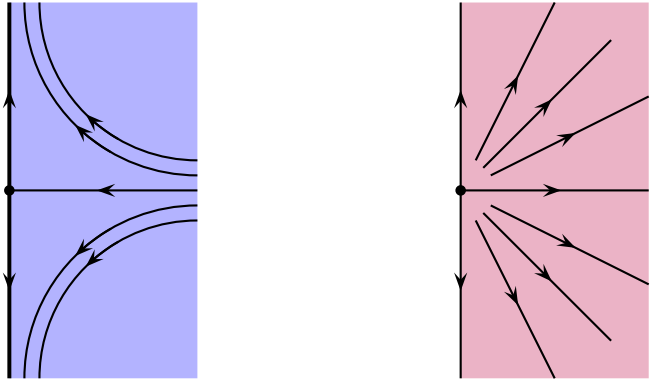
\begin{figure}
\begin{pspicture}(-5,-3)(5,3)
\pspolygon[linestyle=none,fillstyle=solid,fillcolor=blue,opacity=0.3](-4,-2.5)(-4,2.5)(-1.5,2.5)(-1.5,-2.5)
\psline[ArrowInside=->,arrowsize=5pt,linewidth=1.5pt](-4,0)(-4,2.5)
\psline[ArrowInside=->,arrowsize=5pt,linewidth=1.5pt](-4,0)(-4,-2.5)
\psline[ArrowInside=->,arrowsize=5pt](-1.5,0)(-4,0)
\psarc[arrowsize=5pt]{->}(-1.5,-2.5){2.3}{90}{135}
\psarc(-1.5,-2.5){2.3}{120}{180}
\psarc[arrowsize=5pt]{->}(-1.5,-2.5){2.1}{90}{135}
\psarc(-1.5,-2.5){2.1}{120}{180}
\psarc[arrowsize=5pt]{<-}(-1.5,2.5){2.3}{225}{270}
\psarc(-1.5,2.5){2.3}{180}{240}
\psarc[arrowsize=5pt]{<-}(-1.5,2.5){2.1}{225}{270}
\psarc(-1.5,2.5){2.1}{180}{240}
\pscircle[fillstyle=solid,fillcolor=black](-4,0){0.07}
\pspolygon[linestyle=none,fillstyle=solid,fillcolor=purple,opacity=0.3](2,-2.5)(2,2.5)(4.5,2.5)(4.5,-2.5)
\psline[ArrowInside=->,arrowsize=5pt](2,0)(2,2.5)
\psline[ArrowInside=->,arrowsize=5pt](2,0)(2,-2.5)
\psline[ArrowInside=->,arrowsize=5pt](2,0)(4.5,0)
\pscircle[fillstyle=solid,fillcolor=black](2,0){0.07}
\psline[ArrowInside=->,arrowsize=5pt](2.4,0.2)(4.5,1.25)
\psline[ArrowInside=->,arrowsize=5pt](2.3,0.3)(4,2)
\psline[ArrowInside=->,arrowsize=5pt](2.2,0.4)(3.25,2.5)
\psline[ArrowInside=->,arrowsize=5pt](2.4,-0.2)(4.5,-1.25)
\psline[ArrowInside=->,arrowsize=5pt](2.3,-0.3)(4,-2)
\psline[ArrowInside=->,arrowsize=5pt](2.2,-0.4)(3.25,-2.5)
\end{pspicture}
\caption{Boundary stable (on the left) and unstable critical points.}\label{fig:boundstab}
\end{figure}

The index of the boundary critical point $z$ is defined as the dimension of the stable manifold $W^s_z$. If $z$ is boundary unstable, this is the same
as the index of $z$ regarded as a critical point of the restriction $f$ of  $F$ on $Y$.
If $z$ is boundary stable, we have $\ind_F z=\ind_f z+1$. In particular, there
are no boundary stable critical point with index $0$, nor boundary unstable critical points of index $n+1$.

\begin{remark}
We point out that we use the flow of $\nabla F$ and not of
$-\nabla F$ as Kronheimer and Mrowka \cite{KM} do, hence our
definitions and formulae are slightly different from theirs.
\end{remark}

We finish this subsection with three standard results.

\begin{lemma}[Boundary Morse Lemma]\label{lem:boundarymorselemma}
Assume that $F$ has a critical point $z\in Y$ such
that the Hessian $D^2F(z)$ at $z$ is non-degenerate, and $\nabla F$ is
everywhere tangent to $Y$. Then there are local coordinates
$(x_1,\dots,x_{n+1})$  in an open neighbourhood $U\ni z$
such that
$U=\{x_1^2+\dots+x_{n+1}^2\leqslant\e^2\}\cap\{x_1\geqslant
0\}$ and $U\cap Y=\{x_1=0\}$  for some $\e>0$,
and
$F$ in these coordinates has the form $\pm x_1^2\pm x_2^2\pm \dots\pm
x_{n+1}^2+F(z)$.
\end{lemma}
\begin{proof}
We choose a
coordinate system $y_1,\dots,y_{n+1}$ in a neighbourhood
$U\subset\O$ of $z$ such that  $z=(0,\dots,0)$, $Y=\{y_1=0\}$,
$U=\{y_1\geqslant 0\}$, and the vector field $\frac{\p}{\p y_1}$ is
orthogonal to $Y$. We may and will assume  $F(z)=0$.
The tangency of $\nabla F$ to $Y$ implies
that at each point of $Y$
\begin{equation}\label{eq:dF}
\frac{\p F}{\p y_1}(0,y_2,\dots,y_{n+1})=0.
\end{equation}
The Hadamard Lemma applied to $F$ gives smooth functions $K_1,\ldots,K_{n+1}$ such that.
\begin{equation}\label{eq:feqkj}
F=y_1K_1(y_1,\ldots,y_{n+1})+\sum_{j=2}^{n+1}y_jK_j(y_1,y_2,\dots,y_{n+1}),
\end{equation}
We can assume that for $j>1$, $K_j$ does not depend on $y_1$. Indeed, if it does depend, we write (again using the Hadamard Lemma)
\[K_j(y_1,\dots,y_{n+1})=K_j(0,y_2,\dots,y_{n+1})+y_1L_{1j}(y_1,\dots,y_{n+1})\]
for some functions smooth $L_{12},\ldots,L_{1,n+1}$,
and then replace $K_j$ by $K_j(0,y_2,\dots,y_{n+1})$ and $K_1$ by $K_1+\sum y_jL_{1j}$.
The condition \eqref{eq:dF} implies now that $K_1(0,y_2,\dots,y_{n+1})=0$, hence
\[K_1(y_1,\ldots,y_{n+1})=y_1H_{11}(y_1,\dots,y_{n+1})\]
for some function $H_{11}$.
By the Hadamard Lemma  applied to $K_2,\dots,K_{n+1}$ we infer that there exist functions $H_{jk}$ for $j,k=2,\ldots,n+1$ such that
\begin{equation}\label{eq:dF2}
F=y_1^2H_{11}(y_1,\dots,y_{n+1})+\sum_{j,k=2}^n y_jy_k H_{jk}(y_2,\dots,y_{n+1}).
\end{equation}

Notice that the functions $H_{11}$ and $H_{jk}$ for $j,k=2,\ldots,n+1$ evaluated at $z$ correspond to the second derivatives of $F$ at $z$.
The non-degeneracy of $D^2F(z)$
implies that $H_{11}(z)\neq 0$; by continuity $H_{11}$ does not vanish in a neighbourhood of $z$.
After replacing $y_1\sqrt{\pm H_{11}}$ by $x_1$, we can assume that $H_{11}=\pm 1$.
Finally, the sum in \eqref{eq:dF2} can be written as $\sum_{j\geqslant2}\epsilon_j x_j^2$ ($\epsilon_j=\pm 1$)
by the classical Morse lemma
\cite[Lemma~2.2]{Mi-morse}.
\end{proof}

The next result is completely standard by now.
\begin{lemma}\label{lem:novanishnoproblem}
Assume that $F$ is a Morse function on a cobordism $(\O,Y)$ between $(\S_0,M_0)$ and $(\S_1,M_1)$. If $F$ has no critical points then
$(\O,Y)\cong (\S_0,M_0)\times[0,1]$.
 Furthermore, we can choose the diffeomorphism to map the level set $F^{-1}(t)$ to the set $\S_0\times\{t\}$.
\end{lemma}
\begin{proof}
The proof is identical to the classical case, see e.g. \cite[Theorem 3.4]{Mi-hcob}.
\end{proof}

\subsection{Half-handles}\label{ss:halfhandles}

For any $k$ we consider the $k$--dimensional disk
$D^k=\{x_1^2+\dots+x_{k}^2\leqslant 1\}$.
In the classical theory, an $n$--dimensional handle of index $k$ is the
$n$--dimensional manifold $H=D^k \times D^{n-k}$ with boundary
\[\p H~=~\left(\partial D^k \times D^{n-k}\right) \cup \left(D^k \times \partial D^{n-k}\right)~=~B_0 \cup B'_0~.\]

Given an $n$--manifold with boundary $(\S,\p \S)$ and a
distinguished embedding $\phi\colon B_0\to \p \S$, the effect of a
classical handle attachment is the $n$-dimensional manifold with
boundary
$$(\S',\p \S')~=~(\S\cup H,(\p \S \setminus B_0) \cup B'_0),$$
where we glue along  $\phi(B_0)$ identified with $B_0$. The
boundary $\p \S'$ is the effect of surgery on $\phi(B_0) \subset
\p \S$. We now extend this construction to relative cobordisms of
manifolds with boundary, using `half-handles'. Since our ambient
space $\O$ is $(n+1)$--dimensional, $(n+1)$ is the dimension of
the handles, and they induce $n$--dimensional handle attachments
on $Y$.

In order to do this, for any $k\geqslant1$ we distinguish the following subsets of $D^k$:
the `half-disk' $D^{k}_+:= D^k\cap \{x_1\geqslant0\}$, and its boundary subsets
$S^{k-1}_+:=\partial D^k\cap \{x_1\geqslant0\}$, $S^{k-2}_0:=\partial D^k\cap \{x_1= 0\}$
and $D^{k-1}_0:=D^k\cap \{x_1= 0\}$. Clearly, $S^{k-2}_0$ is a boundary the two $(k-1)$--disks
 $S^{k-1}_+$ and $D^{k-1}_0$; see Figure~\ref{fig:halfdisks}. We will call $x_1$ the {\it cutting coordinate}.

\begin{figure}
\begin{pspicture}(-3,-1.5)(3,1.5)
\psarc[fillcolor=blue,opacity=0.2,fillstyle=solid,linestyle=none](0,0){1.3}{270}{90}%
\rput(0.6,0){\psscalebox{0.7}{$D^2_+$}}%
\psarc[linewidth=1.5pt, linestyle=dashed,linecolor=blue](0,0){1.3}{270}{90}\rput(1.55,0){\psscalebox{0.7}{$S^1_+$}}%
\psline[linewidth=1.5pt, linestyle=solid,linecolor=purple](0,-1.3)(0,1.3)\rput(-0.25,0){\psscalebox{0.7}{$D^1_0$}}%
\pscircle[fillstyle=solid,fillcolor=black,linestyle=none](0,1.3){0.1}\rput(0,1.6){\psscalebox{0.7}{$S^0_0$}}%
\pscircle[fillstyle=solid,fillcolor=black,linestyle=none](0,-1.3){0.1}\rput(0,-1.6){\psscalebox{0.7}{$S^0_0$}}%
\end{pspicture}
\caption{Various parts of a `half-disk'. We explain the notation introduced before Definition~\ref{def:rhh}.}\label{fig:halfdisks}
\end{figure}
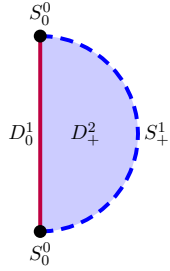
\begin{definition}\label{def:rhh}
Let $0\leqslant k \leqslant n$. An \emph{$(n+1)$-dimensional right
half-handle of index $k$} is the $(n+1)$-dimensional manifold
$\Hr=D^k \times D^{n+1-k}_+ $, with boundary subdivided into
three pieces  $\partial \Hr=B \cup C \cup N$, where
$$B:=\partial D^k\times D_+^{n+1-k},\ \ \ C:=D^k\times D_0^{n-k},
\ \ \ N:=D^k\times S^{n-k}_+.$$ One has the following
intersections too
$$
B_0:=C \cap B=\partial D^k \times D^{n-k}_0,\ \ \
N_0:=C \cap N=D^k \times S^{n-k-1}_0.$$
Hence the handle $H$ is cut along $C$ into two pieces, one of them is the half-handle $\Hr$.
Note that $(C,B_0)$ is a $n$--dimensional handle of index $k$. See Figure~\ref{fig:righthandle} for an example of a right half-handle.
\end{definition}

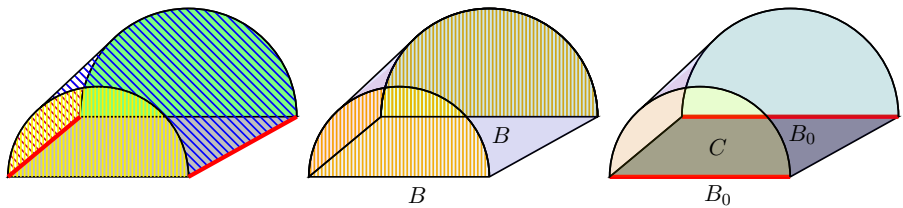
\begin{figure}
\begin{pspicture}(-7,-1.2)(7,1.7)
\rput(-4,0){\psscalebox{0.8}{%
\psarc[fillcolor=green,opacity=0.5,fillstyle=solid](0.5,0){1.8}{0}{180}%
\pspolygon[fillcolor=gray,fillstyle=solid,opacity=0.5](-2.5,-1)(0.5,-1)(2.3,0)(-1.3,0)%
\pscustom[fillcolor=purple,fillstyle=vlines,opacity=0.3, hatchcolor=purple, hatchsep=2pt]{%
\psarc(-1,-1){1.5}{130}{180}%
\psarcn[liftpen=1](0.5,0){1.8}{180}{150}}%
\pscustom[fillcolor=blue,fillstyle=vlines,opacity=0.3, hatchcolor=blue, hatchsep=2pt]{%
\psarc(-1,-1){1.5}{0}{130}%
\psarcn[liftpen=1](0.5,0){1.8}{130}{0}}%
\psarc[fillstyle=vlines,hatchcolor=yellow, opacity=0.7, hatchsep=1pt,hatchangle=0](-1,-1){1.5}{0}{180}
\psline[linecolor=red,linewidth=2pt](-2.5,-1)(-1.3,0)%
\psline[linecolor=red,linewidth=2pt](0.5,-1)(2.3,0)%
%\rput(1.4,-0.8){$B_0$}
%\rput(-1.4,0.9){$B$}
}}
%%%%
\rput(0,0){\psscalebox{0.8}{%
\psarc[fillstyle=vlines,hatchcolor=orange, opacity=0.9, hatchsep=1pt,hatchangle=0](-1,-1){1.5}{0}{180}
\psarc[fillstyle=vlines,hatchcolor=orange, opacity=0.9, hatchsep=1pt,hatchangle=0](0.5,0){1.8}{0}{180}
\psarc[fillstyle=none,opacity=1](-1,-1){1.5}{0}{180}
\psarc[fillstyle=none,opacity=1](0.5,0){1.8}{0}{180}
\psarc[fillcolor=green,opacity=0.5,fillstyle=solid,opacity=0.1](0.5,0){1.8}{0}{180}%
\pspolygon[fillcolor=gray,fillstyle=solid,opacity=0.1](-2.5,-1)(0.5,-1)(2.3,0)(-1.3,0)%
\pscustom[fillcolor=purple,fillstyle=solid,opacity=0.1, hatchcolor=blue, hatchsep=2pt]{%
\psarc(-1,-1){1.5}{130}{180}%
\psarcn[liftpen=1](0.5,0){1.8}{180}{150}}%
\pscustom[fillcolor=blue,fillstyle=solid,opacity=0.1, hatchcolor=blue, hatchsep=2pt]{%
\psarc(-1,-1){1.5}{0}{130}%
\psarcn[liftpen=1](0.5,0){1.8}{130}{0}}%
%\psline[linecolor=red,linewidth=2pt](-2.5,-1)(-1.3,0)%
%\psline[linecolor=red,linewidth=2pt](0.5,-1)(2.3,0)%
%\rput(1.4,-0.8){$B_0$}
\rput(-0.7,-1.3){$B$}
\rput(0.7,-0.3){$B$}}}
%%%
\rput(4,0){\psscalebox{0.8}{%
\pspolygon[fillcolor=gray,fillstyle=solid,opacity=0.8](-2.5,-1)(0.5,-1)(2.3,0)(-1.3,0)%
\psline[linecolor=red,linewidth=2pt](-2.5,-1)(0.5,-1)%
\psline[linecolor=red,linewidth=2pt](-1.3,0)(2.3,0)%
\psarc[fillcolor=green,opacity=0.5,fillstyle=solid,opacity=0.1](0.5,0){1.8}{0}{180}%
\pscustom[fillcolor=purple,fillstyle=solid,opacity=0.1, hatchcolor=blue, hatchsep=2pt]{%
\psarc(-1,-1){1.5}{130}{180}%
\psarcn[liftpen=1](0.5,0){1.8}{180}{150}}%
\pscustom[fillcolor=blue,fillstyle=solid,opacity=0.1, hatchcolor=blue, hatchsep=2pt]{%
\psarc(-1,-1){1.5}{0}{130}%
\psarcn[liftpen=1](0.5,0){1.8}{130}{0}}%
\psarc[fillstyle=solid,fillcolor=yellow, opacity=0.1, hatchsep=1pt,hatchangle=0](-1,-1){1.5}{0}{180}
\rput(0.7,-0.3){$B_0$}
\rput(-0.7,-1.3){$B_0$}
\rput(-0.7,-0.5){$C$}
%\rput(-1.4,0.9){$B$}
}}
\end{pspicture}

\caption{A right half-handle of index $1$. The picture on the left is the handle, the two other pictures explain the notation. The two half-circles form $B$, $C$ is the bottom rectangle.}\label{fig:righthandle}
\end{figure}

Symmetrically, we define the  left half-handles by cutting the handle $H$ along the left--component  disk $D^k$; see Figure~\ref{fig:lhh}.
\begin{definition}\label{def:lhh} Fix $k$ with $1\leqslant k \leqslant n+1$.
\emph{An $(n+1)$--dimensional left half-handle of index $k$} is
the $(n+1)$--dimensional disk $\Hl:=D^k_+\times D^{n+1-k}$ with
boundary subdivided into three pieces $\partial \Hl=B\cup C\cup
N$, where
$$
B:=S^{k-1}_+\times D^{n+1-k}, \ \ \
C:=D^{k-1}_0\times D^{n+1-k}, \ \ \
N:=D^{k}_+\times \partial D^{n+1-k}.
$$
Furthermore, we specify $B_0:=C\cap B=S^{k-2}_0\times D^{n+1-k}$ and $N_0:=N\cap C=D^{k-1}_0\times \partial D^{n+1-k}$.
\end{definition}
\begin{remark}The right half-handle and left-half handle are abstractly diffeomorphic to an $n+1$ dimensional disk. The right half-handle and left-half handle are each abstractly diffeomorphic to an $n+1$ dimensional disk. The difference is that the boundary is
split into several components and this splitting is different for right half-handles and left half-handles.
\end{remark}

\begin{figure}
\begin{pspicture}(-7,-1.2)(7,1.7)
\rput(-4,0){\psscalebox{0.8}{%
\psarc[fillcolor=green,opacity=0.5,fillstyle=solid](0.5,0){1.8}{0}{180}%
\pspolygon[fillcolor=gray,fillstyle=solid,opacity=0.5](-2.5,-1)(0.5,-1)(2.3,0)(-1.3,0)%
\pscustom[fillcolor=purple,fillstyle=vlines,opacity=0.3, hatchcolor=purple, hatchsep=2pt]{%
\psarc(-1,-1){1.5}{130}{180}%
\psarcn[liftpen=1](0.5,0){1.8}{180}{150}}%
\pscustom[fillcolor=blue,fillstyle=vlines,opacity=0.3, hatchcolor=blue, hatchsep=2pt]{%
\psarc(-1,-1){1.5}{0}{130}%
\psarcn[liftpen=1](0.5,0){1.8}{130}{0}}%
\psarc[fillstyle=vlines,hatchcolor=yellow, opacity=0.7, hatchsep=1pt,hatchangle=0](-1,-1){1.5}{0}{180}
\psline[linecolor=red,linewidth=2pt](-2.5,-1)(-1.3,0)%
\psline[linecolor=red,linewidth=2pt](0.5,-1)(2.3,0)%
%\rput(1.4,-0.8){$B_0$}
%\rput(-1.4,0.9){$B$}
}}
%%%%
\rput(0,0){\psscalebox{0.8}{%
\psarc[fillcolor=green,opacity=0.5,fillstyle=solid,opacity=0.1](0.5,0){1.8}{0}{180}%
\pspolygon[fillcolor=gray,fillstyle=solid,opacity=0.1](-2.5,-1)(0.5,-1)(2.3,0)(-1.3,0)%
\pscustom[fillcolor=purple,fillstyle=solid,opacity=0.8, hatchcolor=blue, hatchsep=2pt]{%
\psarc(-1,-1){1.5}{130}{180}%
\psarcn[liftpen=1](0.5,0){1.8}{180}{150}}%
\pscustom[fillcolor=blue,fillstyle=solid,opacity=0.8, hatchcolor=blue, hatchsep=2pt]{%
\psarc(-1,-1){1.5}{0}{130}%
\psarcn[liftpen=1](0.5,0){1.8}{130}{0}}%
\psarc[fillstyle=solid,fillcolor=yellow, opacity=0.1, hatchsep=1pt,hatchangle=0](-1,-1){1.5}{0}{180}
%\psline[linecolor=red,linewidth=2pt](-2.5,-1)(-1.3,0)%
%\psline[linecolor=red,linewidth=2pt](0.5,-1)(2.3,0)%
%\rput(1.4,-0.8){$B_0$}
\rput(-1.4,0.9){$B$}}}
%%%
\rput(4,0){\psscalebox{0.8}{%
\psarc[fillcolor=green,opacity=0.5,fillstyle=solid,opacity=0.1](0.5,0){1.8}{0}{180}%
\pspolygon[fillcolor=gray,fillstyle=solid,opacity=0.8](-2.5,-1)(0.5,-1)(2.3,0)(-1.3,0)%
\pscustom[fillcolor=purple,fillstyle=solid,opacity=0.1, hatchcolor=blue, hatchsep=2pt]{%
\psarc(-1,-1){1.5}{130}{180}%
\psarcn[liftpen=1](0.5,0){1.8}{180}{150}}%
\pscustom[fillcolor=blue,fillstyle=solid,opacity=0.1, hatchcolor=blue, hatchsep=2pt]{%
\psarc(-1,-1){1.5}{0}{130}%
\psarcn[liftpen=1](0.5,0){1.8}{130}{0}}%
\psarc[fillstyle=solid,fillcolor=yellow, opacity=0.1, hatchsep=1pt,hatchangle=0](-1,-1){1.5}{0}{180}
\psline[linecolor=red,linewidth=2pt](-2.5,-1)(-1.3,0)%
\psline[linecolor=red,linewidth=2pt](0.5,-1)(2.3,0)%
\rput(1.5,-0.8){$B_0$}
\rput(-1.6,-0.7){$B_0$}
\rput(-0.6,-0.5){$C$}
%\rput(-1.4,0.9){$B$}
}}
\end{pspicture}
\caption{A left half-handle of dimension $3$ and index $k=2$.
The two lines are $B_0$, the bottom rectangle is $C$. $B$ is
the surface between the two half circles on the picture.
}\label{fig:lhh}
\end{figure}
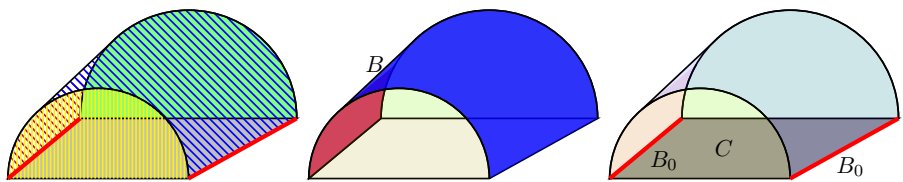
A \emph{half-handle} will from now on refer to either a right half-handle or left half-handle.
We pass to half-handle attachments. We will attach a half-handle along $B$. The definitions of the right half-handle
attachment and the left half-handle attachment are formally very similar, but there are significant differences in the properties of
the two operations.

\begin{definition}\label{def:rha}
Let $(\Omega,Y;\Sigma_0,M_0,\Sigma_1,M_1)$ be an $(n+1)$-dimensional
relative cobordism. Given an embedding $\Phi\colon (B,B_0)\hookrightarrow (\Sigma_1, M_1)$
define the relative cobordism $(\Omega',Y';\Sigma_0,M_0,\Sigma'_1,M'_1)$
obtained from $(\Omega,Y;\Sigma_0,M_0,\Sigma_1,M_1)$ by {\it attaching a (right or left)
 half-handle of index $k$} by
\begin{align*}
\Omega'&=\Omega\cup _BH,& Y'&=Y \cup_{B_0} C,\\
\Sigma'_1&=(\Sigma_1 \setminus B) \cup N,& M'_1&=(M_1 \setminus B_0) \cup N_0.
\end{align*}
See Figure~\ref{fig:rha} and Figure~\ref{fig:lha} for  right,
respectively left half-handle attachments.
\end{definition}

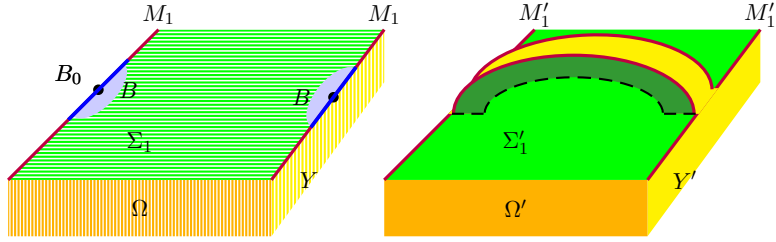
\begin{figure}
\begin{pspicture}(-5,-2)(5,1.5)
\begin{psclip}{%
\pspolygon[linestyle=none,fillstyle=hlines,hatchcolor=green, hatchangle=0, hatchsep=0.6pt](-5,-1)(-3,1)(0,1)(-1.5,-1)%
\rput(-3.25,-0.5){\psscalebox{0.8}{$\Sigma_1$}}%
}
\rput{45}(-3.8,0.2){\psellipse[fillstyle=solid,fillcolor=lightblue,linestyle=none](0.0,0.0)(0.5,0.25)}\rput(-3.4,0.2){\psscalebox{0.8}{$B$}}
\rput{50}(-0.675,0.1){\psellipse[fillstyle=solid,fillcolor=lightblue,linestyle=none](0.0,0.0)(0.5,0.25)}\rput(-1.1,0.1){\psscalebox{0.8}{$B$}}
\end{psclip}
\pspolygon[linestyle=none,fillstyle=vlines,hatchcolor=yellow, hatchangle=0, hatchsep=0.6pt](-1.5,-1.75)(-1.5,-1)(0,1)(0,0.3)\rput(-1,-1){\psscalebox{0.8}{$Y$}}
\pspolygon[linestyle=none,fillstyle=vlines,hatchcolor=orange, hatchangle=0, hatchsep=0.5pt](-1.5,-1.75)(-1.5,-1)(-5,-1)(-5,-1.75)%
\rput(-3.25,-1.375){\psscalebox{0.8}{$\Omega$}}
\psline[linecolor=purple,linewidth=1.2pt](-5,-1)(-3,1)\rput(-3,1.2){\psscalebox{0.8}{$M_1$}}
\psline[linecolor=purple,linewidth=1.2pt](-1.5,-1)(0,1)\rput(0,1.2){\psscalebox{0.8}{$M_1$}}
\pscircle[fillcolor=black,fillstyle=solid](-3.8,0.2){0.07}
\pscircle[fillcolor=black,fillstyle=solid](-0.675,0.1){0.07}
\psline[linewidth=1.4pt,linecolor=blue](-4.2,-0.2)(-3.4,0.6)\rput(-4.2,0.4){\psscalebox{0.8}{$B_0$}}
\psline[linewidth=1.4pt,linecolor=blue](-0.975,-0.3)(-0.375,0.5)\rput(-4.2,0.4){\psscalebox{0.8}{$B_0$}}
\rput(5,0){%
\pspolygon[linestyle=none,fillstyle=solid, fillcolor=yellow](-1.5,-1.75)(-1.5,-1)(0,1)(0,0.3)\rput(-1,-1){\psscalebox{0.8}{$Y'$}}
\pspolygon[linestyle=none,fillstyle=solid, fillcolor=orange](-1.5,-1.75)(-1.5,-1)(-5,-1)(-5,-1.75)%
\rput(-3.25,-1.375){\psscalebox{0.8}{$\Omega'$}}
\pspolygon[linestyle=none,fillstyle=solid, fillcolor=green](-5,-1)(-3,1)(0,1)(-1.5,-1)\rput(-3.25,-0.5){\psscalebox{0.8}{$\Sigma_1'$}}%
\psline[linecolor=purple,linewidth=1.2pt](-5,-1)(-3,1)\rput(-3,1.2){\psscalebox{0.8}{$M_1'$}}
\psline[linecolor=purple,linewidth=1.2pt](-1.5,-1)(0,1)\rput(0,1.2){\psscalebox{0.8}{$M_1'$}}
\pspolygon[linestyle=solid,linecolor=yellow,fillstyle=solid, fillcolor=yellow, linewidth=1.2pt](-4.1,-0.1)(-0.825,-0.1)(-0.6,0.2)(-3.85,0.15)
\rput{2}(-2.26,0.15){\psellipticarc[fillcolor=yellow, fillstyle=solid ,linecolor=purple,linewidth=1.2pt](0,0)(1.64,0.8){0}{180}}
\rput(-2.48,-0.12){\psellipticarc[fillcolor=darkgreen, fillstyle=solid, linecolor=purple, linewidth=1.2pt](0,0)(1.62,0.8){0}{180}%
\psellipticarc[linestyle=dashed,linecolor=black, fillstyle=solid,fillcolor=green](0,0)(1.2,0.5){0}{180}%
\psline[linecolor=green](-1.2,0)(1.2,0)%
\psline[linecolor=black, linestyle=dashed](1.62,0)(1.2,0)%
\psline[linecolor=black, linestyle=dashed](-1.62,0)(-1.2,0)%
}
%\pscircle[fillcolor=black,fillstyle=solid](-3.8,0.2){0.07}
%\pscircle[fillcolor=black,fillstyle=solid](-0.675,0.1){0.07}
%\psline[linewidth=1.4pt,linecolor=blue](-4.2,-0.2)(-3.4,0.6)\rput(-4.2,0.4){\psscalebox{0.8}{$A$}}
%\psline[linewidth=1.4pt,linecolor=blue](-0.975,-0.3)(-0.375,0.5)\rput(-4.2,0.4){\psscalebox{0.8}{$A$}}
}
\end{pspicture}
\caption{Right half-handle attachment. Here $k=1$, $n=2$. On the right, the two black points represent a sphere $S^0$ with a
neighbourhood $B_0$ in $M_1$ and $B$ in $\Sigma_1$. In the picture
on the right the dark green coloured part of the handle belongs to $\Sigma_1$, the dashed lines belong to $\Sigma_1$ and are
drawn only to  make the picture look more `three-dimensional'.}\label{fig:rha}
\end{figure}

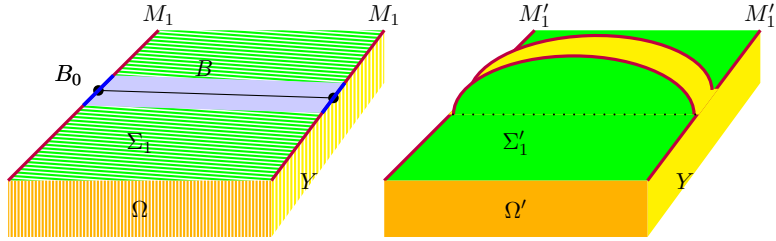
\begin{figure}
\begin{pspicture}(-5,-2)(5,1.5)
\pspolygon[linestyle=none,fillstyle=hlines,hatchcolor=green, hatchangle=-2.5, hatchsep=0.6pt](-5,-1)(-3,1)(0,1)(-1.5,-1)%
\rput(-3.25,-0.5){\psscalebox{0.8}{$\Sigma_1$}}%
\pspolygon[linestyle=none,fillstyle=vlines,hatchcolor=orange, hatchangle=0, hatchsep=0.5pt](-1.5,-1.75)(-1.5,-1)(-5,-1)(-5,-1.75)%
\rput(-3.25,-1.375){\psscalebox{0.8}{$\Omega$}}
\pspolygon[fillstyle=solid,fillcolor=lightblue,linestyle=none](-4.0,0)(-0.825,-0.1)(-0.525,0.3)(-3.6,0.4)\rput(-2.4,0.5){\psscalebox{0.8}{$B$}}
\pspolygon[linestyle=none,fillstyle=vlines,hatchcolor=yellow, hatchangle=0, hatchsep=0.6pt](-1.5,-1.75)(-1.5,-1)(0,1)(0,0.3)\rput(-1,-1){\psscalebox{0.8}{$Y$}}
\psline[linecolor=purple,linewidth=1.2pt](-5,-1)(-3,1)\rput(-3,1.2){\psscalebox{0.8}{$M_1$}}
\psline[linecolor=purple,linewidth=1.2pt](-1.5,-1)(0,1)\rput(0,1.2){\psscalebox{0.8}{$M_1$}}
\pscircle[fillcolor=black,fillstyle=solid](-3.8,0.2){0.07}
\pscircle[fillcolor=black,fillstyle=solid](-0.675,0.1){0.07}
\psline[linewidth=0.4pt](-3.8,0.2)(-0.675,0.1)
\psline[linewidth=1.4pt,linecolor=blue](-4.0,0)(-3.6,0.4)\rput(-4.2,0.4){\psscalebox{0.8}{$B_0$}}
\psline[linewidth=1.4pt,linecolor=blue](-0.825,-0.1)(-0.525,0.3)\rput(-4.2,0.4){\psscalebox{0.8}{$B_0$}}
\rput(5,0){%
\pspolygon[linestyle=none,fillstyle=solid, fillcolor=yellow](-1.5,-1.75)(-1.5,-1)(0,1)(0,0.3)\rput(-1,-1){\psscalebox{0.8}{$Y$}}
\pspolygon[linestyle=none,fillstyle=solid, fillcolor=green](-5,-1)(-3,1)(0,1)(-1.5,-1)\rput(-3.25,-0.5){\psscalebox{0.8}{$\Sigma_1'$}}%
\psline[linecolor=purple,linewidth=1.2pt](-5,-1)(-3,1)\rput(-3,1.2){\psscalebox{0.8}{$M_1'$}}
\psline[linecolor=purple,linewidth=1.2pt](-1.5,-1)(0,1)\rput(0,1.2){\psscalebox{0.8}{$M_1'$}}
\pspolygon[linestyle=none,fillstyle=solid, fillcolor=orange,hatchcolor=orange](-1.5,-1.75)(-1.5,-1)(-5,-1)(-5,-1.75)%
\rput(-3.25,-1.375){\psscalebox{0.8}{$\Omega'$}}
\pspolygon[linestyle=solid,linecolor=yellow,fillstyle=solid, fillcolor=yellow, linewidth=1.2pt](-4.1,-0.1)(-0.825,-0.1)(-0.6,0.2)(-3.85,0.15)
\rput{2}(-2.26,0.15){\psellipticarc[fillcolor=yellow, fillstyle=solid ,linecolor=purple,linewidth=1.2pt](0,0)(1.64,0.8){0}{180}}
\rput(-2.48,-0.12){\psellipticarc[fillcolor=green, fillstyle=solid, linecolor=purple, linewidth=1.2pt](0,0)(1.62,0.8){0}{180}\psline[linestyle=dotted](-1.62,0)(1.62,0)}
}
\end{pspicture}
\caption{Left half-handle attachment with $k=2$ and $n=2$. This time the sphere on the left (denoted by two points) bounds a disk in $\Sigma_1$.}\label{fig:lha}
\end{figure}

We point out that in the case of the right half-handle attachment,
 any embedding of $B_0$ into $M_1$ determines (up to an isotopy) an embedding of pairs
 $(B,B_0)\hookrightarrow (\Sigma_1, M_1) $.  Indeed, as $(B,B_0)=\partial D^k\times(D_+^{n+1-k},D_0^{n-k})$,
a map $\phi\colon B_0\hookrightarrow M_1$ extends to a map
$\Phi\colon B\hookrightarrow \S_1$ in a collar  neighbourhood of $M_1$ in $\S_1$.
(This is not the case in the left half-handle attachment.)

In particular, {\it in the case of right attachments, we specify only the embedding $B_0\hookrightarrow M_1$}.

\begin{example}
(a) The  right half-handle attachment of index $0$  is the
disconnected sum $\O\sqcup D^{n+1}_+$ with boundary $\p
D^{n+1}_+=S^{n}_+\cup D^n_0$. We think of the first disk $S^{n}_+$  as
a part of $\S'_1$, while the second disk as a part of $Y'$, and
$M'_1=M_1\cup S^{n-1}_0$.

(b) We exemplify the left half-handle attachment for $k=1$. In
this case $B_0$ is empty. If we are given an embedding of $B\cong
\{1\}\times D^{n}$ into $\S_1\setminus M_1$, we glue $[0,1]\times
D^n$ to $\O$ along $B$. Then we set
 $Y'=Y\sqcup \{0\}\times D^{n}$, $\S'_1=(\S_1\setminus B)\cup  [0,1]\times \p B$ and $M'=M\sqcup \{0\}\times \p B$.

\end{example}

\begin{example}\label{ex:clay}
There is another way of looking at left half-handle attachments. Suppose we are given a model of $\Omega$ (in Figure~\ref{fig:clay}) made of clay.
The height function is the Morse function $F$. On the top of $\Omega$, that is on $\Sigma_1$ we specify an arc $\gamma$ with boundary in $M_1$ (the arc is a $1$ dimensional
disk, that is, in our situation $k=1+1=2$).  We press down slightly a tubular neighbourhood of the arc as on the right side of Figure~\ref{fig:clay}. The resulting manifold is
a result of a left half--handle attachment of index $2$. Notice that for index $1$ left half--handle attachment 
we should have specified a disk inside $\Sigma_1$ (with boundary disjoint from $M_1$).
\end{example}

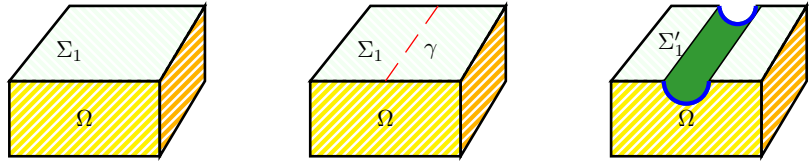
\begin{figure}
\begin{pspicture}(-5,0)(5,3)
\rput(-4,0){\psscalebox{2}{%
\pspolygon[fillstyle=hlines,hatchcolor=yellow,hatchsep=0.5pt,linewidth=0.5pt](0,0.5)(0,1)(1,1)(1,0.5)%
\pspolygon[fillstyle=hlines,hatchcolor=orange,hatchsep=0.5pt,linewidth=0.5pt](1,0.5)(1,1)(1.3,1.5)(1.3,1.0)%
\pspolygon[hatchcolor=lightgreen,hatchsep=0.5pt,fillstyle=vlines,linewidth=0.5pt](0,1)(1,1)(1.3,1.5)(0.4,1.5)%
\rput(0.5,0.75){\psscalebox{0.4}{$\Omega$}}%
\rput(0.4,1.2){\psscalebox{0.4}{$\Sigma_1$}}%
}}%
\rput(0,0){\psscalebox{2}{%
\pspolygon[fillstyle=hlines,hatchcolor=yellow,hatchsep=0.5pt,linewidth=0.5pt](0,0.5)(0,1)(1,1)(1,0.5)%
\pspolygon[fillstyle=hlines,hatchcolor=orange,hatchsep=0.5pt,linewidth=0.5pt](1,0.5)(1,1)(1.3,1.5)(1.3,1.0)%
\pspolygon[hatchcolor=lightgreen,hatchsep=0.5pt,fillstyle=vlines,linewidth=0.5pt](0,1)(1,1)(1.3,1.5)(0.4,1.5)%
\rput(0.5,0.75){\psscalebox{0.4}{$\Omega$}}%
\rput(0.4,1.2){\psscalebox{0.4}{$\Sigma_1$}}%
\psline[linecolor=red,linestyle=dashed,linewidth=0.25pt](0.5,1)(0.85,1.5)%
\rput(0.80,1.2){\psscalebox{0.4}{$\gamma$}}%
}}%
\rput(4,0){\psscalebox{2}{%
\pspolygon[fillstyle=hlines,hatchcolor=yellow,hatchsep=0.5pt,linewidth=0.5pt](0,0.5)(0,1)(1,1)(1,0.5)%
\pspolygon[fillstyle=hlines,hatchcolor=orange,hatchsep=0.5pt,linewidth=0.5pt](1,0.5)(1,1)(1.3,1.5)(1.3,1.0)%
\pspolygon[hatchcolor=lightgreen,hatchsep=0.5pt,fillstyle=vlines,linewidth=0.5pt](0,1)(1,1)(1.3,1.5)(0.4,1.5)%
\rput(0.5,0.75){\psscalebox{0.4}{$\Omega$}}%
\rput(0.4,1.25){\psscalebox{0.4}{$\Sigma_1'$}}%
\psarc[fillcolor=darkgreen,fillstyle=solid,linecolor=darkgreen](0.5,1){0.15}{180}{360}%
\pspolygon[fillstyle=solid,fillcolor=darkgreen,linestyle=none](0.35,1)(0.65,1)(1,1.5)(0.72,1.5)%
\psarc[fillcolor=white,fillstyle=solid,linecolor=darkgreen](0.841,1.5){0.121}{180}{360}%
\psline[linecolor=darkgreen](0.35,1)(0.65,1)%
\psline[linecolor=white](0.722,1.5)(0.960,1.5)%
\psline[linewidth=0.3pt](0.35,1)(0.72,1.5)%
\psline[linewidth=0.3pt](0.65,1)(1,1.5)%
\psarc[linecolor=blue](0.841,1.5){0.121}{180}{360}%
\psarc[linecolor=blue](0.5,1){0.15}{180}{360}%
}}%
\end{pspicture}
\caption{A `clay' variant of a left half-handle attachment as explained in Example~\ref{ex:clay}. There is only a little difference between this picture
and Figure~\ref{fig:lha}. Here the handle is `pushed down inside $\Omega$', in the formal definition it is glued on top of $\Sigma_1$.}\label{fig:clay}
\end{figure}
\begin{remark}
In the next subsection we shall see that crossing a boundary
stable critical points corresponds to left half-handle attachment,
while a boundary unstable critical point corresponds to a right
half-handle attachment. Theorem~\ref{thm:handsplit} can be
interpreted informally, as splitting a handle into a right
half-handle and left half-handle. This also motivates  the name
`half-handle'.
\end{remark}

\subsection{Elementary properties of half-handle
attachments}\label{ss:elem}
The following results are trivial consequences of the definitions.

\begin{lemma}\label{lem:rightislikenormal}
Let $\O'$ be the result of a right half-handle attachment to $\O$ along $(B,B_0)\hookrightarrow(\S_1,M_1)$.
Let $B'$ be $B$ pushed slightly off $M_1$ into the interior of $\S_1$.
Let $\tilde{\O}$ be the result of attaching a (standard) handle of index $k$ to $\O$
along $B'$. Then  $\O'$ and $\tilde{\O}$ are diffeomorphic.
\end{lemma}
\begin{proof}
When we forget about $C$ and $B_0$, the
pair $(\Hr,B)$ is a standard $(n+1)$-dimensional handle of index
$k$.
\end{proof}
For instance, the effect of a right half-handle attachment on $\O$ is the same as the effect of a standard handle attachment
of the same index.

The situation is completely different in the case of left half-handle attachments.
\begin{lemma}\label{lem:leftistrivial}
If $\O'$ is the result of a left half-handle attachment,
then $\O'$ is diffeomorphic to $\O$.
\end{lemma}
\begin{proof}
By definition the pair $(\Hl,B)$ is diffeomorphic to the pair $(D^n\times[0,1],D^n\times\{0\})$.
Attaching $\Hl$ along $B$ to $\O$
does not change the diffeomorphism type of $\O$.
\end{proof}

The effect on $Y$ of a right/left half-handle attachments are almost the same, the only difference is
the index shift by $1$.

\begin{lemma}\label{lem:whathappensonboundary}
If $(\O',Y';\S_0,M_0,\S_1',M_1')$ is the result of a left  (respectively, right) half-handle attachment to
$(\Omega,Y;\Sigma_0,M_0,\Sigma_1,M_1)$ along $(B,B_0)\hookrightarrow (\S_1,M_1)$,
then $Y'$ is the result of a classical handle attachment of index $k-1$ (respectively $k$) along $B_0$.
\end{lemma}
\begin{proof}
This follows immediately from Definition~\ref{def:rha}.
\end{proof}

The effects of half handle attachment on $\Sigma$ are also easily described.
The next lemma is a  direct consequence of the definitions; its
proof is omitted. We refer to Figures~~\ref{fig:rha} and
\ref{fig:lha}.

\begin{lemma}\label{lem:leftonsigma}
(a) \ If $(\O',Y';\S_0,M_0,\S_1',M_1')$ is the result of index $k$
right half-handle attachment to
$(\Omega,Y;\Sigma_0,M_0,\Sigma_1,M_1)$ along $B_0\hookrightarrow
M_1$, then $\S_1'$ is diffeomorphic to $\S_1\cup_{B_0} N$,
where $N$ is an $n$-dimensional disk $D^k\times D^{n-k}$ and $B_0=S^{k-1}\times D^{n-k}$.

(b) \ If $(\O',Y';\S_0,M_0,\S_1',M_1')$ is the result of left
half-handle attachment to $(\Omega,Y;\Sigma_0,M_0,\Sigma_1,M_1)$
along $(B,B_0)\hookrightarrow (\S_1,M_1)$, then $\S_1'$ is diffeomorphic to $\S_1\setminus B$.
\end{lemma}

\begin{example}\label{ex:n=3}
Suppose $n=3$, so $\Sigma_1$ and $\Sigma_1'$ are three dimensional manifolds with boundary. The effects on $\Sigma_1$ of left and right half-handle attachments
to $\O$ are the following (the number $0,1,2,3,4$ is the index of a handle, `l' and `r' stays for `left' and `right'):
\begin{itemize}
\item[(0r)] $\Sigma_1'$ is a disjoint union of $\Sigma_1$ and a $3$-ball;
\item[(1l)] A $3$--ball is removed from the interior of $\Sigma_1$;
\item[(1r)] A $1$--handle (that is a thickened arc) is added to $\Sigma_1$. The attaching region is formed by thickening two point on 
$\partial\Sigma_1$;
\item[(2l)] An arc $\gamma$ is chosen inside $\Sigma_1$ such that $\partial\gamma\subset\partial\Sigma_1$. The manifold $\Sigma_1'$ is then $\Sigma_1$
with a tubular neighbourhood of $\gamma$ removed;
\item[(2r)] A $2$--handle (a thickened two-disk) is added to $\Sigma_1$. The attaching region is formed by thickening a circle belonging to $\partial
\Sigma_1$;
\item[(3l)] A disk $D$ is specified inside $\Sigma_1$ such that $\partial D\subset \partial\Sigma_1$. Then a tubular neighbourhood of $D$ is removed;
\item[(3r)] A $3$--handle (a ball) is added to $\Sigma_1$. Notice that adding a $3$-ball destroys one component of the boundary;
\item[(4l)] A connected component of $\Sigma_1$ that is a ball, is removed from $\Sigma_1$. This is the opposite of the (0r) move.
\end{itemize}
\end{example}

Lemma~\ref{lem:leftonsigma} and Example~\ref{ex:n=3} emphasize  that right half-handle attachments and
left half-handle attachments are somehow dual operations on $\S$.
This can be seen also at the Morse function level: changing a
Morse function $F$ to $-F$ changes all right half-handles to
left-half handles and conversely, see Section~\ref{ss:bcpahh} and
\ref{s:larpc} below. But the above lemma shows another aspect as
well: a right half handle attachment consists of gluing a disk,
a left half-handle attachment consists of removing a disk. Indeed,
in the case of right attachment,  $(\S_1',M_1')~=~(\S_1\cup
D^{k}\times D^{n-k},\p \S_1')$ associated with an embedding
$\Phi\colon \p D^{k}\times D^{n-k}\to M_1$. On the other hand,
 by definition,  for an embedding $\Phi':(D^{k-1}\times D^{n+1-k},\p D^{k-1}\times D^{n+1-k})\to (\S_1,M_1)$
the pair
$$(\S'_1,M'_1)~=~({\rm closure \ of}(\S_1\setminus D^{k-1}\times D^{n+1-k}),\p\S'_1)$$
is obtained from $(\S_1,M_1)$ by a \emph{handle detachment} of
index $k-1$. We formulate this observation as a rephrasing of Lemma~\ref{lem:leftonsigma}.

\begin{corollary}\label{ex:cobordismistrace}
The effect on $(\S_1,M_1)$ of a right half--handle attachment of
index $k$ is  a handle attachment  of
index $k$ to $(\S_1,M_1)$. Likewise, the effect on
$(\S_1,M_1)$ of a left half-handle attachment of index $k$
is a handle detachment  of index $k-1$. In particular, $M_1'$ is obtained from
$M_1$ as the result of a $k$ surgery in the first case, and
$(k-1)$ surgery in the second.
\end{corollary}

The duality can also be seen as follows:   we can cancel any
handle attachment by a suitably defined handle detachment, and
conversely.

The following definition introduces a terminology which is rather self-explanatory. We include it for completeness of the exposition.

\begin{definition}\label{def:hha}
We shall say that a cobordism $(\O',Y')$ between $(\S,M)$ and $(\S',M')$
is a \emph{right (respectively left) half-handle attachment} of index $k$, if $(\O',Y',\S',M')$ is a result of
right (respectively  left)
half-handle attachments of index $k$ (in the sense of Definition~\ref{def:rha})
to $(\S\times[0,1],M\times[0,1],\S\times\{0\},M\times\{0\},\S\times\{1\},M\times\{1\})$.
\end{definition}

We conclude this section by studying homological properties of handle attachment. These
properties will be used in \cite{BNR2}. The proofs are standard and are left
to the reader.

Let $(H_+,C,B,N)$ be a half-handle of index $k$.

\begin{lemma}\label{lem:natiso}
If $(\Hr,C,B,N)$ is a right half-handle, then the pair $(C,B_0)$ is a strong deformation
 retract of $(H_+^r,B)$, while $(D^k,\p D^k)$ is a strong deformation
 retract of $(C,B_0)$. In particular,  $H_j(H_+^r,B)\cong H_j(C,B_0)=\Z$ for $j=k$, and it is zero otherwise.
\end{lemma}

The situation is completely different for left half-handles.
\begin{lemma}
If $(\Hl,C,B,N)$ is a left half-handle, then the pair $(\Hl,B)$ retracts
 onto the trivial pair $(\mathrm{point},\mathrm{point})$. In particular, all the relative homologies
$H_*(\Hl,B)$ vanish. On the other hand, $(D_0^{k-1},S_0^{k-2})$ is a strong deformation retract of
$(C,B_0)$, hence $H_j(C,B_0)=\Z$ for $j=k-1$, and it is zero otherwise.
Therefore,  the inclusion $(C,B_0)\hookrightarrow (\Hl,B)$ induces a surjection on homologies.
\end{lemma}

\subsection{Boundary critical points and half-handles}\label{ss:bcpahh}

Consider a Morse function $F$ on a cobordism $(\O,Y)$ and assume that it has a single boundary critical point $z$ of index $k$ with critical value $c$
and no interior critical points.

\begin{theorem}\label{thm:morsehalfhandle}
If $z$ is boundary stable (unstable),
then the cobordism is a left (right) half-handle attachment of index $k$ respectively.
\end{theorem}
\begin{proof} We can assume that $c=F(z)=0$.
Let us chose a neighbourhood $U$ of $z$ in $\Omega$. Shrinking $U$ if necessary, we can assume that
there are Morse coordinates $x_1,\dots,x_{n+1}$ on $U$ (see Lemma~\ref{lem:boundarymorselemma}) and
in these coordinates $U$ is a half-ball of radius $2\rho$ for some positive number $\rho$:
\[U=\{x_1^2+\dots+x_{n+1}^2\leqslant 4\rho^2\}\cap\{x_1\geqslant 0\}.\]
The intersection $Y\cap U$ defined by $\{x_1= 0\}$, and
\[F(x_1,\dots,x_{n+1})=-a^2+b^2,\]
where if $z$ is boundary stable we set
\begin{equation}\label{eq:abstable}
a^2=x_1^2+x_2^2+\dots+x_k^2,\ \ \ b^2=x_{k+1}^2+\dots+x_{n+1}^2 \ \ \ \ \ (k\geqslant 1),
\end{equation}
and if $z$ is boundary unstable
\begin{equation}\label{eq:abunstable}
a^2=x_2^2+\dots+x_{k+1}^2,\ \ \ b^2=x_1^2+x_{k+2}^2+\dots+x_{n+1}^2\ \ \ \ \ (k\geqslant 0).
\end{equation}
We also assume that  $x_1,\dots,x_{n+1}$
is an  Euclidean orthonormal coordinate system.

Next, we consider  $\e>0$ such that $\e\ll\rho$, and
we define the space $\widetilde{H}$ bounded by the following conditions (see Figure~\ref{fig:Hprim})
\[\widetilde{H}:=\{-a^2+b^2\in [-\e^2,\e^2],\ \ a^2b^2\leqslant \rho^4-\e^4, \ \ x_1\geqslant 0\}.\]
Observe that
\[\widetilde{H}\subset U.\]
Let us now define the following parts of the boundary of $\widetilde{H}$
\begin{equation}\label{eq:BPKC}
\begin{split}
\widetilde{B}&=\p \widetilde{H}\cap \{-a^2+b^2=-\e^2\}\subset F^{-1}(-\e^2),\\
\widetilde{P}&=\p \widetilde{H}\cap \{-a^2+b^2=\e^2\}\subset F^{-1}(\e^2),\\
\widetilde{K}&=\p \widetilde{H}\cap \{a^2b^2=\rho^4-\e^4\},\\
\widetilde{C}&=\p \widetilde{H}\cap \{x_1=0\}\subset Y.
\end{split}
\end{equation}
We have
$\widetilde{B}\cup \widetilde{P}\cup \widetilde{K}\cup \widetilde{C}=\p \widetilde{H}$ (in Figure~\ref{fig:Hprim} we do not
see $\widetilde{C}$, because this would require one more dimension).
If $z$ is boundary unstable  and $k=0$ in (\ref{eq:abunstable}) then the term $a^2$ is missing
and $\widetilde{B}=\emptyset$. Otherwise  $\widetilde{B}\not=\emptyset$.
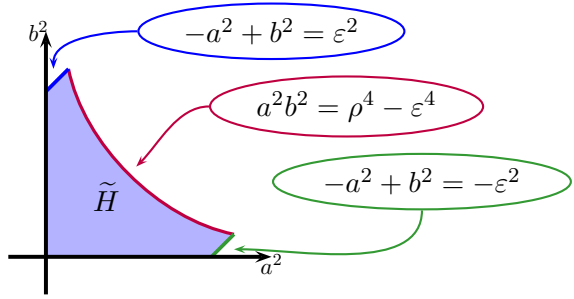
\begin{figure}
\begin{pspicture}(-1,-1)(5,4)

\rput(3,-0.1){\psscalebox{0.8}{$a^2$}}%
\rput(-0.1,3){\psscalebox{0.8}{$b^2$}}
\pscustom[linestyle=none, fillstyle=solid, fillcolor=blue, opacity=0.3]{%
\psbezier(0.3,2.5)(0.5,1.5)(1.5,0.5)(2.5,0.3)%
\psline[liftpen=1](2.2,0.0)(0,0)(0,2.2)%
}
\psbezier[linecolor=purple,linewidth=1.2pt](0.3,2.5)(0.5,1.5)(1.5,0.5)(2.5,0.3)
\psline[linecolor=darkgreen,linewidth=1.4pt](2.5,0.3)(2.2,0)
\psline[linecolor=blue,linewidth=1.4pt](0.3,2.5)(0,2.2)
\rput(4,2){\ovalnode[linecolor=purple]{E1}{$a^2b^2=\rho^4-\e^4$}}
\rput(1.2,1.2){\rnode{E2}{}}
\rput(5,1){\ovalnode[linecolor=darkgreen]{F1}{$-a^2+b^2=-\e^2$}}
\rput(2.5,0.1){\rnode{F2}{}}
\rput(3,3){\ovalnode[linecolor=blue]{G1}{$-a^2+b^2=\e^2$}}
\rput(0.1,2.5){\rnode{G2}{}}
\nccurve[linecolor=purple,angleA=180,angleB=45]{->}{E1}{E2}
\nccurve[linecolor=darkgreen,angleA=270,angleB=0]{->}{F1}{F2}
\nccurve[linecolor=blue,angleA=180,angleB=90]{->}{G1}{G2}
\rput(0.8,0.8){$\widetilde{H}$}
\psline[linewidth=1.5pt]{->}(-0.5,0)(3,0)%
\psline[linewidth=1.5pt]{->}(0,-0.5)(0,3)%
\end{pspicture}
\caption{A schematic presentation of $\widetilde{H}$,
 $\widetilde{B}$, $\widetilde{P}$, $\widetilde{K}$
 from the proof of Theorem~\ref{thm:morsehalfhandle}.
 To each point $(a^2,b^2)$ in $\widetilde{H}$ on the picture, correspond
all those points $(x_1,\dots,x_{n+1})$ for which \eqref{eq:abstable} or \eqref{eq:abunstable} holds and $x_1\geqslant 0$.}\label{fig:Hprim}
\end{figure}

\begin{lemma}\label{lem:tanga3}
The flow of $\nabla F$ is tangent to $\widetilde{K}$.
\end{lemma}
\begin{proof}
Assume the critical point is boundary stable. The differential equation
\[\frac{d\mathbf{x}}{dt}=\nabla F=(-2x_1,\ldots,-2x_k,2x_{k+1},\ldots,2x_{n+1})\]
has solution
\[(x_1,\dots,x_{n+1})\to (e^{-2t}x_1,\dots,e^{-2t}x_{k},e^{2t}x_{k+1},\dots,e^{2t}x_{n+1}).\]
It follows that $a^2\to e^{-4t}a^2$ and $b^2\to e^{4t}b^2$, and  the hypersurface $a^2b^2=\textrm{const}$ is preserved by the flow of $\nabla F$.
\end{proof}

\begin{lemma}\label{lem:retracts} The inclusion of pairs of spaces
$$\left(\
F^{-1}(-\e^2)\cup_{\widetilde{B}} \widetilde{H},Y\cap \left(F^{-1}(-\e^2)\cup_{\widetilde{B}}\widetilde{H}\right)\ \right)\subset (\O,Y)$$
admits a strong deformation retract.
\end{lemma}
\begin{proof}
By Lemma~\ref{lem:novanishnoproblem} we can assume that
$(\O,Y) $ is $( F^{-1}([-\e^2,\e^2]),Y\cap F^{-1}([-\e^2,\e^2])$.

First we assume that $\widetilde{B}$ is not empty, and it is given by the equation \eqref{eq:BPKC} in $U$.
Set $\S_-=\textrm{closure of } (F^{-1}(-\e^2)\setminus \widetilde{B})$ and let $T_-$ be the part of the boundary of $\S_-$ given by
\[T_-=\textrm{closure of } (\, \p\S_-\setminus \p F^{-1}(-\e^2)\,).\]
We have $T_-\subset \widetilde{B}$, see Figure~\ref{fig:retractsone}. Let us choose a collar
of $T_-$ in $\S_-$, that is a subspace $U_-\subset \S_-$ diffeomorphic to $T_-\times[0,1]$,
$T_-$ identified with $T_-\times\{0\}$ and $\p T_-\times[0,1]\subset \p \S_-
\cap \p F^{-1}(-\e^2)$.
Let $T_-'$ be the space identified with $T_-\times\{1\}$ by this diffeomorphism.

Similarly, let  $\S_+=\textrm{closure of }(F^{-1}(\e^2)\setminus \widetilde{P})$,  and
 $T_+=\textrm{closure of }(\p\S_+\setminus\p F^{-1}(\e^2))$.
 We also define $\O_{0}$ as the closure of $\O \setminus \widetilde{H}$.
Clearly $F$ has no critical points in $\O_0$ and $\nabla F$ is everywhere tangent to $\p\O_0\setminus(\S_-\cup\S_+)=(Y\cap\O_{0})\cup \widetilde{K}$ by
Lemma~\ref{lem:tanga3}. In particular, by Lemma~\ref{lem:novanishnoproblem},
 the flow of $\nabla F$ on $\O_0$ yields a diffeomorphism between $\S_-$ and $\S_+$,
mapping $T_-$ to $T_+$. We define $V\subset\O$ as the closure of the set
of points $v$ such that a trajectory going through $v$ hits $U_-$.
Lemma~\ref{lem:novanishnoproblem}
implies that there is a diffeomorphism $V \cong T_-\times[0,1]\times[-\e^2,\e^2]$ such that for $(x,t,s)\in V$ we have $F(x,t,s)=s$.
Finally, we also define $V^*:=\{(x,t,s)\in V\,:\, s\leqslant\e^2(1-2t)\}$.

\begin{figure}
\begin{pspicture}(-5,-2)(5,2)
\pspolygon[fillcolor=yellow,opacity=0.1,fillstyle=solid,linewidth=1.2pt](-5,-1.5)(-1,-1.5)(-1,1.5)(-5,1.5)
\psarc[fillstyle=solid,fillcolor=blue,opacity=0.05,linewidth=0.5pt](-5,0){1.2}{270}{90}
\psarc[fillstyle=solid,opacity=0.5,fillcolor=purple,linewidth=1.1pt](-5,0){0.7}{270}{90}
\rput(-4.8,0){\psscalebox{0.8}{$\widetilde{B}$}}
\rput(-4.5,0.8){\psscalebox{0.8}{$U_-$}}
\rput(-5.2,0.7){\psscalebox{0.8}{$T_-$}}
\rput(-5.2,-1.2){\psscalebox{0.8}{$T_-'$}}
\rput(-3,0){\psscalebox{0.8}{$\S_-$}}
%%%%
\pspolygon[fillcolor=green,opacity=0.1,fillstyle=solid](1,-1.5)(5,-1.5)(5,1.5)(1,1.5)
\pscustom[fillstyle=solid,fillcolor=blue,opacity=0.3]{%
\psbezier(2,1.5)(1.5,1)(1.5,-1)(2,-1.5)%
\psline[liftpen=1](1,-1.5)(1,1.5)}
\psline[linewidth=1.5pt](1,1.5)(5,1.5)\rput(4.8,1.7){\psscalebox{0.8}{$\S_+$}}
\psline[linewidth=1.5pt](1,-1.5)(5,-1.5)\rput(4.8,-1.7){\psscalebox{0.8}{$\S_-$}}
\rput(1.3,0.05){\psscalebox{0.8}{$\widetilde{H}$}}
\rput(0.7,1.2){\psscalebox{0.8}{$Y$}}
\pscircle[fillstyle=solid,fillcolor=black](2,-1.5){0.07}\rput(2,-1.8){\psscalebox{0.8}{$T_-$}}
\pscircle[fillstyle=solid,fillcolor=black](2,1.5){0.07}\rput(2,1.8){\psscalebox{0.8}{$T_+$}}
\pscustom[linestyle=none,fillstyle=solid,fillcolor=purple,opacity=0.1]{%
\psbezier(2.5,1.5)(2,1)(2,-1)(2.5,-1.5)%
\psbezier[liftpen=1](2,-1.5)(1.5,-1)(1.5,1)(2,1.5)%
}
\psbezier[linestyle=dashed](2.5,1.5)(2,1)(2,-1)(2.5,-1.5)%
\pscircle[fillstyle=solid,fillcolor=black](2.5,-1.5){0.07}\rput(2.5,-1.8){\psscalebox{0.8}{$T_-'$}}
\rput(1.9,0){\psscalebox{0.8}{$V$}}
\end{pspicture}
\caption{Notation used in Lemma~\ref{lem:retracts}. Please note that the left picture is drawn on $\S_-$, while the right
one is on $\Omega$.}\label{fig:retractsone}
\end{figure}
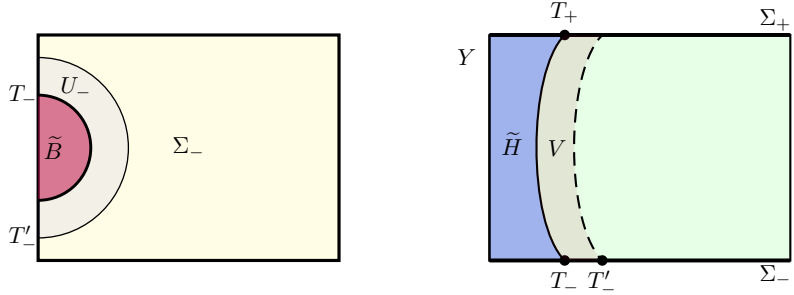

We define the contraction in two steps: vertical and horizontal.
The vertical contraction is defined as follows. For $v\in \widetilde{H}\cup V^*$  we define
$\Pi_V(v)=v$. For a point
$v\in\O_0\setminus V$ we take for $\Pi_V(v)$ the unique point $s\in\S_-$ such
that a trajectory of $\nabla F$ goes from $s$ to $v$.
 Finally if $v=(x,t,s)\in V\setminus V^*$
 we define $\Pi_V(v)=(x,t,\e^2(1-2t))$.

By construction, the image of $\Pi_V$ is $\widetilde{H}\cup V^*\cup F^{-1}(-\e^2)$.
Next,  we define $\Pi_H$. Note, that $\Pi_H$ will be defined only on the image of $\Pi_V$.

It is an identity on $\widetilde{H}\cup
F^{-1}(-\e^2)$, and maps $(x,t,s)\in V^*$ to
$(x,t-(\e^2+s)/(2\e^2),-\e^2)$ if $s\leqslant
\e^2(2t-1)$, and to $(x,0,s-2\e^2t)$ otherwise. Note that the expressions
agree for any $(x,t,s)$ with $s=\e^2(2t-1)$ and these points are sent to $(x,0,-\e^2)$.
Both $\Pi_H$ and $\Pi_V$ are
continuous retractions, by  smoothing corners  we can
modify them into smooth retractions; also they can be extended in a natural way to strong
deformation retracts. By construction, the retracts preserve $Y$ too.  See also Figure~\ref{fig:contractions}.

If $\widetilde{B}$ is empty, then $\widetilde{H}$ is necessarily a unstable (right) half-handle of index $0$,
$F^{-1}([-\e^2,\e^2])$ is a disconnected sum of $\widetilde{H}$ and the manifold $F^{-1}(-\e^2)\times [-\e^2,\e^2]$.
\end{proof}

\begin{figure}
\begin{pspicture}(-6,-2)(6,2)
\pspolygon[fillstyle=solid,fillcolor=green,opacity=0.1](-6,-1)(-3,-1)(-3,-1.5)(-6,-1.5)\rput(-4.5,-1.25){\psscalebox{0.8}{$F^{-1}(-\e^2)$}}
\psline(-6,-1.5)(-6,-1)(-3,-1)(-3,-1.5)\rput(-6.2,-0.6){\psscalebox{0.8}{$Y$}}
\pspolygon[fillstyle=solid,fillcolor=blue,opacity=0.1](-6,-1)(-5,-1)(-5,1)(-6,1)\rput(-5.5,0){\psscalebox{0.8}{$\widetilde{H}$}}
\pspolygon[fillstyle=solid,fillcolor=purple,opacity=0.1](-5,-1)(-4,-1)(-4,1)(-5,1)\rput(-4.5,-0.5){\psscalebox{0.8}{$V$}}
\pspolygon[fillstyle=solid,fillcolor=yellow,opacity=0.1](-4,-1)(-3,-1)(-3,1)(-4,1)
\rput(-4.5,1.2){\psscalebox{0.8}{$F^{-1}(\e^2)$}}
%% now add vectors...
\psline{->}(-3.66,0.8)(-3.66,0.1)\psline{->}(-3.33,0.8)(-3.33,0.1)
\psline{->}(-3.66,-0.1)(-3.66,-0.8)\psline{->}(-3.33,-0.1)(-3.33,-0.8)
\psline{->}(-4.66,0.8)(-4.66,0.1)\psline{->}(-4.33,0.8)(-4.33,0.1)
\psline[linewidth=2pt,arrowsize=8pt]{->}(-2.7,-0.5)(-1.7,-0.5)\rput(-2.2,-0.2){\psscalebox{0.8}{$\Pi_V$}}%
\pspolygon[fillstyle=solid,fillcolor=green,opacity=0.1](-1.5,-1)(1.5,-1)(1.5,-1.5)(-1.5,-1.5)
\psline(-1.5,-1.5)(-1.5,-1)(1.5,-1)(1.5,-1.5)
\pspolygon[fillstyle=solid,fillcolor=blue,opacity=0.1](-1.5,-1)(-0.5,-1)(-0.5,1)(-1.5,1)\rput(-1,0){\psscalebox{0.8}{$\widetilde{H}$}}
\pspolygon[fillstyle=solid,fillcolor=purple,opacity=0.1](-0.5,-1)(0.5,-1)(-0.5,1)\rput(1,0.9){\ovalnode[linecolor=purple]{Z1}{\psscalebox{0.8}{$V^*=\Pi_V(V)$}}}
\rput(-0.1,-0.6){\rnode{Z2}{}}
\psline{->}(-0.25,0.4)(-0.4,0.1)\psline{->}(0,-0.1)(-0.3,-0.7)\psline{->}(0.25,-0.6)(0.1,-0.9)
\nccurve[linecolor=purple,angleA=270,angleB=45]{->}{Z1}{Z2}
%%%%
\psline[linewidth=2pt,arrowsize=8pt]{->}(1.7,-0.5)(2.7,-0.5)\rput(2.2,-0.2){\psscalebox{0.8}{$\Pi_H$}}%
\pspolygon[fillstyle=solid,fillcolor=green,opacity=0.1](3,-1)(6,-1)(6,-1.5)(3,-1.5)
\psline(3,-1.5)(3,-1)(6,-1)(6,-1.5)
\pspolygon[fillstyle=solid,fillcolor=blue,opacity=0.1](3,-1)(4,-1)(4,1)(3,1)\rput(3.5,0){\psscalebox{0.8}{$\widetilde{H}$}}

\end{pspicture}
\caption{Contractions $\Pi_H$ and $\Pi_V$ from the proof of Lemma~\ref{lem:retracts}. The set $V$ is now drawn as a rectangle.}\label{fig:contractions}
\end{figure}
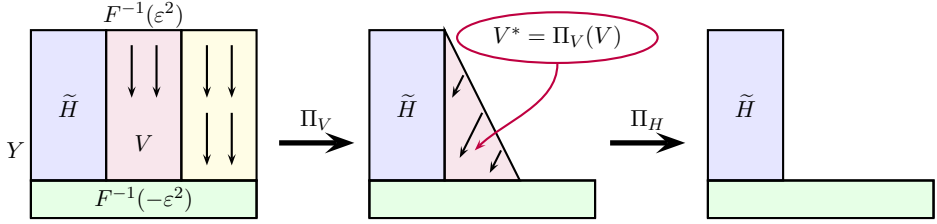

\emph{Continuation of the proof of
Theorem~\ref{thm:morsehalfhandle}.} We want to show that $\widetilde{H}$ is a
half-handle.

By subsection \ref{ss:halfhandles}
we have the following description in local coordinates of the left half-handle (\ref{eq:leftmodel})
and right half-handle (\ref{eq:rightmodel})  with cutting coordinate $x_1$:
\begin{equation}\label{eq:leftmodel}
\Hl=\{x_1^2+\dots+x_{k}^2\leqslant 1\}\cap \{x_{k+1}^2+\dots+x_{n+1}^2\leqslant 1\}\cap\{x_1\geqslant 0\}
\end{equation}
\begin{equation}\label{eq:rightmodel}
\Hr=\{x_2^2+\dots+x_{k+1}^2\leqslant 1\}\cap \{x_1^2+x_{k+2}^2+\dots+x_{n+1}^2\leqslant 1\}\cap\{x_1\geqslant 0\}.
\end{equation}

We consider the subsets $R$ and $S$  of $\R^2$ given by
\[R=\{(u,v)\in\R^2\colon u\geqslant 0,\ v\geqslant 0,\ uv\leqslant \rho^4-\e^4, \ -u+v\in[-\e^2,\e^2]\},\]
\[S\colon \{(u,v)\in\R^2\colon u\in[0,\e],\ v\in[0,\e]\}.\]
(Note
that $R$ can be seen in Figure~\ref{fig:Hprim} if we replace $a^2$ by $u$ and $b^2$ by $v$.)
These subsets are clearly diffeomorphic. We choose a diffeomorphism $\psi$ that maps
the edge of $R$ given by $\{-u+v=-\e^2\}$ to the edge $\{u=\e\}$ of $S$ and the images of coordinate axes are the corresponding coordinate axes.

We use $\psi$ to construct a diffeomorphism $\Psi$ between $\widetilde{H}$ and $\Hr$ (respectively $\Hl$) as follows. First let us
write $\psi(u,v)=(\psi_1(u,v),\psi_2(u,v))$. As $\psi$ maps axes to axes, we have $\psi_1(0,v)=0$ and $\psi_2(u,0)=0$. Furthermore $\psi_1,\psi_2\geqslant 0$.
By Hadamard's lemma there exist smooth functions $\xi$ and $\eta$ such that
\[\psi(u,v)=(u\xi(u,v)^2,v\eta(u,v)^2).\]
We define now
\[\Psi(x_1,\dots,x_{n+1})=(\xi(a,b)x_1,\dots,\xi(a,b)x_{k},\eta(a,b)x_{k+1},\dots,\eta(a,b)x_{n+1})\]
if $z$ is boundary stable, and
\[\Psi(x_1,\dots,x_{n+1})=(\eta(a,b)x_1,\xi(a,b)x_2,\dots,\xi(a,b)x_{k},\eta(a,b)x_{k+1},
\dots,\eta(a,b)x_{n+1})\]
if $z$ is boundary unstable. Here $a$ and $b$ are given by \eqref{eq:abstable} or \eqref{eq:abunstable}.
By construction, $\Psi$ maps $(\widetilde{H},\widetilde{B},\widetilde{C})$ diffeomorphically to the
 triple $(H,B,C)$, where
\begin{align*}
H&=\{a^2\in[0,\e^2],\ b^2\in[0,\e^2],\ x_1\geqslant 0\}\\
B&=\{a^2=-\e^2,\ b^2\in[0,\e^2],\ x_1\geqslant 0\}\\
C&=\{a^2\in[0,\e^2],\ b^2\in[0,\e^2],\ x_1=0\}.
\end{align*}
After substituting for $a$ and $b$ the values from \eqref{eq:abstable} or \eqref{eq:abunstable} (depending on whether $z$ is
boundary stable or unstable), we recover the model \eqref{eq:rightmodel} of a right half-handle   if $z$
is boundary unstable; or the model \eqref{eq:leftmodel} of a left half-handle (both of index $k$).
\end{proof}

The fact that each half-handle can be presented in a left or right
model will be now used to show the following converse to
Theorem~\ref{thm:morsehalfhandle}. The result for non--boundary case
can be found in \cite[Theorem 3.12]{Mi-hcob}.

\begin{proposition}\label{prop:handleiscrit}
Let $(\O,Y)=(\S_0\times[0,1],M_0\times[0,1])$ be a product
cobordism between $(\S_0,M_0)$ and $(\S_1,M_1)\cong(\S_0,M_0)$.
Let us be given a half-handle $(H,C,B)$ of index $k$ and an
embedding of $B_0=C\cap B$ into $M_1$ $($respectively an embedding
of $(B,B_0)$ into $(\S_1,M_1)\,)$, and let $(\O',Y')$ be the result
of a right half-handle attachment along $B_0$ $($respectively, a
left half-handle attachment along $(B,B_0)$$)$  of index $k$. Then,
there exists a Morse function $F\colon(\O',Y')\to\R$, which has a
single boundary unstable critical point  $($respectively, a single
boundary stable critical point$)$ of index $k$ on $H$ and no other
critical points. In particular, $F$ is a Morse function on
a cobordism $(\O',Y')$.
\end{proposition}
\begin{proof}
We shall prove the result for right half-handle attachment, the
other case is completely analogous. The proof consists mostly
on reading `back to front' the proof of Theorem~\ref{thm:morsehalfhandle};
we shall use notation from this theorem, with $\e=1$ and $\rho=2$.

In the case of a right half-handle $B_0$ is
embedded into $M_1$ and we extend this embedding to an embedding
of $B$ into $\S_1$ (see Definition~\ref{def:rha}) and the remark
just after it.

The manifold $\O'$ is constructed in two steps. First,
we glue a handle $\widetilde{H}$ to $\Omega=\Sigma_0\times[0,1]$ along $B$ obtaining a manifold $\O''$. The
result is as in Figure~\ref{fig:contractions} (the figure on the right).
After this gluing, a vertical component of $\partial\widetilde{H}$ appears (in notation of \eqref{eq:BPKC}
this vertical component is $\widetilde{K}$).

We glue now $\Sigma_-\times[-1,1]$ to $\Omega''$ so as to obtain $\Omega'$ as in Figure~\ref{fig:contractions} on the left
and in the way that $\Omega''$ is diffeomorphic to $\Omega'$. The way we do that is the following. The boundary of
$\Sigma_-\times[-1,1]$ decomposes into three parts. The first part is $\Sigma_-\times\{-1\}$; we glue it to
$\Sigma_0\times\{1\}$ (notice that $\Sigma_-$ is $\Sigma_0$ with $B$ removed). The second part is $\partial\Sigma_-\times[-1,1]$.
This part is identified with $\widetilde{K}$, in fact, the flow of $\nabla F$ studied in Lemma~\ref{lem:tanga3}, induces a diffeomorphism
of $\widetilde{K}$ with $\partial\Sigma_-\times[-1,1]$. We glue together $\widetilde{K}$ and $\partial\Sigma_-\times[-1,1]$ using this identification.
The third part of the boundary, that is, $\Sigma_-\times\{1\}$ is not glued. It follows from an argument as in Lemma~\ref{lem:retracts} that $\O'$
is diffeomorphic to $\O''$.

The manifold $\O'$ consists of three components: $\Sigma_0\times[0,1]$, $H$ and $\Sigma_-\times[-1,1]$. We define a function $F$ on each component
separately, namely.

\[F(x)=\begin{cases}
t& \mbox{if} \ x=(v,t)\in\Sigma_0\times\{t\}\subset \Sigma_0\times[0,1]=\O\\
2+t& \mbox{if} \ x=(v,t)\in\Sigma_-\times\{t\}\subset \Sigma_-\times[-1,1]\\
2-\sum\limits_{j=1}^k x_j^2+\sum\limits_{j=k+1}^{n+1}x_j^2&\mbox{if} \ x=(x_1,\dots,x_{n+1})\in H.
\end{cases}
\]
As defined, $F$ is smooth on each of the three components. It is also globally continuous. In fact, the identification of $\widetilde{K}$
with $\Sigma_-\times[-1,1]$ can be done so that $F$ is continuous on $\widetilde{K}$. On the part $\Sigma_0\times\{1\}\subset\O'$, all the three components
give the same value, that is $1$.

Given the construction of $F$, it remains to perturb $F$ (that is, to approximate it uniformly near $\wt{K}\cup\S_0\times\{1\}$)
to a smooth function and in the way that $F$ does not get any new critical points. For a general piecewise smooth
function this is impossible, we can consider the real valued function $x\mapsto |x|$: any smooth approximation must have a critical point near $x=0$. The reason for this
is that near the non--smooth point the topology of level sets of $|x|$ changes. This is essentially the main obstruction

In our situation, the topology of the level sets of $F$ does not change near $\widetilde{K}$, nor near $\Sigma_0\times\{1\}$, that is, near
any gluing region. This is enough to show that $F$ can be approximated near its non--smooth locus by a smooth function without
introducing additional critical points. The proof of this fact is standard, but technical.
Instead of giving all the details, we sketch a proof of a weaker result, Lemma~\ref{lem:approx}. This result takes care of approximating the function $F$ near
$\S_0\times\{1\}$. Approximation near the whole of $\wt{K}\cup\S_0\times\{1\}$ follows essentially the same pattern and is left to the reader.

Given the approximation result, the proof of Proposition~\ref{prop:handleiscrit} is finished.
\end{proof}

\begin{lemma}[Approximating piecewise smooth functions by smooth functions]\label{lem:approx}
Suppose that $N$ is a smooth, compact manifold. Let $\pi\colon N\times[-1,1]\to[-1,1]$ be the projection
onto the second factor. Let $N_0=N\times\{0\}$, $N_+=N\times[0,1]$ and $N_-=N\times[-1,0]$.
Let $f\colon N\times[-1,1]$ be a continuous function. Let $f_+$ and $f_-$ be the restrictions to $N_+$ and $N_-$ respectively.
Suppose that
\begin{itemize}
\item[(a)] $f_+$ and $f_-$ are smooth and have no critical points on $N\times[-1,1]$;
\item[(b)] $f_+^{-1}(0)=f_-^{-1}(0)=N_0$;
\item[(c)] the image of $f_+$ is contained in $\R_{\ge 0}$ and the image of $f_-$ is contained in $\R_{\le 0}$;
\item[(d)] the scalar product $\langle f_{\pm},\nabla\pi\rangle$ is positive on $N_{\pm}\setminus N_0$.
\end{itemize}
Then for any $\theta>0$ there exist $\varepsilon,\delta\in(0,\theta)$ and a smooth function $g\colon N\times[-1,1]\to\R$ such that
\begin{itemize}
\item[(i)] $g$ agrees with $f_-$ on $N\times[-1,-\delta]$ and with $f_+$ on $N\times[\delta,1]$;
\item[(ii)] $g$ takes values in $[-\varepsilon,\varepsilon]$ on $N\times[-\delta,\delta]$;
\item[(iii)] $g$ has no critical points on $N\times[-1,1]$.
\end{itemize}
\end{lemma}
\begin{proof}[Sketch of proof]
By compactness, the continuity of $f_\pm$ and assumptions (b), (c)
there exists $\delta'>0$ such that $f_+(N\times[0,\delta'])\subset[0,\theta/2]$ and $f_-(N\times[-\delta',0])\subset[-\theta/2,0]$.
We set $\varepsilon=\theta$ and
$\delta=\min(\delta',\theta/2)$.

Choose a partition of unity subordinate to the covering $[-1,1]=[-1,-\delta/2)\cup(-\delta,\delta)\cup(\delta/2,1]$. The three
functions corresponding to this partition are denoted by $\phi_-$, $\phi_0$ and $\phi_+$ respectively.

Define $\Phi_\bullet\colon N\times[-1,1]\to[0,1]$ as compositions $\phi_{\bullet}\circ\pi$, where $\bullet$ is any of `+', `-' and `$0$'.
Consider the vector field
\[v=\Phi_+\nabla f_++\Phi_0\nabla\pi+\Phi_-\nabla f_-.\]
By point (a) of the assumptions $v$ is a smooth vector field. Assumption~(d) implies that
$\langle \nabla\pi,v\rangle>0$ everywhere on $N\times[-1,1]$, that is, $v$ is a gradient--like vector field for $\pi$.
In particular, $v$ does not vanish on $N\times[-1,1]$.

Let $h$ be a positive $C^\infty$ function. Set $v_h=hv$. Then $v_h$ is also a gradient--like vector field for $\pi$.
The trajectories of $v$ coincide with those of $v_h$: multiplication by $h$ changes only the speed of going along a trajectory.

We integrate the vector field $v_h$ to a function $g_h\colon N\times[-1,1]\to\R$. This means
we first set $g_h\equiv f_-$ on $N\times\{-1\}$. Next, suppose $x\in N\times(-1,1]$.
Let $\gamma\colon U\to N\times[-1,1]$ (here $U$ is a closed interval) be a trajectory of $v_h$  such that $\gamma(0)=x$.
Since $v_h$ is gradient--like for $\pi$,  $\gamma$
must have come from $N\times\{-1\}$ in the past, more precisely, there exist $t_x<0$ and $y\in N\times\{-1\}$ such that $\gamma(t_x)=y$. We set
\[g_h(x)=g_h(y)-t_x.\]
Since $v_h$ is smooth, by the implicit function theorem $g_h$ is a smooth function.

Choosing the normalizing function $h$ appropriately we can guarantee that $g:=g_h$ satisfies (i).
Namely, we set $h=||\nabla f_-||^{-2}$ so that the directional derivative $\langle v_h,\nabla f_-\rangle\equiv 1$
on $N\times[-1,-\delta]$. This implies that $g=f_-$ on $N\times[-1,-\delta]$.
The choice of $h$ on $N\times[-\delta,\delta]$ is such that the time the trajectory
goes from a point $x_-\in N\times\{-\delta\}$ to some $x_+\in N\times\{\delta\}$ is equal to $f_+(x_+)-f_-(x_-)$. The
latter expression is positive by assumption (b). This implies that
$g=f_+$ on $N\times\{\delta\}$ and condition (ii) is satisfied automatically.
Finally
we set $h=||\nabla f_+||^{-2}$ on $N\times[\delta,1]$. The verification of condition
(i) is straightforward.

As $g_h$ is strictly increasing on trajectories of $v_h$, it cannot have any critical points.
\end{proof}

\subsection{Left and right product cobordisms and traces of handle attachments}\label{s:larpc}

In this subsection we create  a dictionary between surgery theoretical notions
(traces of handle attachments and detachments) and Morse
theoretical (additions of half-handles). The main result of
this subsection, Proposition~\ref{prop:lpc}, is a direct consequence of the results
proved earlier in the article.

To begin with,
let $(\O,Y)$ be a cobordism between $(\S_0,M_0)$ and $(\S_1,M_1)$.

\begin{definition}
We shall say that $\O$ is a \emph{left product cobordism} if
$\O\cong \S_0\times[0,1]$. Similarly, if $\O\cong\S_1\times[0,1]$,
then we shall say that $\O$ is a \emph{right product cobordism}.
\end{definition}
\begin{proposition}\label{prop:lpc}(a)
If $(\O,Y)$ is a cobordism between $(\S_0,M_0)$ and $(\S_1,M_1)$
consisting only of left half-handle attachments, then it is a
left-product cobordism. Likewise, if it consists only of right
half-handle attachments, then it is a right product cobordism.

(b) Let $F\colon \O\to[0,1]$ be a Morse function in the sense of
Definition~\ref{def:Morse}. Assume that $F$ has no critical points
in the interior of $\O$. If all critical points on the boundary
are boundary stable, then $F$ is a left-product cobordism. If all
critical points are boundary unstable, then $F$ is a right product
cobordism.
\end{proposition}
\begin{proof}
The two statements (a) and (b) are equivalent via
Theorem~\ref{thm:morsehalfhandle} and
Proposition~\ref{prop:handleiscrit}. The stable-unstable
(right-left) statements are also equivalent by replacing the
Morse function $F$ by $-F$. The stable case follows from
Lemma~\ref{lem:leftistrivial}.
\end{proof}

The next results of this subsection will be not used in this paper, but we insert them
because they bridge surgery techniques and applications,
e.g. with  \cite{Ra} or \cite{BNR2}.

In order to clarify what we wish, let us recall that by
Theorem~\ref{thm:morsehalfhandle} if  a Morse function $F$ defined
on a cobordism $(\O,Y)$ has only one   critical point of boundary
type then $(\O,Y)$ is a half-handle attachment.
Proposition~\ref{prop:handleiscrit} is the converse of this; the
(total) space of a half-handle attachment can be thought as a
cobordism  with a Morse function on it with only one critical
point.

We wish to establish the analogues of these statements `at the level
of $\Sigma$'. In Subsection \ref{ss:elem} we proved that the
output of a right/left half-handle attachment induces
a handle attachment/detachment at the level of $\Sigma$. The next lemma is
the converse of this statement. (In fact, the output cobordism
provided by it can be identified with the cobordism constructed in
 Proposition~\ref{prop:handleiscrit}.)

\begin{lemma}\label{lem:traceconstruction}
Assume that $(\S_1,M_1)$ is the result of a handle attachment
(respectively detachment) to $(\S_0,M_0)$. Then, there exists a
cobordism $(\O,Y;\S_0,M_0,\S_1,M_1)$ such that
$\O\cong\S_1\times[0,1]$ (respectively $\O\cong\S_0\times[0,1]$).
\end{lemma}
\begin{proof}
Assume that $(\S_1,M_1)$
arises from a handle attachment to $(\S_0,M_0)$, i.e.
$\S_1=\S_0\cup D^k\times D^{n-k}$. Let us define
$\O=\S_1\times[0,1]$.
The boundary $\p\O$ can be split as
\begin{align*}
\partial\Omega=&\left(\S_0\cup D^k\times D^{n-k}\right)\times\{0\}\cup \left(M_1\times[0,1]\right)\cup\left(\S_1\times\{1\}\right)\\
&= \S_0\times\{0\}\cup Y\cup \S_1\times\{1\},
\end{align*}
where $Y= D^k\times D^{n-k}\cup \left(M_1\times[0,1]\right)$. Its
$D^k\times D^{n-k}$ part can be `pushed inside' $\O$ transforming
(diffeomorphically) $\O$ into a cobordism, see Figure~\ref{fig:pushedinside}.

An analogous construction can be used in the case of a handle detachment.
If $(\S_1',M_1')$ is  the result of a handle detachment from $(\S_0,M_0)$, then the
trace of the handle detachment is the
cobordism between $(\S_0,M_0)$ and $(\S'_1,M'_1)$ such that
$$(\O',Y')=(\S_0\times[0,1],M_0\times[0,1]\cup D^{k}\times
D^{n-k}).$$
\end{proof}

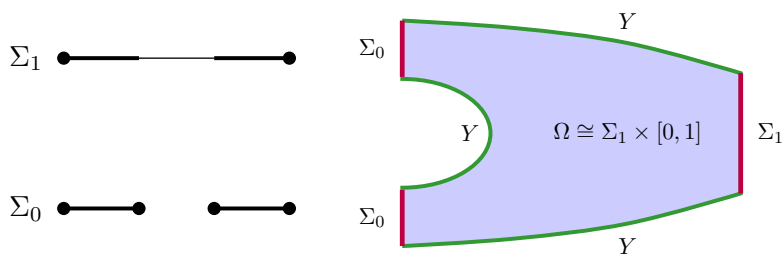
\begin{figure}
\begin{pspicture}(-5,-2)(5,2)
\psline[linewidth=1.5pt]{*-*}(-4.5,-1)(-3.5,-1)
\psline[linewidth=1.5pt]{*-*}(-2.5,-1)(-1.5,-1)\rput(-5,-1){$\S_0$}
\psline[linewidth=1.5pt]{*-}(-4.5,1)(-3.5,1)\psline[linestyle=solid,linewidth=0.5pt](-3.5,1)(-2.5,1)\psline[linewidth=1.5pt]{-*}(-2.5,1)(-1.5,1)\rput(-5,1){$\S_1$}
\pscustom[fillstyle=solid,fillcolor=blue,opacity=0.2,linecolor=darkgreen,linewidth=1.5pt]{%
\pscurve(0,1.5)(1.5,1.4)(3,1.2)(4.5,0.8)
\pscurve[liftpen=1](4.5,-0.8)(3,-1.2)(1.5,-1.4)(0,-1.5)}
\psellipticarc[fillcolor=white,fillstyle=solid,linecolor=darkgreen,linewidth=1.5pt](0,0)(1.2,0.75){268}{92}
\psline[linecolor=purple,linewidth=1.8pt](0,1.5)(0,0.75)
\psline[linecolor=purple,linewidth=1.8pt](0,-1.5)(0,-0.75)
\rput(-0.4,1.125){\psscalebox{0.8}{$\S_0$}}
\rput(-0.4,-1.125){\psscalebox{0.8}{$\S_0$}}
\psline[linecolor=purple,linewidth=1.8pt](4.5,-0.8)(4.5,0.8)\rput(4.9,0){\psscalebox{0.8}{$\S_1$}}
\rput(3,0){\psscalebox{0.8}{$\O\cong\S_1\times[0,1]$}}
\rput(3,1.5){\psscalebox{0.8}{$Y$}}
\rput(3,-1.5){\psscalebox{0.8}{$Y$}}
\rput(0.9,0){\psscalebox{0.8}{$Y$}}
\end{pspicture}
\caption{Lemma~\ref{lem:traceconstruction}. On the left a 1-handle is attached to $\S_0$. On the right there is a cobordism between $\S_0$
and $\S_1$, which is a right product cobordism.}\label{fig:pushedinside}
\end{figure}

\begin{definition}\label{def:trace}
The cobordism $(\O,Y;\S_0,M_0,\S_1,M_1)$ determined by the
Lemma~\ref{lem:traceconstruction} is called the \emph{trace} of a
handle attachment of $(\S_0,M_0)$ (respectively the \emph{trace} of a handle detachment).
\end{definition}

\section{Splitting interior handles}\label{S:sih}
We prove here the theorem about moving critical points to the boundary.

\begin{theorem}\label{thm:handsplit}
Assume that on a cobordism $(\O,Y)$ between $(\S_0,M_0)$ and
$(\S_1,M_1)$ we have a Morse function $F$ with a single
critical point $z$ of index $k\in\{1,\dots,n\}$ in the interior of
$\O$ situated on the  level set $\Sigma_{1/2}=F^{-1}(F(z))$. If
\begin{equation}\label{eq:connectwithboundary}
\textrm{\emph{the connected component of $\Sigma_{1/2}$ containing $z$ has non-empty intersection with $Y$}},
\end{equation}
then there exists a function $G\colon\O\to[0,1]$, such that
\begin{itemize}
\item $G$ agrees with $F$ in a neighbourhood of $\S_0\cup\S_1$;
\item $\nabla G$ is everywhere tangent to $Y$;
\item $G$ has exactly two critical points $z^s$ and $z^u$, which are both on the boundary and of index $k$. The point $z^s$
is boundary stable and $z^u$ is boundary unstable.
\item There exists a Riemannian metric such that there is a single trajectory of $\nabla G$ from $z^s$ to $z^u$ inside $Y$.
\end{itemize}
\end{theorem}

\begin{remark}\label{rem:carefulreading}
A careful reading of the proof shows that we can in fact construct a smooth homotopy $G_t$ such that $F=G_0$, $G=G_1$ and
there exists $t_0\in(0,1)$ such that $G_t$ has a single interior critical point for $t<t_0$, two boundary critical points for $t>t_0$ and a degenerate critical
point on the boundary for $t=t_0$. See Remark~\ref{rem:carefulreading2}.
\end{remark}

The proof of Theorem~\ref{thm:handsplit} occupies
Sections~\ref{ss:withhyp} to \ref{ss:proofhyp}. We make a detailed
discussion of  Condition~\eqref{eq:connectwithboundary} in
Section~\ref{ss:step1}.

\subsection{About the proof}\label{sec:motivation}

The argument is based on  the following two-dimensional picture.
Consider the set $Z=\{(x,y)\in\R^2\colon x\geqslant 0\}$ and the
function $D\colon Z\to\R$ given by
\[D(x,y)=y^3-yx^2+ay,\]
where $a\in\R$ is a parameter. Observe that the boundary of $Z$ given by $\{x=0\}$ is invariant under the gradient flow of $D$ (see Figure~\ref{fig:stream}).

\begin{lemma} For $a>0$, $D$ has a single Morse critical point in the interior of $Z$. For $a<0$, $D$ has two Morse critical points on the boundary of $Z$.
\end{lemma}
\begin{proof}
Critical points of $D$ are given by $\frac{\partial D}{\partial x}=\frac{\partial D}{\partial y}=0$, that is, $xy=0$ and $3y^2-x^2+a=0$. The first
equation means that $y=0$ or $x=0$ and then we get solutions $(\pm \sqrt{a},0)$ and $(0,\pm\sqrt{-a/3})$. In the case $a>0$ we consider only first
two solutions (and only one of them belongs to $Z$), while if $a<0$, only the last two solutions are real and they correspond to boundary critical points.
Checking that these critical points are Morse is straightforward and is left to the reader.
\end{proof}
For $a=0$, $D$ acquires a $D_4^{-}$ singularity at the origin (see e.g.  \cite[Section 17.1]{AGV}).

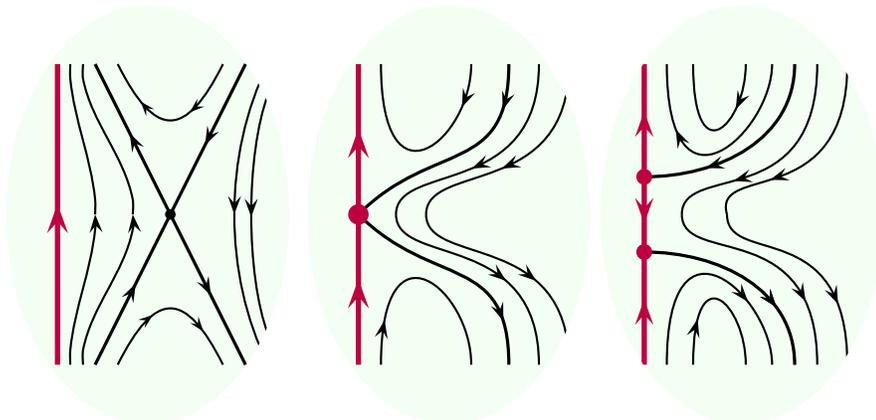
\begin{figure}
\begin{pspicture}(-5,-3)(5,3)
\rput(-2,-3){\begin{psclip}{\psellipse[linestyle=none,fillcolor=green,opacity=0.05,fillstyle=solid](-2.3,3)(1.9,2.8)}
\psline[linewidth=2pt,ArrowInside=->,arrowsize=8pt,linecolor=purple](-3.5,1)(-3.5,5)
\psline[linewidth=1.2pt, ArrowInside=->,arrowsize=5pt](-3,1)(-2,3)
\psline[linewidth=1.2pt, ArrowInside=->,arrowsize=5pt](-1,5)(-2,3)
\psline[linewidth=1.2pt, ArrowInside=->,arrowsize=5pt](-2,3)(-3,5)
\psline[linewidth=1.2pt, ArrowInside=->,arrowsize=5pt](-2,3)(-1,1)
\pscircle[fillstyle=solid,fillcolor=black](-2,3){0.07}
\pscurve[arrowsize=5pt]{->}(-3.16,1)(-3.16,1.2)(-2.5,2.8)(-2.5,3)
\pscurve(-2.5,3)(-2.5,3.2)(-3.16,4.8)(-3.16,5)
\pscurve[arrowsize=5pt]{->}(-3.33,1)(-3.33,1.2)(-3.0,2.8)(-3.0,3)
\pscurve(-3.0,3)(-3.0,3.2)(-3.33,4.8)(-3.33,5)
\psbezier[ArrowInside=->, arrowsize=5pt](-0.7,5)(-1.3,3.8)(-1.3,2.2)(-0.7,1)
%\psline[arrowsize=5pt]{->}(-1.4,3.2)(-1.4,3)
\psbezier[ArrowInside=->, arrowsize=5pt](-0.7,4.6)(-1.0,3.8)(-1.0,2.2)(-0.7,1.4)
%\psline[arrowsize=5pt]{->}(-1.07,3.2)(-1.07,3)
\psbezier[ArrowInside=->, arrowsize=4pt, ArrowInsidePos=23](-2.7,1)(-2.2,2)(-1.8,2)(-1.3,1)
%\psline[arrowsize=5pt]{->}(-2,2.04)(-1.9,2.04)
\psbezier[ArrowInside=->, arrowsize=4pt, ArrowInsidePos=22](-1.3,5)(-1.8,4)(-2.2,4)(-2.7,5)
%\psline[arrowsize=5pt]{->}(-2,3.97)(-2.1,3.97)
\end{psclip}}
\rput(2,3){\begin{psclip}{\psellipse[linestyle=none,fillcolor=green,opacity=0.05,fillstyle=solid](-2.3,-3)(1.9,2.8)}
\psline[linewidth=2pt,ArrowInside=->,arrowsize=8pt,linecolor=purple](-3.5,-5)(-3.5,-3)
\psline[linewidth=2pt,ArrowInside=->,arrowsize=8pt,linecolor=purple](-3.5,-3)(-3.5,-1)
\psbezier[ArrowInside=->, arrowsize=5pt, ArrowInsidePos=0.8, linewidth=1.2pt](-3.5,-3)(-2.5,-4)(-1.5,-3.5)(-1.5,-5)
\psbezier[ArrowInside=->, arrowsize=5pt, ArrowInsidePos=0.2, linewidth=1.2pt](-1.5,-1)(-1.5,-2.5)(-2.5,-2)(-3.5,-3)
\psbezier[ArrowInside=->, arrowsize=5pt, ArrowInsidePos=0.2](-3.2,-5)(-3.2,-3)(-2,-4)(-2,-5)
\psbezier[ArrowInside=->, arrowsize=5pt, ArrowInsidePos=0.2](-2,-1)(-2,-2)(-3.2,-3)(-3.2,-1)
\psbezier[ArrowInside=->, arrowsize=5pt, ArrowInsidePos=50](-1.1,-1)(-1.1,-2.8)(-3,-2.3)(-3,-3)
\psbezier[ArrowInside=->, arrowsize=5pt, ArrowInsidePos=50](-3,-3)(-3,-3.7)(-1.1,-3.2)(-1.1,-5)
\psbezier[ArrowInside=->, arrowsize=5pt, ArrowInsidePos=50](-0.7,-1)(-0.7,-2.8)(-2.6,-2.3)(-2.6,-3)
\psbezier[ArrowInside=->, arrowsize=5pt, ArrowInsidePos=50](-2.6,-3)(-2.6,-3.7)(-0.7,-3.2)(-0.7,-5)
\pscircle[fillstyle=solid,fillcolor=purple,linestyle=none](-3.5,-3){0.15}
\end{psclip}}
\rput(1.3,3){\begin{psclip}{\psellipse[linestyle=none,fillcolor=green,opacity=0.05,fillstyle=solid](2.3,-3)(1.9,2.8)}
\psline[linewidth=2pt,ArrowInside=->,arrowsize=7pt,linecolor=purple](1,-5)(1,-3.5)
\psline[linewidth=2pt,ArrowInside=->,arrowsize=7pt,linecolor=purple](1,-2.5)(1,-1)
\psline[linewidth=2pt,ArrowInside=->,arrowsize=7pt,linecolor=purple](1,-2.5)(1,-3.5)
\psbezier[ArrowInside=->, arrowsize=5pt, ArrowInsidePos=0.7, linewidth=1.2pt](3,-1)(3,-2)(2,-2.5)(1,-2.5)
\psbezier[ArrowInside=->, arrowsize=5pt, ArrowInsidePos=0.7, linewidth=1.2pt](1,-3.5)(2,-3.5)(3,-4)(3,-5)
\psbezier[ArrowInside=->, arrowsize=5pt, ArrowInsidePos=0.7](1,-3.5)(2,-3.5)(3,-4)(3,-5)
\psbezier[ArrowInside=->, arrowsize=5pt, ArrowInsidePos=0.7](1.3,-5)(1.3,-2.8)(2.7,-4.1)(2.7,-5)
\psbezier[ArrowInside=->, arrowsize=5pt, ArrowInsidePos=0.3](1.65,-5)(1.65,-3.5)(2.35,-4.2)(2.35,-5)
\psbezier[ArrowInside=->, arrowsize=5pt, ArrowInsidePos=0.7](2.7,-1)(2.7,-1.9)(1.3,-3.2)(1.3,-1)
\psbezier[ArrowInside=->, arrowsize=5pt, ArrowInsidePos=0.3](2.35,-1)(2.35,-1.8)(1.65,-2.5)(1.65,-1)
\psbezier[ArrowInside=->, arrowsize=5pt, ArrowInsidePos=0.7](3.3,-1)(3.3,-3)(1.5,-2.3)(1.5,-3)
\psbezier[ArrowInside=->, arrowsize=5pt, ArrowInsidePos=0.7](1.5,-3)(1.5,-3.7)(3.3,-3)(3.3,-5)
\psbezier[ArrowInside=->, arrowsize=5pt, ArrowInsidePos=0.7](3.7,-1)(3.7,-3)(1.9,-2.3)(2.1,-3)
\psbezier[ArrowInside=->, arrowsize=5pt, ArrowInsidePos=0.7](2.1,-3)(2.1,-3.7)(3.7,-3)(3.7,-5)
\pscircle[fillstyle=solid,fillcolor=purple,linestyle=none](1,-3.5){0.12}
\pscircle[fillstyle=solid,fillcolor=purple,linestyle=none](1,-2.5){0.12}
\end{psclip}}
\end{pspicture}
\caption{The trajectories of the gradient vector field of $D$ for values of $a>0$, $a=0$ and $a<0$.}\label{fig:stream}
\end{figure}

The proof of Theorem~\ref{thm:handsplit} starts with introducing
`local/global' coordinates $(x,y,u_1,\dots,u_{n-1})$ at $z$, in
which $F$ has the form $D(x,y)\pm u_1^2\pm\dots\pm u_{n-1}^2$,
hence it also parametrizes a neighbourhood of a path  connecting
$z$  with a point of $Y$. Then we change the parameter $a$ (which
we originally assume to be equal to $1$) to $-\delta$, where
$\delta$ is very small positive number (which corresponds to
moving the critical point to the boundary along the chosen path).

\subsection{Proof of Theorem~\ref{thm:handsplit} under additional assumption}\label{ss:withhyp}

We first give the proof assuming the existence of such coordinate
system as in \ref{sec:motivation}, described explicitly
 in the next  proposition (which is proved in Section~\ref{ss:proofhyp}).
 We use the hypotheses and notation of Theorem~\ref{thm:handsplit}.

\begin{proposition}\label{prop:localcoordinates}
There exists $\eta>0$, $\eta\ll 1$ and an open `half-disk' $U\subset\O$,
intersecting $Y$ along a disk, and  coordinates
$x,y,u_1,\dots,u_{n-1}$ such that in these coordinates $U$ is
given by
\[0\leqslant x<3+\eta,\ |y|<\eta,\ \sum_{j=1}^{n-1}u_j^2<\eta^2,\]
$U\cap Y$ is given by $\{x=0\}$, and in these coordinates $F$ is
given by
\[y^3-yx^2+y+\frac12+\sum_{j=1}^{n-1}\epsilon_ju_j^2,\]
where $\epsilon_1,\dots,\epsilon_{n-1}\in\{\pm 1\}$ are choices of signs. In particular $\#\{j\colon \epsilon_j=-1\}=k-1$, where $k=\ind_zF$.
\end{proposition}
 Assuming the proposition, we prove Theorem~\ref{thm:handsplit}.
Let us introduce some abbreviations.
\be\label{eq:notationonu}
\vu=(u_1,\dots,u_{n-1}),\ \ \ \ \ \
\vu^2=\sum_{j=1}^{n-1}\epsilon_ju_j^2,\ \ \ \ \ \
||\vu||^2=\sum_{j=1}^{n-1}u_j^2.
\ee
We fix a small real number $\e>0$ such that $\e\ll \eta$ and
two subsets $U_{1}\subset U_2$ of $U$ by
\begin{align*}
U_1&=\{|y|\leqslant\e,\ x\leqslant 3\}\cup\{(x-3)^2+y^2\leqslant\e^2\}\\
U_2&=\{|y|\leqslant 2\e,\ x\leqslant 3\}\cup\{(x-3)^2+y^2\leqslant4\e^2\}.
\end{align*}
The difference $\DU:=\overline{U_2\setminus U_1}$ splits into two subsets $S_1\cup S_2$
(see Figure~\ref{fig:subsets}), where
\[S_1=\DU\cap\{x\leqslant 3\}, \ \ \ \ \ S_2=\DU\cap\{x\geqslant 3\}.\]
\begin{figure}
\begin{pspicture}(-5,-3)(5,3)
\rput(0,0){\psscalebox{0.7}{%
\pspolygon[linestyle=none,hatchcolor=green,hatchsep=1.5pt, fillstyle=vlines,opacity=0.3](-5,-3)(2,-3)(2,3)(-5,3)
\psarc[linewidth=1pt,hatchcolor=orange,hatchsep=1.5pt, fillstyle=hlines,opacity=0.3](1.95,0){3}{270}{90}
\pspolygon[linestyle=none,fillcolor=lightblue,fillstyle=solid](-5,-1.5)(2,-1.5)(2,1.5)(-5,1.5)
\psarc[linewidth=1pt,fillcolor=lightblue,fillstyle=solid](1.95,0){1.5}{270}{90}
\psline[linewidth=1.3pt,arrowsize=7pt]{->}(-5,-4)(-5,4)\rput(-5.2,4){\psscalebox{1.2}{$y$}}%
\psline[linewidth=1.3pt,arrowsize=7pt]{->}(-5,0)(6,0)\rput(6,0.2){\psscalebox{1.2}{$x$}}%
\psline[linewidth=1pt](-5,-3)(2,-3)%
\psline[linewidth=1pt](-5,3)(2,3)%
\psline[linewidth=1pt](-5,1.5)(2,1.5)%
\psline[linewidth=1pt](-5,-1.5)(2,-1.5)%
\rput(1,0.7){\psscalebox{1.2}{$U_1$}}
\rput(0,2){\psscalebox{1.2}{$S_1$}}
\rput(4,0.3){\psscalebox{1.2}{$S_2$}}
\rput(2,2.25){\psscalebox{1.2}{$\DU=S_1\cup S_2$}}
\pscircle[fillstyle=solid,fillcolor=black](-3,0){0.1}\rput(-3,-0.35){\psscalebox{1.4}{$z$}}
\pscircle[fillstyle=solid,fillcolor=black](2,0){0.1}\rput(2,-0.35){\psscalebox{1.2}{$(3,0)$}}}}
\end{pspicture}
\caption{Sets $U_1,\DU,S_1$ and $S_2$ in two dimensions (coordinates $x$ and $y$).}\label{fig:subsets}
\end{figure}
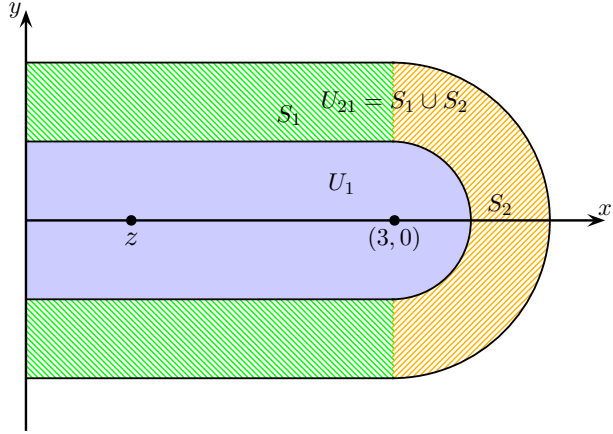
For a point $v=(x,y,u_1,\dots,u_{n-1})\in U$, let us define:
\[
\widetilde{s}(v)=\begin{cases}
1&\mbox{if} \  v \in U_1,\\
0&\mbox{if} \ v\in \overline{U \setminus  U_2},\\
2-\frac{|y|}{\e}&\mbox{if} \  v\in S_1,\\
2-\frac{\sqrt{(x-3)^2+y^2}}{\e}&\mbox{if} \  v\in S_2.
\end{cases}
\]
The above formula defines a continuous function $\widetilde{s}\colon U_2\to[0,1]$. It is smooth away of $\partial S_1\cup\partial S_2$.
We can perturb it to a $C^{\infty}$ function $s\colon U_2\to[0,1]$, with the following
properties:
\begin{itemize}
\item[(S1)] $s^{-1}(1)=U_1$, $s^{-1}(0)=\{|y|\geqslant 2\e-\e^2\}\cup \{(x-3)^2+y^2\geqslant 4\e^2-\e^3, x\ge 3\}$ (this is a thin region near
the boundary of $U_2$).
\item[(S2)] $\frac{\partial s}{\partial u_j}=0$ for any $j=1,\dots,n-1$;
\item[(S3)] $\frac{\partial s}{\partial x}=0$, and $\left|\frac{\partial s}{\partial y}\right|<\frac{2}{\e}$ at all points of $S_1$. Furthermore
$y\frac{\p s}{\p y}<0$ at all points of $S_1$;
\item[(S4)] if $v\in S_2$ and we choose radial coordinates $x=3+r\cos\theta$, $y=r\sin\theta$ (where $r\in[\e,2\e]$ and $\theta\in[-\frac{\pi}{2},\frac{\pi}{2}]$),
then $\left|\frac{\partial
s}{\partial r}\right|<\frac{2}{\e}$ and $\left|\frac{\partial s}{\partial \theta}\right|<\e$
\end{itemize}
Observe that $\wt{s}$ satisfies (S1)--(S4) at every point for which
it is smooth; the only issue is that on $S_1\cap S_2$, $\wt{s}$ fails to be $C^2$.

Now let us choose a smooth decreasing function $\phi\colon[0,\eta^2]\to[0,1]$, which is equal to $0$ on $[\frac34\eta^2,\eta^2]$ and $\phi(0)=1$.
We define now a new function $b\colon U_2\to[0,1]$ by the formula

\be\label{eq:defofb}
b(x,y,\vu)=s(x,y,\vu)\cdot\phi(||\vu||^2).
\ee

Let us finally define the function $G\colon\O\to[0,1]$ by
\be\label{eq:defG}
G(w)=\begin{cases} F(w)& \mbox{if} \ w\not\in U_2\\
y^3-yx^2+y-(\delta+1)b(x,y,\vu)y+\frac12+\vu^2& \mbox{if} \  w=(x,y,\vu)\in U_2,
\end{cases}
\ee
where $\delta>0$ is a very small number. Later we shall show that it is enough to take $\delta<\e^2/2$.
In the following lemmas we shall prove that $G$ satisfies the conditions of Theorem~\ref{thm:handsplit}.

\begin{lemma}
The function $G$ is smooth.
\end{lemma}
\begin{proof}
It is a routine checking and we leave it for the reader.
\end{proof}

In the next two lemmas we show that $G$ has no critical points in $\DU$.

\begin{lemma}\label{lem:nocritU11}
$G$ has no critical points on $\DU\cap\{y=0\}$.
\end{lemma}
\begin{proof}
If $(x,0,u_1,\dots,u_{n-1})\in \DU$ then $x>3$. Consider the derivative over $y$ of $G$:
\be\label{eq:dgdy}
\frac{\p G}{\p y}=3y^2-x^2+1-(\delta+1)b-(\delta+1)\phi(u_1^2+\dots+u_{n-1}^2)\frac{\partial s}{\partial y}y.
\ee
Taking $y=0$ we get
$-x^2+1-(\delta+1)b.$
Since $b$ takes values in $[0,1]$ and $x>3$, one gets $\frac{\p G}{\p y}<0$.
\end{proof}

\begin{lemma}\label{lem:nocritU12}
If $\delta<3\e^2$, then $G$ has no critical points on $\DU\cap\{y\neq 0\}$.
\end{lemma}
\begin{proof}
Assume that $\frac{\partial G}{\partial x}=0$ for some $(x,y,\vu)$. Then
\[y\left(-2x-(\delta+1)\frac{\p s}{\p x}\,\phi\right)=0.\]
As $y\neq 0$, the expression in parentheses should be zero.
If $0<x\leqslant 3$, then by (S3) we have $\frac{\p s}{\p x}=0$. Hence the above equality can not hold. Assume that $x=0$.
In the derivative over $y$ (see equation \eqref{eq:dgdy}), the expression
$-(\delta+1)\phi\frac{\p s}{\p y}y$
is non-negative by (S3).  Furthermore $b<1$, hence
\[\frac{\p G}{\p y}\geqslant 3y^2-\delta.\]
Now if $\delta<3\e^2$ then there are no critical points with $x=0$.
It remains to deal with the case $(x,y,u_1,\dots,u_{n-1})\in S_2$.
Consider the derivative $\frac{\p G}{\p y}$. By (S4) and the chain rule we have
\[
\left|y\frac{\partial s}{\partial y}\right|=\left|y\frac{\partial r}{\partial y}\frac{\partial s}{\partial r}+y\frac{\partial\theta}{\partial y}\frac{\partial s}{%
\partial\theta}\right|\le
\left|\frac{y^2}{r}\cdot\frac{\partial s}{\partial r}\right|+
\left|\frac{(x-3)y}{r^2}\cdot\frac{\partial s}{\partial\theta}\right|<
r\frac{2}{\e}+\e<5.
\]
Furthermore $|1-(\delta+1)b|\leqslant 1$, and $|3y^2|<1$ because
$\e$ is small. As $x\geqslant 3$, we have $\frac{\p G}{\p y}<0$ on
$S_2$.
\end{proof}

On $U_1$ the function $G$ is given by
\be\label{eq:gonu1}
G(x,y,\vu)=y^3-yx^2-\delta y+\vu^2+\frac12.
\ee

As in Section~\ref{sec:motivation} we study the critical points in $U_1$.
\begin{lemma}
$G$ has two critical points on $U_1$ at
\begin{align*}
z^s&:=(0,\sqrt{\delta/3},0,\dots,0)\\
z^u&:=(0,-\sqrt{\delta/3},0,\dots,0).
\end{align*}
Both critical points are boundary, both of Morse index $k$,
$z^s$ is stable, while $z^u$ is unstable.
\end{lemma}
\begin{proof}
The derivative of $G$  vanishes only at $z^s$ and $z^u$. Indices are immediately computed from \eqref{eq:gonu1}.
The point $z^s$ is boundary stable, because for $z^s$ the expression $-yx^2$ is negative and the boundary
is given by $x=0$, hence it is attracting in the normal direction. Similarly we prove for $z^u$. See also Figure~\ref{fig:stream} for
the two-dimensional picture.
\end{proof}
\begin{remark}\label{rem:carefulreading2}
If we define $G_t=y^3-yx^2+y-t(\delta+1)b\cdot y+\frac12+\vu^2$ for $t\in[0,1]$, then the same argument as in Lemmas~\ref{lem:nocritU11} and \ref{lem:nocritU12}
shows that $G_t$ has no critical points in $U_2\setminus U_1$. As for critical points in $U_1$, observe that on $U_1$ we have
\[G_t=y^3-yx^2+(1-t(1+\delta))y+\frac12+\vu^2.\]
Let $t_0=\frac{1}{1+\delta}$. If $t>t_0$, the function $G_t$
has two critical points on the boundary $Y$, while for $t<t_0$, $G_t$ has a single critical point in the interior $U_1\setminus Y$. If $t=t_0$, $G_t$ has
a single degenerate critical point on $Y$. In this way we construct an `isotopy'
between $F$ and $G$.
\end{remark}

Let us now choose a Riemannian metric $g'$ on
\[U_1':=U_1\cap \{||\vu||<\e\}\]
by the condition that $(x,y,u_1,\dots,u_{n-1})$ be orthonormal coordinates (compare Remark~\ref{rem:metric2} below).
Clearly, any metric $g$ on $\O$ can be changed near $U_1$ so as to agree with $g'$ on $U_1'$. In this metric the gradient of $G$ is
\begin{equation}\label{eq:gradG}
(-2xy,3y^2-x^2-\delta,2\epsilon_1u_1,\dots,2\epsilon_{n-1}u_{n-1}).
\end{equation}
We want to show that there is a single trajectory starting from $z^s$ and terminating at $z^u$. Clearly, there is one
trajectory from $z^s$ to $z^u$ which stays in $U_1'$ (having $y=0$ and $\vu=0$). In order to eliminate the others,
we need the following lemma.
\begin{lemma}\label{lem:nablaGtraj}
Let $\gamma$ be a trajectory of $\nabla G$ starting from $z^s$. Let $w$ be the point, where $\gamma$ hits $\p U_1'$ for the first time.
If $\delta$ is sufficiently small, then $G(w)>G(z^u)$.
\end{lemma}
\begin{proof}
Assume that $\gamma(t)$ is such trajectory.  Assume that among numbers $\epsilon_i$,
we have $\epsilon_i=-1$ for $i\leqslant k-1$ and
 $\epsilon_i=1$ otherwise. As $z^s$ is a critical point of the vector field $\nabla G$ with a non-degenerate linear part,
we conclude that the limit
\[\lim_{t\to-\infty}\frac{\gamma'(t)}{||\gamma'(t)||}=:v=(x_0,y_0,u_{01},\dots,u_{0,n-1})\]
exists. The vector $v$ is the tangent vector to the curve $\gamma$ at the point $z^s$,
 and it lies in the unstable space.
Hence  $x_0=0$ as $(1,0,\dots,0)$ is a stable direction; similarly
  $u_{01}=\dots=u_{0,k-1}=0$. Therefore, until $\gamma$
hits the boundary of $U_1'$ for the first time, we have
\[
x=u_1=\dots=u_{k-1}=0.
\]
Set also $g(y)=y^3-\delta y$. One has the following  cases,
depending the position of $w$, where $\gamma$ hits $\p U_1'$ for
the first time: (a)  $y=-\e$, (b) $y=\e$, or  (c) $||\vu||^2=\eta^2$.
The case (a) cannot happen since $G$ is increasing along the
trajectory, hence $G(w)>G(z^s)$, a fact which contradicts
$g(-\e)<g(\sqrt{\delta/3})$ valid for $2\delta<\e^2$. In case (b),
$G(w)>G(z^u)$ follows from $g(\e)>g(-\sqrt{\delta/3})$.
Finally, assume the case (c).
Then, as $u_{01}=\dots=u_{0,k-1}=0$, we obtain $\vu^2=||\vu||^2=\eta^2$. Then $G(w)-G(z^s)\geqslant \eta^2$, because
the contribution to $G$ from  $y^3-\delta y$ increases along $\gamma$.
Hence $G(w)>G(z^u)$ follows again since $\e\ll\eta$.
\end{proof}
Given the above lemma it is clear that if a trajectory $\gamma$ leaves $U_1'$, then $G$ becomes bigger than $G(z^u)$. As $G$ increases along any trajectory, it
is impossible that such trajectory limits in $z^u$. The proof of Theorem~\ref{thm:handsplit},
up to Proposition~\ref{prop:localcoordinates}, is accomplished.

\begin{remark}\label{rem:metric2}
The metric $g'$ defined below Remark~\ref{rem:carefulreading2} can be chosen so that $(x,y,\vec{u})$ forms an \emph{orthogonal}, but not necessarily
\emph{orthonormal} coordinate system. Each component of the vector field \eqref{eq:gradG} is then multiplied by a positive constant, the statement of
Lemma~\ref{lem:nablaGtraj} still holds with essentially the same proof. However, $g'$ cannot be just any metric; we can choose a metric $g'$ in a way that
there is an arbitrary number of trajectories from $z^s$ to $z^u$ (topologically changing the metric can produce a pair of mutually cancelling
intersection points between the unstable manifold of $z^s$ and the stable manifold of $z^u$).
\end{remark}

\subsection{An auxiliary construction.}\label{sec:Z1Z2}
The following construction  is  a crucial ingredient in the proof of Proposition~\ref{prop:localcoordinates}; see the next section. Set
\[Z=\{(x,y)\in\R^2\colon x\geqslant 0\},\]
and define the two functions
\begin{equation}\label{eq:AB}
A(x,y)=\frac{x^3}{3\sqrt{3}}-\sqrt{3}xy^2-\frac{x}{\sqrt{3}}+\frac{2}{3\sqrt{3}},
\ \ \ \  B(x,y)=y^3-yx^2+y.
\end{equation}
Observe that
\[A+iB=\left(\frac{x}{\sqrt{3}}-iy\right)^3-\left(\frac{x}{\sqrt{3}}-iy\right)+\frac{2}{3\sqrt{3}}.\]
Up to a  linear transformation, the map $(x,y)\mapsto A+iB$ is a holomorphic map. Thus it shares several geometric properties of a holomorphic
map. For example, it is an open map, and the singular points are precisely the points where the gradient of $B$ vanishes.
%The idea is to look at $A$ and $B$ as a real and complex part of an (anti)holomorphic function, and claim that the level sets of $A$ and $B$
%are orthogonal by Cauchy--Riemann equations.
%A little technicality is that we need to change the complex structure in the codomain, so that the map $(x,y)\mapsto(\frac{x}{\sqrt{3}}+iy)$ is
%antiholomorphic. This corresponds to rescaling the metric so that $x$ and $y$ form an orthogonal but not orthonormal coordinate set,
%more precisely the length of $(1,0)$ is $1/\sqrt{3}$ while the length of $(0,1)$ is $1$. With respect to this metric,
%$A$ is constant
%on trajectories of $\nabla B$ (i.e. $A$ is a first integral of $\nabla B$).
%\begin{remark}
%This specific metric is chosen only for the sake of the proof of Proposition~\ref{prop:localcoordinates}, in the core of the
%proof of Theorem~\ref{thm:handsplit} we can choose another metric.
%\end{remark}

Let us choose $\delta>0$ smaller than $\frac{2}{3\sqrt{3}}$. Consider two sets
\begin{equation}\label{eq:Z1Z2}
Z_1=\{(x,y)\in Z,\ x<1,\ A(x,y)\ge\delta\},\ \ \ \ Z_2=\{(x,y)\in
Z,\ x>1,\ A(x,y)\ge\delta\}.
\end{equation}

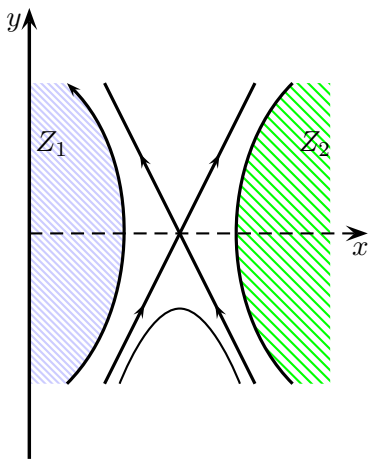
\begin{figure}
\begin{pspicture}(-10,-3)(5,3)
\pscustom[linestyle=none,hatchcolor=lightblue,hatchsep=1pt,fillstyle=vlines]{%
\psline(-5,-2)(-5,2)%
\psbezier[liftpen=1](-4.5,2)(-3.5,1)(-3.5,-1)(-4.5,-2)%
}\rput(-4.7,1.2){$Z_1$}
\pscustom[linestyle=none,hatchcolor=green,hatchsep=2pt,fillstyle=vlines]{%
\psline(-1,-2)(-1,2)%
\psbezier[liftpen=1](-1.5,2)(-2.5,1)(-2.5,-1)(-1.5,-2)%
}\rput(-1.2,1.2){$Z_2$}
\psbezier[linewidth=1.2pt]{->}(-4.5,-2)(-3.5,-1)(-3.5,1)(-4.5,2)
\psbezier[linewidth=1.2pt](-1.5,2)(-2.5,1)(-2.5,-1)(-1.5,-2)%
\psline[linewidth=1.2pt,ArrowInside=->](-4,-2)(-3,0)%
\psline[linewidth=1.2pt,ArrowInside=->](-2,-2)(-3,0)%
\psline[linewidth=1.2pt,ArrowInside=->](-3,0)(-4,2)%
\psline[linewidth=1.2pt,ArrowInside=->](-3,0)(-2,2)%
\pscurve[ArrowInside=->,ArrowInsidePos=0.3](-3.8,-2)(-3,-1)(-2.2,-2)%
\psline[linewidth=1.3pt, arrowsize=6pt]{->}(-5,-3)(-5,3)\rput(-5.2,2.8){$y$}
\psline[linewidth=0.8pt, arrowsize=6pt, linestyle=dashed]{->}(-5,0)(-0.5,0)\rput(-0.6,-0.2){$x$}
\end{pspicture}
\caption{Sets $Z_1$ and $Z_2$ from Section~\ref{sec:Z1Z2}. There is also drawn the singular level set $A^{-1}(0)$.}\label{fig:Z1Z2}
\end{figure}

We have the following result.
\begin{lemma}\label{lem:psiisiso}
The map $\psi(x,y)=(A(x,y),B(x,y))$ maps $Z_1$ and $Z_2$ diffeomorphically onto $E_1$ and $E_2$ respectively, where
\[E_1=\Big\{(a,b)\in\R^2\colon a\in\left[\delta,\frac{2}{3\sqrt{3}}\right]\Big\},
\ E_2=\Big\{(a,b)\in\R^2\colon a\ge\delta\Big\}.\]
%Moreover,
%$\psi$ maps trajectories of $\nabla B$ onto vertical lines.
\end{lemma}
\begin{proof}
One readily checks that $\psi\colon Z_1\to V_1$ and $\psi\colon Z_2\to V_2$ are bijections. As the derivative $D\psi$ is non-degenerate on $Z_1\cup Z_2$,
$\psi$ is a diffeomorphism between the two pairs of sets.
%Now $A$ is a first integral of the flow $\nabla B$, so each trajectory
%of $\nabla B$ lies in the set $A^{-1}(c)$.
\end{proof}

\subsection{Proof of Proposition~\ref{prop:localcoordinates}}\label{ss:proofhyp}
First, as $z$ is a critical point of index $k\in\{1,\dots,n\}$, by the Morse lemma~\ref{lem:boundarymorselemma} we can find a neighbourhood $\wt{V}$ of $z$ and a chart
$h_1\colon\wt{V}\to\R^{n+1}$, with coordinates $(x',y,\vu)$
such that
\[F\circ h_1^{-1}(x',y,\vu)=x'y+\vu^2+\frac12.\]
\begin{remark}
The term $x'y$ (corresponding to a hyperbolic quadratic form) is the moment when the assumption that $k\neq 0,n+1$ is used.
\end{remark}
Let us define a map $h_2(x,y,\vu)=(x',y,\vu)$, where
$x'=y^2+1-x^2$.
By the inverse function theorem, $h_2$ is a local diffeomorphism near $(1,0,\dots,0)$. Shrinking $\wt{V}$ if needed, and considering
$h_3=h_2^{-1}\circ h_1$, we obtain  $h_3(z)=(1,0,\dots,0)$ and
\begin{equation}\label{eq:Fcirch3}
F\circ h_3^{-1}(x,y,\vu)=y^3-yx^2+y+\vu^2+\frac12=B(x,y)+\vu^2+\frac12.\end{equation}
Let us pick now $\xi>0$ such that the cylinder
\[V=\{|x-1|<\xi,\ |y|<\xi,\ ||\vu||<\xi\}\subset\R^{n+1}\]
lies entirely in $h_3(\wt{V})$. By shrinking $\wt{V}$ we may in fact assume that $h_3(\wt{V})=V$.
If $0<\delta\ll \frac{2}{3\sqrt{3}}$ is sufficiently small then $A(x,0)< \delta$ implies $|x-1|<\xi$. Choose such a
$\delta$, and set
\begin{figure}%\label{pic13}
\begin{pspicture}(-6,-6)(6,6)
\pscustom[fillstyle=solid,linestyle=solid,linecolor=orange,opacity=0.2,fillcolor=orange]{%
\pscurve(-3.2,4)(-3,4.8)(-2,4.9)(-1,5)(0,4.9)(1,4.8)(1.2,4)%
\pscurve[liftpen=1](1.2,4)(1,3.2)(0,3.1)(-1,3)(-2,3.1)(-3,3.2)(-3.2,4)}\rput(-1,5.2){\psscalebox{0.8}{{$\wt{V}$}}}
\pscurve[linecolor=purple](-6,4)(-5.3,3.8)(-4.6,4.2)(-3.9,3.9)(-3.2,4)(-2,4)\rput(-5.2,4){\psscalebox{0.8}{{\purple
$\gamma$}}}
\pscurve[linewidth=1.2pt,linecolor=blue](-2.5,3.1)(-2.1,3.7)(-2,4)(-2.1,4.3)(-2.5,4.9)
\rput(-2.7,5.2){\psscalebox{0.8}{$\wt{V}_1$}}
\psline[linewidth=1.2pt](-2.1,4.3)(-2.8,4.3)(-2.8,3.7)(-2.1,3.7)%
\rput(-2.5,4.5){\psscalebox{0.8}{$\wt{D}_1$}}
\pscircle[fillcolor=blue,fillstyle=solid](-1,4){0.1}\rput(-1,4.3){\psscalebox{0.8}{$z$}}
\pscurve[linewidth=1.2pt,linecolor=blue](0.5,3.1)(0.1,3.7)(0,4)(0.1,4.3)(0.5,4.9)
\begin{psclip}{\psccurve[linestyle=none,linecolor=black](0.1,3.7)(0,4)(0.1,4.3)(0.1,4.3)(0.8,4.3)(0.8,3.7)(0.1,3.7)}
\pspolygon[linestyle=none,fillstyle=vlines,hatchcolor=gray,hatchsep=2pt,hatchangle=45](-1.1,4.3)(0.8,4.3)(0.8,3.7)(-1.1,3.7)%
\end{psclip}
\psline[linewidth=1.2pt](0.1,4.3)(0.8,4.3)(0.8,3.7)(0.1,3.7)%
\rput(0.7,5.2){\psscalebox{0.8}{$\wt{V}_2$}}
\rput(0.5,4.5){\psscalebox{0.8}{$\wt{D}_2$}}
\pscustom[fillcolor=green,opacity=0.2,linecolor=green,fillstyle=solid]{%
\pscurve(-6,4.3)(-5.3,4.1)(-4.6,4.5)(-3.9,4.2)(-3.2,4.3)(-2.1,4.3)(-2.1,4.3)(-2,4)(-2.1,3.7)%
\pscurve[liftpen=1](-2.1,3.7)(-3.2,3.7)(-3.9,3.8)(-4.6,3.9)(-5.3,3.5)(-6,3.7)%
}\rput(-4,4.7){\psscalebox{0.8}{{$\wt{W}$}}}%
\psline[linewidth=1.2pt](-6,2.5)(-6,5.5)\rput(-6,5.8){$Y$}
\begin{psclip}{\psccurve[linestyle=none,linecolor=black](-2.1,3.7)(-2,4)(-2.1,4.3)(-2.1,4.3)(-2.8,4.3)(-2.8,3.7)(-2.1,3.7)}
\pspolygon[fillstyle=vlines,hatchcolor=gray,hatchsep=2pt,hatchangle=45](-1.1,4.3)(-2.8,4.3)(-2.8,3.7)(-1.1,3.7)%
\end{psclip}
% \rput(4.5,4){\textrm{Objects on $\Omega$.}}
%%
\psline[arrowsize=6pt,linewidth=1.2pt,linestyle=dashed]{->}(-1,2.7)(-1,1.2)\rput(-0.7,2.1){\psscalebox{0.8}{$h_3$}}
\pspolygon[fillcolor=orange,fillstyle=solid,linecolor=orange,opacity=0.2](-3,1)(1,1)(1,-1)(-3,-1)\rput(-1,0.7){\psscalebox{0.8}{$V$}}%
\pscircle[fillcolor=blue,fillstyle=solid](-1,0){0.1}\rput(-1,-0.3){\psscalebox{0.7}{$(1,0,\dots,0)$}}%
\pscustom[fillcolor=green,opacity=0.2,linecolor=green,fillstyle=solid]{%
\pscurve(-6,0.3)(-5,0.3)(-4,0.3)(-3,0.3)(-2.1,0.3)(-2.1,0.3)(-2,0)%
\pscurve[liftpen=1](-2,0)(-2.1,-0.3)(-2.1,-0.3)(-3,-0.3)(-4,-0.3)(-5,-0.3)(-6,-0.3)}\rput(-4,0.5){\psscalebox{0.8}{$\Psi_1^{-1}(W)$}}
\pscustom[fillcolor=green,opacity=0.2,linecolor=green,fillstyle=solid]{%
\pscurve(3,0.3)(2,0.3)(1,0.3)(0.1,0.3)(0.1,0.3)(0,0)%
\pscurve[liftpen=1](0,0)(0.1,-0.3)(0.1,-0.3)(1,-0.3)(2,-0.3)(3,-0.3)}\rput(2.7,0.5){\psscalebox{0.8}{$\Psi_2^{-1}(h_4(D_2))=h_3'(\wt{D}_2)$}}
\pscurve[linecolor=blue,linewidth=1.2pt](-2.5,-1)(-2.1,-0.3)(-2,0)(-2.1,0.3)(-2.5,1)\rput(-2.7,1.2){\psscalebox{0.8}{$V_1$}}
\pscurve[linecolor=blue,linewidth=1.2pt](0.5,-1)(0.1,-0.3)(0,0)(0.1,0.3)(0.5,1)\rput(0.7,1.2){\psscalebox{0.8}{$V_2$}}
\psline[linestyle=dashed](-6,-1)(-6,1)\rput(-5.7,1.2){\psscalebox{0.8}{$x=0$}}
\psbezier{->}(-2.7,-1.2)(-2.7,-1.4)(-3.2,-2.7)(-3.5,-3)\rput(-2.6,-2){\psscalebox{0.8}{$\Psi_1$}}
\psbezier{->}(0.7,-1.2)(0.7,-1.4)(-0.2,-2.7)(-0.5,-3)\rput(0.6,-2){\psscalebox{0.8}{$\Psi_2$}}
\pscustom[fillcolor=orange,opacity=0.2,fillstyle=solid,linecolor=orange]{%
\pscurve(-2.75,-3.25)(-4.25,-3.25)(-4.5,-4)%
\pscurve[liftpen=1](-4.5,-4)(-4.25,-4.75)(-2.75,-4.75)}\rput(-3.7,-5.2){\psscalebox{0.8}{$C_1=\Psi_1(V_1)$}}
\pscustom[fillcolor=orange,opacity=0.2,fillstyle=solid,linecolor=orange]{%
\pscurve(-1.5,-3.25)(0,-3.25)(0.25,-4)%
\pscurve(0.25,-4)(0,-4.75)(-1.5,-4.75)}\rput(0,-5.2){\psscalebox{0.8}{$C_2=\Psi_2(V_2)$}}
\psline[linecolor=orange](-2.75,-3.25)(-2.75,-4.75)
\psline[linecolor=orange](-1.5,-3.25)(-1.5,-4.75)
\pspolygon[linewidth=0.6pt,linecolor=green,fillstyle=solid,%
fillcolor=green,opacity=0.2](-6,-4.3)(-2.75,-4.3)(-2.75,-3.7)(-6,-3.7)\rput(-4.5,-3.5){\psscalebox{0.8}{$W$}}%
\pspolygon[linewidth=1.2pt,fillstyle=vlines,hatchcolor=black,hatchsep=2pt](-2.75,-4.3)(-3.5,-4.3)(-3.5,-3.7)(-2.75,-3.7)
\rput(-3,-3.5){\psscalebox{0.8}{$D_1$}}
\psline[linewidth=0.5pt,linestyle=dashed](-2.75,-5.2)(-2.75,-3)\rput(-2.75,-2.7){\psscalebox{0.8}{$a=\delta$}}
\psline[linewidth=0.5pt,linestyle=dashed](-1.5,-5.2)(-1.5,-3)\rput(-1.5,-2.7){\psscalebox{0.8}{$a=\delta$}}
\pspolygon[linewidth=1.2pt,fillstyle=vlines,hatchcolor=black,hatchsep=2pt](-1.5,-4.3)(-0.75,-4.3)(-0.75,-3.7)(-1.5,-3.7)
\rput(-1,-3.5){\psscalebox{0.8}{$D_2$}}
\psbezier{->}(-0.7,-4)(-0.45,-3.8)(1,-3.8)(1.5,-4)\rput(0.4,-3.7){\psscalebox{0.8}{$h_4$}}
\pspolygon[linewidth=1.2pt,linecolor=black,fillcolor=gray,opacity=0.2,fillstyle=solid](1.5,-4.3)(4.5,-4.3)(4.5,-3.7)(1.5,-3.7)%
\rput(3,-4.6){\psscalebox{0.8}{$h_4(D_2)$}}
\psbezier{->}(3,-3.5)(2.5,-3)(1.7,-1)(1.7,-0.2)\rput(1.8,-2){\psscalebox{0.8}{$\Psi_2^{-1}$}}
\psline[linestyle=dashed](-6,-5)(-6,-3)\rput(-5.5,-2.7){\psscalebox{0.8}{$a=\frac{2}{3\sqrt{3}}$}}
\rput(-5.5,-2.7){\psscalebox{0.8}{$a=\frac{2}{3\sqrt{3}}$}}
\rput(-5.5,-2.7){\psscalebox{0.8}{$a=\frac{2}{3\sqrt{3}}$}}
\rput(-5.2,-4.7){\psscalebox{0.8}{$E_1$}}
\end{pspicture}
\caption{Notation used in Section~\ref{ss:proofhyp}. The top line is the picture on $\Omega$, the middle line is in coordinates such that $F$
is equal to $y^3-yx^2+y+\frac12+\vu^2$. The bottom line is in coordinates such that $F=b+\frac12$. There is no mistake, the line $a=\delta$ appears
twice on the picture, in coordinates on $C_1$ and on $C_2$.}\label{fig:proofhyp}
\end{figure}
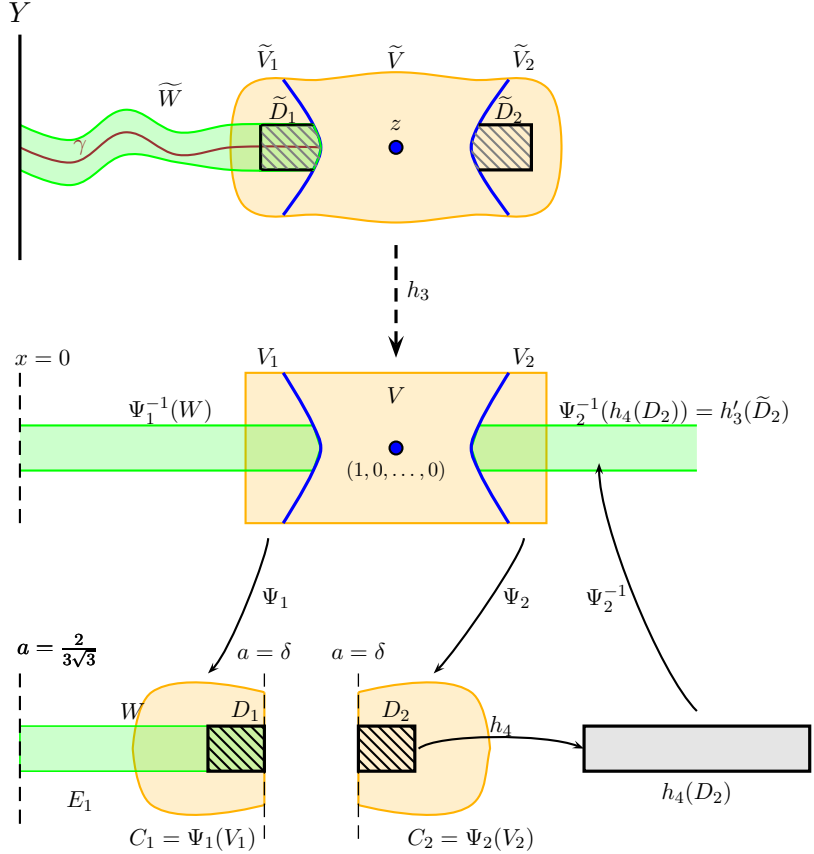
\[V_1:=V\cap\{x<1,\ A(x,y)\ge\delta\}\]
(compare \eqref{eq:Z1Z2}). By Lemma~\ref{lem:psiisiso} the map
\be\label{eq:wtPsi} \Psi_1(x,y,\vu)=(A(x,y),B(x,y)+\vu^2,\vu), \ee
is a diffeomorphism  (being the composition of $\psi\oplus Id_{\R^{n+1}}$ and a
`triangular' map). Set $C_1:=\Psi_1(V_1)$ and
$\wt{V}_1:=h_3^{-1}(V_1)$.
Finally, let
\[h=\Psi_1\circ h_3.\]
Using \eqref{eq:Fcirch3} we obtain that
\[F\circ h^{-1}(a,b,\vu)=b+\frac12.\]
Let  $\theta >\delta$ be sufficiently close to $\delta$ satisfying the inclusion
\[D_1:=[\delta,\theta]\times (-\theta,\theta)\times(-\theta,\theta)^{n-1}\subset C_1.\]
Let $\wt{D}_1=h^{-1}(D_1)\subset \wt{V}_1$, see Figure \ref{fig:proofhyp}.

\begin{lemma}\label{lem:neighofcurve}
If $\theta$ and $\delta$ are small enough, there is an closed ball $\wt{W}$ in $\Omega$,
 containing $\wt{D}_1$,
such that $h$ extends to a diffeomorphism between $\wt{W}$ and
$[\delta,\frac{2}{3\sqrt{3}}]\times [-\theta,\theta]\times[-\theta,\theta]^{n-1}$ with
$F\circ h^{-1}(a,b,\vu)=b+\frac12$, sending points with $a=\frac{2}{3\sqrt{3}}$ to $Y$.
\end{lemma}

In the proof we shall use the following result.
\begin{lemma}\label{lem:curve}
There exists a smooth curve
$\gamma\colon[\delta,\frac{2}{3\sqrt{3}}]\to\O$, such that
\begin{itemize}
\item $\gamma(\frac{2}{3\sqrt{3}})\in Y$;
\item $\gamma(t)\in\Sigma_{1/2}$;
\item $\gamma(t)\in \wt{D}_1$ if and only if $t\in[\delta,\theta]$;
\item $h(\gamma(t))=(t,0,\dots,0)$
\item $\gamma$ omits $\wt{V}\setminus\wt{V}_1$.
\item $\gamma$ is transverse to $Y$.
\end{itemize}
\end{lemma}
\begin{proof}[Proof of Lemma~\ref{lem:curve}]
Let $p=h^{-1}(\theta,0,\dots,0)\in \Sigma_{1/2}$. Let
$B\subset\Sigma_{1/2}$ be an open ball with centre $z$ and
$p\in\partial B$. Let $\S'$ be the connected component of
$\S_{1/2}$ containing $p$. We consider two cases.

\underline{Case 1.} If $\S'\setminus B$ is connected, it is also path connected. By
\eqref{eq:connectwithboundary}, there exists a path
$\wt{\gamma}\subset\S'\setminus B$ joining $p$ with a point on the
boundary. We can assume that $\wt{\gamma}$ is transverse to $Y$.
We choose
$\gamma=h^{-1}([\delta,\theta]\times\{0,\dots,0\})\cup\wt{\gamma}$
(and we smooth a possible corner at $p$). It is clear that
$\gamma$ omits $\wt{V}\setminus\wt{V}_1$ and that we can find a
parametrization of $\gamma$ by the interval
$[\delta,\frac{2}{3\sqrt{3}}]$.

\underline{Case 2.} If $\S'\setminus B$ is not connected, then as $\S'$ is connected,
by a homological argument we have $n=1$ and $k=1$. Since $\S'$ is connected
and has boundary, then $\S'$ is an interval and $B$ is an interval
too. Then $\S'\setminus B$ consists of two intervals, each
intersecting $Y$. One of these intervals contains $p$. So $p$ is
connected to $Y$ by an interval, which omits $B$. We conclude
the proof by the same argument as in the above case, when
$\S'\setminus B$ was connected.
\end{proof}

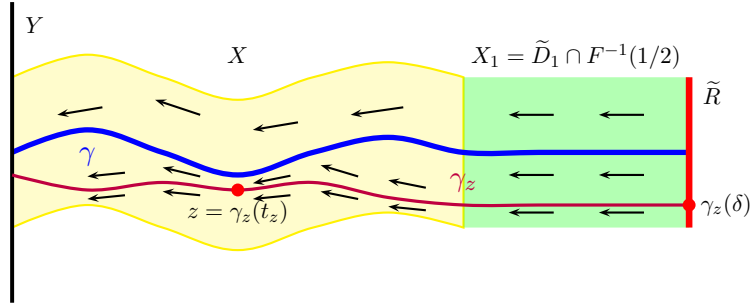
\begin{figure}
\begin{pspicture}(-5,-3)(5,3)
\pscustom[fillstyle=solid,linestyle=solid,linecolor=yellow,opacity=0.2, fillcolor=yellow]{%
\pscurve(-5,1)(-4,1.3)(-3,1)(-2,0.7)(-1,1)(0,1.2)(1,1)%
\pscurve[liftpen=1](1,-1)(0,-0.8)(-1,-1)(-2,-1.3)(-3,-1)(-4,-0.7)(-5,-1)}
\pspolygon[linestyle=none,fillstyle=solid,fillcolor=green,
opacity=0.3](1,-1)(4,-1)(4,1)(1,1)
\pscurve[linecolor=blue,linewidth=2pt](-5,0)(-4,0.3)(-3,0)(-2,-0.3)(-1,0)(0,0.2)(1,0)(2,0)(3,0)(4,0)
\rput(-4,-0.05){{\blue $\mathbf{\gamma}$}}
\rput(2.5,1.3){\psscalebox{0.8}{$X_1=\wt{D}_1\cap F^{-1}(1/2)$}}
\rput(-2,1.3){\psscalebox{0.8}{$X$}}
\psline[linecolor=black,linewidth=1.5pt](-5,-2)(-5,2)\rput(-4.7,1.7){\psscalebox{0.8}{$Y$}}
\psline[linecolor=red,linewidth=2.5pt](4,-1)(4,1)\rput(4.3,0.8){\psscalebox{0.8}{$\wt{R}$}}
\psline{<-}(1.6,0.5)(2.2,0.5) \psline{<-}(1.6,-0.3)(2.2,-0.3)
\psline{<-}(1.6,-0.8)(2.2,-0.8) \psline{<-}(2.8,0.5)(3.4,0.5)
\psline{<-}(2.8,-0.3)(3.4,-0.3) \psline{<-}(2.8,-0.8)(3.4,-0.8)
\psline{<-}(-4.4,0.5)(-3.8,0.6) \psline{<-}(-3.1,0.7)(-2.5,0.5)
\psline{<-}(-1.8,0.3)(-1.2,0.4) \psline{<-}(-0.5,0.5)(0.2,0.6)
\pscurve[linecolor=purple,linewidth=1.2pt](-5,-0.3)(-4,-0.5)(-3,-0.4)(-2,-0.5)(-1,-0.4)(0,-0.6)(1,-0.7)(2,-0.7)(3,-0.7)(4,-0.7)
\psline[linewidth=0.8pt]{->}(-3.5,-0.57)(-4,-0.62)
\psline[linewidth=0.8pt]{->}(-2.5,-0.57)(-3,-0.52)
\psline[linewidth=0.8pt]{->}(-1.3,-0.57)(-1.8,-0.62)
\psline[linewidth=0.8pt]{->}(-0.4,-0.62)(-0.9,-0.52)
\psline[linewidth=0.8pt]{->}(0.5,-0.77)(0,-0.72)
\psline[linewidth=0.8pt]{->}(-3.5,-0.27)(-4,-0.32)
\psline[linewidth=0.8pt]{->}(-2.5,-0.32)(-3,-0.20)
\psline[linewidth=0.8pt]{->}(-1.3,-0.31)(-1.8,-0.41)
\psline[linewidth=0.8pt]{->}(-0.4,-0.32)(-0.9,-0.17)
\psline[linewidth=0.8pt]{->}(0.5,-0.47)(0,-0.37)
\pscircle[linestyle=none,fillstyle=solid,fillcolor=red](-2,-0.5){0.1}
\rput(-2,-0.8){\psscalebox{0.8}{$z=\gamma_z(t_z)$}}
\rput(1,-0.4){{\purple $\mathbf{\gamma}_z$}}
\pscircle[linestyle=none,fillstyle=solid,fillcolor=red](4,-0.7){0.1}\rput(4.5,-0.7){\psscalebox{0.8}{$\gamma_z(\delta)$}}
\end{pspicture}
\caption{Proof of Lemma~\ref{lem:neighofcurve}. Construction of the vector field $\tau$. Picture on $F^{-1}(1/2)$. The parallel vector field from the region
on the right is extended to the whole $X$ so that it is tangent to $\gamma$.
}
\label{fig:neighofcurve}
\end{figure}

\begin{proof}[Proof of Lemma~\ref{lem:neighofcurve}]
Given Lemma~\ref{lem:curve}, let us choose a tubular neighbourhood $X$
of $\gamma$ in $F^{-1}(1/2)\setminus (\wt{V}\setminus \wt{V}_1)$. Shrinking $X$
if needed we can assume that it is a disk and $X_1:=X\cap\wt{V}=\wt{D}_1\cap F^{-1}(1/2)$.
Now let $\tau$ be the vector field on $\wt{D}_1$ given
by $(Dh)^{-1}(1,0,\dots,0)$, where $Dh$ denotes the derivative of $h$. This vector field is everywhere tangent to $X_1$ and
\begin{equation}\label{eq:vongamma}
\tau|_{\gamma\cap\wt{D}_1}=\frac{d}{dt}\gamma(t)
\end{equation}
by definition of $\gamma$.
We extend $\tau$ to a smooth vector field on the of whole
$X$, such that \eqref{eq:vongamma} holds on the whole of $\gamma$.
For any point $z\in\gamma$, the trajectory of $\tau$ (which is
$\gamma$) eventually hits $Y$ and, on the other end, it hits the
`right wall'
\[\wt{R}=h^{-1}(\{\delta\}\times\{0\}\times (-\theta,\theta)^{n-1}).\]
(compare Figure~\ref{fig:neighofcurve}; note that the horizontal
coordinate there increases from right to left for consistency with
Figure~\ref{fig:proofhyp}). Since $\gamma$ is transverse to
$\wt{R}$ and to $Y$, by the implicit function theorem trajectories
close to $\tau$ also start at $\wt{R}$ and end up at $Y$.
Shrinking $X$ if necessary we may assume that each point of $X$
lies on the trajectory of $\tau$ which connects a point of
$\wt{R}$ to some point of $Y$, and all the trajectories are
transverse to both $Y$ and $\wt{R}$.

We can now rescale $\tau$ (that is multiply by a suitable smooth
function constant on trajectories) so that all the trajectories go
from $\wt{R}$ to $Y$ in time $\frac{2}{3\sqrt{3}}-\delta$, i.e.
the same time as $\gamma$ does. The rescaled vector field allows
us to introduce coordinates on $X$ in the following way. For any
point $z\in X$, let $\gamma_z$ be the trajectory of $\tau$, going
through $z$.  We can assume that $\gamma_z(\delta)\in \wt{R}$. Let
$t_z=\gamma_z^{-1}(z)$, i.e. the moment when $\gamma_z$ passes
through $z$. Since we normalized $\gamma_z$, we know that
$t_z\in[\delta,\frac{2}{3\sqrt{3}}]$ and $t_z=\frac{2}{3\sqrt{3}}$
if and only if $z\in Y\cap X$.

Let $\vu_z$ be such that $h(\gamma_z(\delta))=(\delta,0,\vu_z)$. The vector $\vu_z$ might be thought of as a coordinate on $\wt{R}$. We define now
\[h(z)=(t_z,0,\vu_z).\]
This maps clearly extends $h$ to the whole of $X$.

Now let $\wt{W}$ be a tubular neighbourhood of $X$ in $\O\setminus(\wt{V}\setminus\wt{V}_1)$. We use the flow of $\nabla F$ to extend coordinates
from $X$ to $\wt{W}$. More precisely, shrinking $\wt{W}$ if needed we may assume that
for each $w\in\wt{W}$ the trajectory of $\nabla F$ intersects $X$. This intersection is necessarily transverse and it is in one point, which
we denote by $z_w\in X$. We define now
\[h(w)=(t_{z_w},F(w)-F(z_w),\vu_{z_w}).\]
As $h$ is a local diffeomorphism on $X$ (because $\nabla F$ is transverse to $X$), it is also a local diffeomorphism near $X$. We put $W=h(\wt{W})$.
Clearly both definitions of $h$ on $\wt{V}$ and $\wt{W}$ agree. We may now
decrease $\theta$ and shrink $W$ so that
\[W=[\delta,\frac{2}{3\sqrt{3}}]\times (-\theta,\theta)\times(-\theta,\theta)^{n-1}.\]
We have $F\circ h^{-1}(a,b,\vu)=b+\frac12$. We now extend $h_3$
over $\wt{W}$ by the formula $h_3=\Psi_1^{-1}\circ h$.
\end{proof}

Consider now
\[V_2:=V\cap\{x>1,\ A(x,y)\ge\delta\}.\]
Let $\Psi_2\colon V\to\R^{n+1}$ be given by
$\Psi_2(x,y,\vu)=(a,b,\vu)=(A(x,y),B(x,y)+\vu^2,\vu)$, provided by
the same formula as $\Psi_1$ in (\ref{eq:wtPsi}) but the image now
satisfies $a\geqslant \delta$, cf. Lemma \ref{lem:psiisiso}.

Let $C_2=\Psi_2(V_2)$, and let us choose $\theta'$ sufficiently
small such that
\[D_2:=[\delta,\theta']\times (-\theta',\theta')\times(-\theta',\theta')^{n-1}\subset C_2.\]
We shall denote $h=\Psi_2\circ h_3$ and $\wt{D}_2=h^{-1}(D_2)$.

Let us now fix $M>0$ large enough and consider a map
$h_4\colon\R^{n+1}\to\R^{n+1}$ of the form
\[h_4(a,b,\vu)=(\phi(a),b,\vu),\]
where $\phi\colon[\delta,\theta']\cong[\delta,M]$ is a strictly
increasing smooth function, which is an identity near $\delta$.
Consider the map $h_3'\colon \Psi_2^{-1}\circ h_4\circ
h\colon\wt{D}_2\to\R^{n+1}$. Since $h$ is an identity for $a$
close to $\delta$, this map agrees with $h_3$ for $a$ close to
$\delta$. Furthermore $F\circ h_4^{-1}(a,b,\vu)=F\circ h^{-1}\circ
h_4^{-1}(a,b,\vu)=b+\frac{1}{2}$ by a straightforward computation.
On the other hand, the point
$h^{-1}(\theta',0,\dots,0)\in\wt{D}_2$ is mapped  by $h_3'$
to $(M,0,\dots,0)\in\R^{n+1}$, where $M$ can be arbitrary large,
e.g. $M>3$.

\smallskip
Having gathered all the necessary maps, we now conclude the proof. Let
\[\wt{U}=\wt{W}\cup(\wt{V}\setminus h_3^{-1}(V_1\cup V_2)\cup\wt{D}_2.\]
The map $h_3\colon\wt{U}\to[0,\infty)\times\R^n$ is given by $h_3$
on $\wt{W}$ and  on $\wt{V}\setminus h_3^{-1}(V_2)$,  and by
$h_3'$ on $\wt{D}_2$. This map is a diffeomorphism onto
its image, so it is a chart near $z$. By construction
$F\circ h_3^{-1}$ is equal to $y^3-yx^2+y+\vu^2+1/2$ and
$h_3(\wt{W})$ contains the segment with endpoints $(0,0,\dots,0)$
and $(3,0,\dots,0)$. Since it is an open subset, it contains
$[0,3+\eta)\times(-\eta,\eta)\times(-\eta,\eta)^{n-1}$ for
$\eta>0$ small enough. The inverse image of this cube gives the
required chart.

\vspace{1mm}

This ends the proof of Theorem \ref{thm:handsplit} which moves a
single   interior critical point to the boundary. Section~\ref{S:rear}
generalizes this fact for multiple points; one of the needed
tools will be  the rearrangements of the critical values/points.

\subsection{Condition~\eqref{eq:connectwithboundary} revisited.}\label{ss:step1}
We will provide two sufficient conditions which imply Condition \eqref{eq:connectwithboundary}. One is valid for arbitrary $n\geqslant 1$,
the other one holds only in the case $n=1$.
We shall keep
the notation from previous subsections, in particular $(\O,Y)$ is a cobordism between $(\S_0,M_0)$ and $(\S_1,M_1)$,
$F\colon\O\to[0,1]$ is a Morse function with a single critical point $z$ in the interior of $\Omega$, and $F(z)=1/2$.
Let $\S_{1/2}=F^{-1}(1/2)$ and $\S'$ be the connected component of $\S_{1/2}$ such that $z\in\S'$.

\begin{proposition}\label{prop:noconnect}
If $\S_0$, $\S_1$ and $\Omega$ have no closed connected components, then $\S'\cap Y\neq\emptyset$.
In particular, in Theorem~\ref{thm:handsplit} we can assume that $\S_0,\S_1$ and $\Omega$ have no closed connected
components instead of \eqref{eq:connectwithboundary}.
\end{proposition}
\begin{proof}
Let
$p=h^{-1}(\theta,0,\dots,0)\in \wt{D}_1\subset \O$ and let $B$ be
an open ball in $\S'$ near $z$, such that $p\in\partial B$. It is enough to
show that $p$ can be connected to $Y$ by a path in $\Sigma_{1/2}$,
which misses $B$ (compare Lemma~\ref{lem:curve}).

Let us choose a Riemannian metric on $\O$. Let $W^s_z$ be the
stable manifold of $z$ and let $T$ be the intersection of $W^s_z$ and
$\S_0$. This is a $(k-1)$--dimensional sphere. The flow of $\nabla
F$ induces a diffeomorphism $\Phi\colon \S_{1/2} \setminus B
\cong\S_0\setminus B_0$, where $B_0$ is a tubular neighbourhood of
$T$ in $\S_0$ (here we tacitly use the fact that $\delta$ and
$\theta$ are small enough), see Figure~\ref{fig:crit2}. Let
$p_0=\Phi(p)$. Let $\S_0'$ be the connected component of $\S_0$
which contains $B_0$.

Now we will analyze several cases. Recall that
$k=\ind_zF\in\{1,\ldots, n\}$. First we assume that $k<n$. Then
$\S'_0\setminus T$ is connected, so $p_0$ can be connected to the
boundary of $\S'_0$ --- which is non-empty by the assumptions of
the proposition --- by a path $\gamma_0$. Now the
inverse image $\Phi^{-1}(\gamma_0)$ is the required path.

\begin{figure}
\begin{pspicture}(-5,-3)(5,0)
\pspolygon[fillcolor=lightgreen,fillstyle=solid](-5.5,-3)(4,-3)(5,0)(-4,0)
\psellipse[fillcolor=blue,opacity=0.3,linestyle=none,fillstyle=solid](0.5,-1.5)(2.5,1.2)
\psellipse[linewidth=2pt](0.5,-1.5)(2.3,1.1)
\psellipse[fillcolor=lightgreen,fillstyle=solid,linestyle=none](0.5,-1.5)(2.1,1.0)
\rput(2.3,-0.5){\psscalebox{0.8}{$T$}}
\rput(-2.1,-1.3){\psscalebox{0.8}{$p_0$}}
\rput(-3.5,-1.2){\psscalebox{0.8}{$\gamma_0$}}
\pscurve(-4.75,-1.5)(-3.5,-1.5)(-2.8,-2.5)(-2,-1.5)
\pscircle[fillcolor=black,fillstyle=solid](-4.75,-1.5){0.05}
\pscircle[fillcolor=black,fillstyle=solid](-2,-1.5){0.05}
\end{pspicture}
\caption{Notation on $\Sigma_0$.}\label{fig:crit2}
\end{figure}
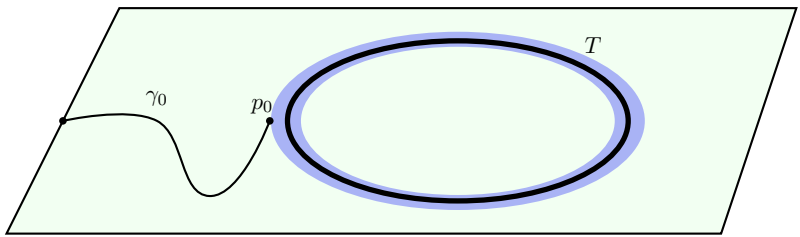

If $k=n>1$ then  we reverse the cobordism and look at $-F$, hence
this case is covered by the previous one (since $k=n$ will be
replaced by $k=1<n$).

Finally, it remains to deal with the situation $k=n=1$. Then
$\dim\S_0=1$. $T$ consists of two points. Assume first that they
lie in a single connected component $\S_0'$ of $\S_0$. We shall
show that this is impossible. As $\S_0'$ is connected with
non-trivial boundary, it is an interval. The situation is like on
Figure~\ref{fig:crit3}.
Now as $F$ has precisely one Morse critical point of index $1$,
$\S_1$ is the result of a surgery on $\S_0$. This surgery consists
of removing two inner segments from $\S_0$ and gluing back two
other segments, which in Figure~\ref{fig:crit3} are drawn as
dashed arc. But then $\S_1$ has a closed connected component,
which contradicts assumptions of Theorem~\ref{thm:handsplit}.

Therefore, $T$ lies in two connected components of $\S_0$. The situation is drawn in Figure~\ref{fig:crit4}, and it is straightforward to
see that $p_0$ (either $p_0'$ or $p_0''$ in Figure~\ref{fig:crit4}) can be connected to $M_0$ by a segment omitting $B_0$.
\end{proof}

\begin{figure}
\begin{pspicture}(-6,-1)(6,4)
\psarc[linestyle=dashed,linecolor=darkgreen](0,0){3.6}{0}{180}
\psarc[linestyle=dashed,linecolor=darkgreen](0,0){1.2}{0}{180}
\psline[linewidth=1.4pt](-6,0)(6,0)
\pscircle[fillstyle=solid,fillcolor=black](6,0){0.1}\rput(6,0.3){\psscalebox{0.8}{$M_0$}}
\pscircle[fillstyle=solid,fillcolor=black](-6,0){0.1}\rput(-6,0.3){\psscalebox{0.8}{$M_0$}}
\psline[linewidth=2.2pt,linecolor=blue](-3.6,0)(-1.2,0)\rput(-3.1,0.3){\psscalebox{0.8}{$B_0$}}
\psline[linewidth=2.2pt,linecolor=blue](3.6,0)(1.2,0)\rput(-3.1,0.3){\psscalebox{0.8}{$B_0$}}
\pscircle[fillstyle=solid,fillcolor=blue,linestyle=none](-2.4,0){0.1}\rput(-2.4,0.3){\psscalebox{0.8}{$T$}}
\pscircle[fillstyle=solid,fillcolor=blue,linestyle=none](2.4,0){0.1}\rput(2.4,0.3){\psscalebox{0.8}{$T$}}
\pscircle[fillstyle=solid,fillcolor=red,linestyle=none](3.6,0){0.1}\rput(3.6,-0.3){\psscalebox{0.8}{$p'_0$}}
\pscircle[fillstyle=solid,fillcolor=red,linestyle=none](1.2,0){0.1}\rput(1.2,-0.3){\psscalebox{0.8}{$p''_0$}}

\end{pspicture}
\caption{Proof of Proposition~\ref{prop:noconnect}. Case $k=1$ and $n=1$ and $T$ lies in two components of $\S$.
$\S_0$ is the horizontal segment. The points $p_0'$ and $p''_0$ are the two possible positions of the point $p_0$.}\label{fig:crit3}
\end{figure}
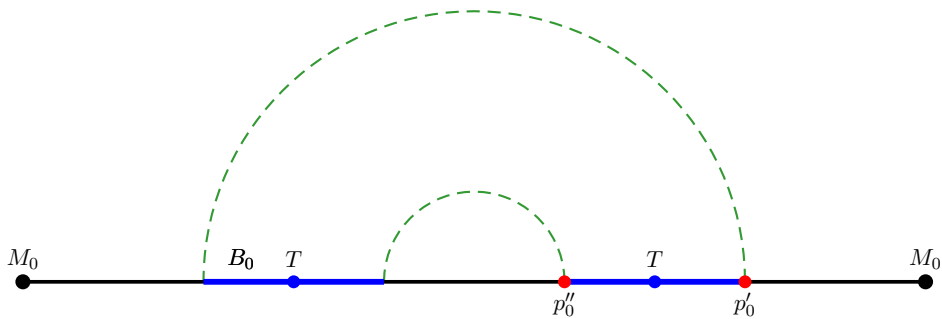
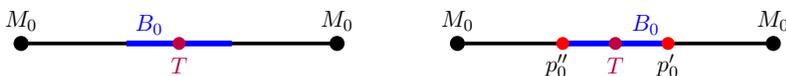
\begin{figure}
\begin{pspicture}(-6,-1)(6,2)
\psline[linewidth=1.4pt](-5,0)(-0.8,0)
\psline[linewidth=1.4pt](5,0)(0.8,0)
\pscircle[fillstyle=solid,fillcolor=black](5,0){0.1}\rput(5,0.3){\psscalebox{0.8}{$M_0$}}
\pscircle[fillstyle=solid,fillcolor=black](-5,0){0.1}\rput(-5,0.3){\psscalebox{0.8}{$M_0$}}
\pscircle[fillstyle=solid,fillcolor=black](0.8,0){0.1}\rput(0.8,0.3){\psscalebox{0.8}{$M_0$}}
\pscircle[fillstyle=solid,fillcolor=black](-0.8,0){0.1}\rput(-0.8,0.3){\psscalebox{0.8}{$M_0$}}
\psline[linewidth=2.2pt,linecolor=blue](-2.2,0)(-3.6,0)\rput(-3.3,0.25){\psscalebox{0.8}{{\blue $B_0$}}}
\psline[linewidth=2.2pt,linecolor=blue](2.2,0)(3.6,0)\rput(3.3,0.25){\psscalebox{0.8}{{\blue $B_0$}}}
\pscircle[fillstyle=solid,fillcolor=purple,linestyle=none](2.9,0){0.1}\rput(2.9,-0.3){\psscalebox{0.8}{{\purple $T$}}}
\pscircle[fillstyle=solid,fillcolor=purple,linestyle=none](-2.9,0){0.1}\rput(-2.9,-0.3){\psscalebox{0.8}{{\purple $T$}}}
\pscircle[fillstyle=solid,fillcolor=red,linestyle=none](3.6,0){0.1}\rput(3.6,-0.3){\psscalebox{0.8}{$p'_0$}}
\pscircle[fillstyle=solid,fillcolor=red,linestyle=none](2.2,0){0.1}\rput(2.12,-0.3){\psscalebox{0.8}{$p''_0$}}
%\psarc[fillstyle=dashed,linecolor=green](0,0){0.4}{0}{180}
%\psarc[fillstyle=dashed,linecolor=green](0,0){1.2}{0}{180}
\end{pspicture}
\caption{Proof of Proposition~\ref{prop:noconnect}. Case $k=1$ and $n=1$ and
$T$ lies in two components of $\S_0$. The points $p_0'$ and
$p''_0$ are the two possible positions of the point $p_0$. Both
can be connected to the boundary $M_0$.}\label{fig:crit4}
\end{figure}

The proof of Proposition~\ref{prop:noconnect} suggests that the case $n=1$ is different from case $n>1$. We shall provide now a full characterization
of the failure to \eqref{eq:connectwithboundary}.

\begin{proposition}\label{prop:boundarysurface}
Assume that $k=n=1$ and $\Omega$ is connected. If
\eqref{eq:connectwithboundary} does not hold, then $\O$ is a pair
of pants, $\S_0$ is a circle and $\S_1$ is a disjoint union of two
circles; or vice versa:  $\S_1=S^1$ and $\S_0$ is a disjoint union
of two circles. In particular, $Y=\emptyset$.
\end{proposition}
\begin{proof}
A one-handle attached to a surface changes the number of boundary
components by $\pm 1$. Let us assume that $\S_1$ has fewer
components than $\S_0$, if not we can reverse the cobordism. As
$\O$ is connected, $\S_0$ has two components
 and $\S_1$ only one. Let $A_0\subset\S_0$
be the attaching region, i.e. the union of two closed intervals to
which the one-handle is attached. With the notation of
Section~\ref{ss:step1} we  have $(\Sigma',z)\cong(\S_0/A_0,A_0/A_0)$, where the quotient
denotes collapsing a space to a point.  In particular $z$ cannot be joined to $Y$
by a path in $\S'$ if and only if $\S_0$ is disjoint from $Y$.
Hence $\S_0$ is closed, that is,  it is a union of two circles.
\end{proof}

\section{Rearrangements of boundary handles}\label{S:rear}

\subsection{Preliminaries}

Let $(\O,Y)$ be a cobordism between two $n$-di\-men\-sio\-nal
manifolds with boundary $(\Sigma_0,M_0)$ and $(\Sigma_1,M_1)$. Let
$F$ be a Morse function, with critical points
$w_1,\dots,w_k\in \Int \O$ and $y_1,\dots,y_l\in Y$. In the
classical theory (that is, when $Y=\emptyset$), the
Thom--Milnor--Smale theorem (see \cite[Section 4]{Mi-hcob}) says
that we can alter $F$ without introducing new critical  points
such that if $\ind w_i<\ind w_j$, then $F(w_i)<F(w_j)$ as well. We
want to prove similar results in our more general case.

In this section we rely very strongly on \cite[Section 4]{Mi-hcob}.

\subsection{Elementary rearrangement theorems}\label{sub:ert}
We shall begin with the case $k+l=2$, i.e. $F$ has two critical
points. For a critical point $p$ we shall denote by $K_p$ the
union $W^s_p\cup \{p\}\cup W^u_p$, i.e. the set of all points
$x\in \O$, such that the trajectory $\phi_t(x)$ ($t\in\R$), of the
gradient vector field $\nabla F$ contains $p$ in its limit set.
Elementary rearrangement theorems deal with the case when the two
sets $K_{p_1}$ and $K_{p_2}$ for the two critical points are
disjoint.

\begin{proposition}[Rearrangement of critical points]\label{prop:rearii}
Let $p_1$ and $p_2$ be two critical points, and assume that
$K_1:=K_{p_1}$ and $K_2:=K_{p_2}$ are disjoint. Let us choose
$a_1,a_2\in(0,1)$. Then, there exist a Morse function $G\colon
\O\to[0,1]$, with critical points exactly at $p_1$ and $p_2$, such
that $G(p_i)=a_i$, $i=1,2$; furthermore, near $p_1$ and $p_2$, the
difference $F-G$ is a locally constant function.
\end{proposition}
\begin{remark}
If both critical points are on the boundary, in order to guarantee
the above existence, we need even to change the Riemannian metric
away from $K_1$ and $K_2$.
\end{remark}
\begin{proof}
Similarly to \cite[Section 4]{Mi-hcob} we  will use  an
auxiliary result. Its proof is postponed after the end of proof of Proposition~\ref{prop:rearii}.

\begin{lemma}\label{lem:improvedMilnor}
There exists a smooth function $\mu\colon \O\to[0,1]$ with the following properties:
\begin{itemize}
\item[(M1)] $\mu\equiv 0$ in a neighbourhood of $K_1$;
\item[(M2)] $\mu\equiv 1$ in a neighbourhood of $K_2$;
\item[(M3)] $\mu$ is constant on trajectories of $\nabla F$.
\end{itemize}
Furthermore, if at least one of the critical points is interior, we have
\begin{itemize}
\item[(M4)] $\mu$ is constant on $Y$.
\end{itemize}
\end{lemma}

We continue with the proof of Proposition~\ref{prop:rearii}. We choose a smooth function
$\Psi\colon[0,1]\times[0,1]\to[0,1]$ with
\begin{itemize}
\item[(PS1)] $\frac{\p\Psi}{\p x}(x,y)>0$ for all $(x,y)\in[0,1]\times[0,1]$;
\item[(PS2)] there exists $\delta>0$ such that $\Psi(x,y)=x$ for all $x\in[0,\delta]\cup[1-\delta,1]$ and $y\in [0,1]$;
\item[(PS3)] for any $s\in(-\delta,\delta)$ we have $\Psi(F(p_1)+s,0)=a_1+s$ and
$\Psi(F(p_2)+s,1)=a_2+s$.
\end{itemize}

For any $\eta\in\Omega$ we define
$G(\eta)=\Psi(F(\eta),\mu(\eta))$. From the properties (PS3), (M1)
and (M2) we see that near $p_i$, $G$ differs from $F$ by a
constant. The property (PS2) ensures that $G$ agrees with $F$ in a
neighbourhood of $\S_0$ and $\S_1$. Let us show that $\nabla G$
does not vanish away from $p_i$. By the chain rule we have
\begin{equation}\label{eq:chainrule}
\nabla G=\frac{\p\Psi}{\p x}\nabla F+\frac{\p\Psi}{\p y}\nabla\mu.
\end{equation}
Since $\mu$ is constant on all trajectories of $\nabla F$, the
scalar product $\langle\nabla F,\nabla\mu\rangle=0$. Then the
property (PS1) guarantees  that $\langle\nabla G,\nabla
F\rangle>0$ away from $p_1$ and $p_2$.

We need to show that $\nabla G$ is everywhere tangent to $Y$. If
one of the points is interior, by (M4) $\nabla\mu$ vanishes on
$Y$, hence $\nabla G$ is parallel on $Y$ to $\nabla F$ and we are
done. Next assume that  both critical points are on the boundary.
Let us choose an open subset $U$ of $Y$ such that
$\nabla\mu|_{U}=0$ and $K_1\cup K_2\subset U$. This is possible,
because of the properties (M1) and (M2). Then let us choose a
neighbourhood $W$ in $\O$ of $Y\setminus U$, disjoint from $K_1$
and $K_2$. Observe that $dG(\nabla F)=\langle \nabla G,\nabla
F\rangle>0$. As $\nabla F\in TY$ one has $TY\not\subset\ker dG$,
so by Lemma~\ref{lem:altermet} we can change the metric in $W$ so
that $\nabla G$ is everywhere tangent to $Y$.
\end{proof}

\begin{proof}[Proof of Lemma~\ref{lem:improvedMilnor}]
Let us define $T_1=K_1\cap\Sigma_0$ and $T_2=K_2\cap\Sigma_0$. Assume that $T_1$ and $T_2$ are not empty.
For each $\eta\in \O\setminus K_{1}\cup K_{2}$, let $\pi(\eta)$ be the intersection of
the trajectory of $\eta$ under $\nabla F$ with $\S_0$. This gives a map
$\pi\colon \O\setminus (K_1\cup K_2)\to \S_0\setminus (T_1\cup T_2)$.

Let us define $\mu$ first on $\S_0$ by the following conditions:
$\mu\equiv 1$ in a neighbourhood of $T_2$, $\mu\equiv 0$ in a
neighbourhood of $T_1$. Furthermore, if either $T_1$ or $T_2$ is
disjoint from the boundary $M_0$ we extend $\mu$ to a constant
function on $M_0$.
Finally,  we extend $\mu$ to the whole $\O$ by picking
$\mu(\eta)=\mu(\pi(\eta))$ if $\eta\not\in K_{1}\cup K_{2}$, and
$\mu|_{K_i}(\eta)=i-1$, $i=1,2$.

If $T_1=\emptyset$, then $\ind_F p_1=0$ and the proof of the rearrangement theorem is completely straightforward.
\end{proof}

\subsection{Morse--Smale condition on manifolds with boundary}\label{ss:msc}

In the classical theory, the Morse--Smale condition imposed on a
Morse function $F\colon M\to\R$ means that for each pair of two
critical points $p_1,p_2$ of $M$ the intersection of stable
manifold $W^s_{p_1}$ with the unstable manifold of $W^u_{p_2}$ is
transverse. (Note that this Morse--Smale condition also depends on
the choice of Riemannian metric on $M$.) Following
\cite[Definition~2.4.2]{KM}, we reformulate the Morse--Smale
condition in the following way

\begin{definition}\label{def:KMMS}
The function $F$ is called \emph{Morse--Smale} if for any two
critical points $p_1$ and $p_2$, the intersection of $\Int\O\cap
W^s_{p_1}$ with $\Int\O\cap W^u_{p_2}$ is transverse (as the
intersection in the $(n+1)$-dimensional manifold $\O$) and the
intersection of $Y\cap W^s_{p_1}$ with $Y\cap W^u_{p_2}$ is
transverse (as an intersection in the $n$-dimensional manifold
$Y$).
\end{definition}

The Morse--Smale functions form an open-dense subset of all $C^2$ smooth functions
satisfying the condition~\eqref{eq:kerdfnot}. The proof is the same as in the case
of the Morse functions on manifolds without boundary; see for example \cite[Theorem 2.27]{Nic}.

Assume now that $F$ is Morse--Smale. Given two critical points of
$F$, $p_1$ and $p_2$, we want to check whether $W^s_{p_1}\cap
W^u_{p_2}=\emptyset$. This depends not only on the indices, but also
on whether either of the two points are boundary stable. We show this
in a tabulated form in Table~\ref{tab:morsesmale}, where $\ind
p_1=k$ and $\ind p_2=l$. In studying the intersection, we remark
that $W^s_{p_1}\cap W^u_{p_2}$ is formed from trajectories, so if
for dimensional reasons we have $\dim W^s_{p_1}\cap W^u_{p_2}<1$,
it immediately follows that this intersection is
empty.

\begin{table}[t]
\begin{tabular}{|c|c|c|c|c|c|c|}\hline
type of $p_1$ & type of $p_2$ & $\dim_\O W^s_{p_1}$ & $\dim_Y
W^s_{p_1}$ & $\dim_\O W^u_{p_2}$ & $\dim_Y W^u_{p_2}$ & empty if
\\ \hline
interior & interior & $k$ & $\emptyset$ & $n+1-l$ & $\emptyset$ & $k\leqslant l$ \\
interior & b. stable & $k$ & $\emptyset$ & $\emptyset$ & $n+1-l$ & always \\
interior & b. unstable & $k$ & $\emptyset$ & $n+1-l$ & $n-l$ & $k\leqslant l$ \\ \hline
b. stable & interior & $k$ & $k-1$ & $n+1-l$ & $\emptyset$ & $k\leqslant l$ \\
b. stable & b. stable & $k$ & $k-1$ & $\emptyset$ & $n-l$ & $k\leqslant l$ \\
b. stable & b. unstable & $k$ & $k-1$ & $n+1-l$ & $n-l$ & $k\leqslant l$ \\ \hline
b. unstable & interior & $\emptyset$ & $k$ & $n+1-l$ & $\emptyset$ & always\\
b. unstable & b. stable & $\emptyset$ & $k$ & $\emptyset$ & $n+1-l$ & $\mathbf{k<l}$ \\
b. unstable & b. unstable & $\emptyset$ & $k$ & $n+1-l$ & $n-l$ & $k\leqslant l$.\\ \hline
\end{tabular}
\smallskip
\caption{Under the Morse--Smale condition, the last column shows,
whether there might exist trajectories from $z_2$ to $z_1$.
We write $\dim_\O W^s=\dim (W^s\cap \Int\O)$ and $\dim_YW^s=\dim (W^s\cap Y)$.
$\emptyset$ means that the manifold in question is empty.}
\label{tab:morsesmale}
\end{table}

\subsection{Global rearrangement theorem}\label{ss:grt}

Let us combine the rearrangement theorems from Section~\ref{sub:ert} with the computations in Table~\ref{tab:morsesmale}.

\begin{proposition}\label{prop:grt}
Let $F$ be a Morse function on  a cobordism $(\O,Y)$ between
$(\S_0,M_0)$ and $(\S_1,M_1)$. Let $w_1,\dots,w_m$ be the interior
critical points of $F$ and let $y_1,\dots,y_{k+l}$ be the boundary
critical points, where $y_1,\dots,y_k$ are boundary stable and
$y_{k+1},\dots,y_{k+l}$ are boundary unstable. Let us choose real
numbers $0<c_0<c_1<\dots<c_{n+1}<1$, $0<c_1^s<\dots<c_{n+1}^s<1$,
$0<c_0^u<\dots<c_n^u<1$ satisfying
\begin{equation}\label{eq:ineq}
\begin{split}
c_{i-1}^s<c_i<c_{i+1}^s\\
c_{i-1}^u<c_i<c_{i+1}^u\\
c_{i-1}^u<c_i^s<c_i^u
\end{split}
\end{equation}
for all $i\in\{0,\dots,n+1\}$ (we can assume that $c_{-1}=c_0^s=c_{-1}^u=0$,
$c_{n+2}=c_{n+2}^s=c_{n+1}^u=1$ so that \eqref{eq:ineq} makes sense for
all $i$).

Then there exists another Morse function $G$ on the  cobordism
$(\O,Y)$ with critical points $w_1,\dots,w_m$ in the interior,
$y_1,\dots,y_{k+l}$ on the boundary, such that if $\ind w_j=l$,
then $G(w_j)=c_l$, if $\ind y_i=l$ and $y_i$ is boundary stable
then $G(y_i)=c_l^s$, and if $\ind y_i=l$ and $y_i$ is boundary
unstable, then $G(y_i)=c_l^u$.
\end{proposition}
\begin{proof}
Given the elementary rearrangement result
(Proposition~\ref{prop:rearii}), the proof is completely standard
(see the proof of Theorem~4.8 in \cite{Mi-hcob}). Note only that
we need to have $c_i^s<c_i^u$ in the statement, because there
might be a trajectory from a boundary stable critical point to a
boundary unstable of the same index. However, we are free to
choose $c_i<c_i^s$ or $c_i\in (c_i^s,c_i^u)$ or $c_i>c_i^u$.
\end{proof}

\subsection{Moving more handles to the boundary at once}\label{ss:fhsproof}

Before we formulate Theorem~\ref{thm:fullhandsplit}, let us
introduce the following technical notion.

\begin{definition}\label{def:techgood}
The Morse function $F$ on the cobordism $(\O,Y)$ is called
\emph{technically good} if it has the following properties.
\begin{itemize}
\item[(TG1)] If $p_1$, $p_2$ are (interior or boundary)
critical points of $F$ then $\ind p_1<\ind p_2$ implies
$F(p_1)<F(p_2)$;
\item[(TG2)] There exist regular values of $F$, say $c,d\in[0,1]$, with $c<d$ such that $F^{-1}[0,c]$ contains those and only those critical points
which have index $0$ or which are boundary stable critical points of index $1$, and $F^{-1}[d,1]$ contains those and only those critical points
which have index $n+1$ and boundary unstable critical points of index $n$.
\item[(TG3)] There are no pairs of $0$ and $1$ (interior) handles of $F$ that can be cancelled (in the sense of Section~\ref{S:bhc});
\item[(TG4)] There are no pairs of $n$ and $n+1$ (interior) handles of $F$ that can be cancelled.
\end{itemize}
\end{definition}

\begin{lemma}\label{lem:techgood}
Each function $F$ can be made technically good without introducing
new critical points.
\end{lemma}

\begin{proof}
By Proposition~\ref{prop:grt} we can rearrange the critical
points of $F$, proving (TG1) and (TG2). The properties (TG3) and (TG4) can
be guaranteed, using the handle cancellation theorem (e.g. \cite[Theorem 5.4]{Mi-hcob}.
We refer to the beginning of Section~\ref{S:bhc} for an explanation that one
can use the handle cancellation theorem if the manifold in question has boundary.
\end{proof}

\begin{theorem}\label{thm:fullhandsplit}
Let $(\O,Y)$ be a cobordism between $(\S_0,M_0)$ and $(\S_1,M_1)$.
Let $F$ be a technically good Morse function on that
cobordism, which has critical points $y_1,\dots,y_k$ on the
boundary $Y$, $z_1,\dots,z_{l+m}$ in the interior $\Int\O$, of
which $z_{m+1},\dots,z_{l+m}$ have index $0$ or $n+1$ and the
indices of $z_1,\dots,z_m$ are in $\{1,\dots,n\}$. Suppose
furthermore the following properties are satisfied:
\begin{itemize}
\item[(I1)] $\S_0$ and $\S_1$ have no closed connected components;
\item[(I2)] $\O$ has no closed connected component.
\end{itemize}
Then there exists a Morse function $G\colon\O\to[0,1]$, on the
cobordism $(\O,Y)$, with critical points $y_1,\dots,y_k\in Y$,
$z_{m+1},\dots,z_{l+m}$ and $z_1^s,z_1^u,\dots,z_m^s,z_m^u$ such
that:
\begin{itemize}
\item[$\bullet$] $\ind_G y_i=\ind_F y_i$ for $i=1,\dots,k$ and for $j=m+1,\dots,m+l$ we have $\ind_G z_j=\ind_F z_j$;
\item[$\bullet$] for $j=1,\dots,m$, $\ind_G z_j^s=\ind_G z_j^u=\ind_F z_j$;
\item[$\bullet$] for $j=1,\dots,m$, $z_j^s$ and $z_j^u$ are on the boundary $Y$, furthermore $z_j^s$ is boundary stable, $z_j^u$ is boundary unstable
and $G(z_j^s)<G(z_j^u)$.
\end{itemize}
\end{theorem}
In other words, we can move all critical points to the boundary at
once. To prove Theorem~\ref{thm:fullhandsplit} we use
Theorem~\ref{thm:handsplit} independently for each critical point $z_1,\dots,z_m$. We need to
ensure that Condition~\ref{eq:connectwithboundary} holds. This is done in
Proposition~\ref{prop:mtl} stated below. Given these two ingredients the proof is straightforward.

\subsection{Topological ingredients needed in  the proof of Theorem~\ref{thm:fullhandsplit}}\label{ss:mtl}

\begin{proposition}\label{prop:mtl}
Let $F$ be a technically good Morse function on the cobordism
$(\O,Y)$. Assume that $\S_0$, $\S_1$ and $\O$ have no closed connected components.
Let $c,d$ be as in Definition~\ref{def:techgood}. Then

\begin{itemize}
\item[(a)] If $n>1$, then for any $y\in [c,d]$, the inverse image
$F^{-1}(y)$ has no closed connected component.
\item[(b)] If $n=1$, then after possibly rearranging the critical values of the interior critical points of index $1$, for any interior
critical point $z\in\O$ of $F$ of index $1$, $z$ can be connected with $Y$ by a curve lying entirely in $F^{-1}(F(z))$; and furthermore
all the critical points are on different levels.
\end{itemize}
\end{proposition}
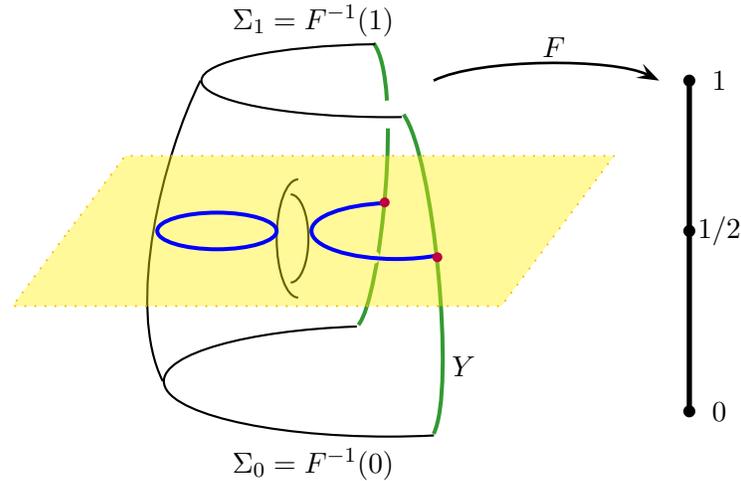
\begin{figure}
\begin{pspicture}(-5,-3.5)(5,3.5)
%\psgrid[gridcolor=blue](-5,-4)(5,4)
\psbezier[linecolor=darkgreen,linewidth=1.5pt](0.8,2.5)(1.2,2.0)(0.9,-0.9)(0.6,-1.25)
\pscircle[fillstyle=solid,fillcolor=white,linestyle=none](1,1.55){0.2}
\pscircle[fillstyle=solid,fillcolor=white,linestyle=none](0.8,-0.35){0.1}
\psellipticarc(1,-2)(3,0.75){120}{310}
\psellipticarc(1,2)(2.5,0.5){110}{290}
\psellipticarc(-0.2,-0.1)(0.3,0.8){90}{270}
\psellipticarc(-0.3,-0.1)(0.25,0.6){270}{90}
\psbezier(-1.5,2)(-2,1)(-2.5,-1)(-2,-2)
\psbezier[linecolor=darkgreen,linewidth=1.5pt](1.2,1.55)(1.6,1)(1.9,-2.1)(1.6,-2.7)
\pspolygon[linestyle=dotted,linecolor=orange,fillstyle=solid,fillcolor=yellow,opacity=0.4](-2.5,1)(4,1)(2.5,-1)(-4,-1)
\psellipse[linecolor=blue,linewidth=1.5pt](-1.28,0)(0.82,0.27)
\psellipticarc[linecolor=blue,linewidth=1.5pt](1.1,0)(1.15,0.4){110}{330}
\rput(0,-3.1){$\S_0=F^{-1}(0)$}
\rput(0,2.8){$\S_1=F^{-1}(1)$}
\psline[linewidth=2pt](5,-2.4)(5,2)\rput(5.4,2){$1$}\rput(5.4,-2.4){$0$}\rput(5.4,0){$1/2$}
\rput(2,-1.8){$Y$}
\psbezier[linewidth=1pt,arrowsize=6pt]{->}(1.6,2)(2.2,2.3)(4.0,2.3)(4.6,2)
\pscircle[fillcolor=black,fillstyle=solid](5,-2.4){0.08}
\pscircle[fillcolor=black,fillstyle=solid](5,2){0.08}
\pscircle[fillcolor=black,fillstyle=solid](5,0){0.08}
\pscircle[fillcolor=purple,fillstyle=solid,linestyle=none](0.95,0.38){0.08}
\pscircle[fillcolor=purple,fillstyle=solid,linestyle=none](1.65,-0.35){0.08}
\rput(3.2,2.45){$F$}
\end{pspicture}
\caption{The statement of Proposition~\ref{prop:mtl}(a) does not hold if $n=1$. Here, $F$ is the height function. The level set $F^{-1}(1/2)$,
drawn on the picture, has two connected components, one of which is closed.}\label{fig:puncturedtorus}
\end{figure}
\begin{remark}
The distinction between the cases $n>1$ and $n=1$ is necessary. The conclusion of point (a) of Proposition~\ref{prop:mtl} is not necessarily valid if $n=1$,
see Figure~\ref{fig:puncturedtorus} for a simple counterexample.
\end{remark}

First let us prove several lemmas, which are simple consequences of the assumptions of Proposition~\ref{prop:mtl}. We use
assumptions and notation of Proposition~\ref{prop:mtl}.

\begin{lemma}\label{lem:mtl-lemma1}
Let $x,y\in[0,1]$ with $x<y$. If $\O'$ is a connected component of $F^{-1}[x,y]$ then either $\O'\cap Y=\emptyset$, or for any $u\in[x,y]\cap[c,d]$
we have $F^{-1}(u)\cap\O'\cap Y\neq\emptyset$.
\end{lemma}
\begin{proof}
Assume that for some $u\in[x,y]\cap[c,d]$ the intersection $F^{-1}(u)\cap\O'\cap Y=\emptyset$ and $\O'\cap Y\neq\emptyset$. Then either $\O'\cap Y\cap F^{-1}[0,u]$
or $\O'\cap Y\cap F^{-1}[u,1]$ is not empty. Assume the first possibility (the other one is symmetric) and let $Y'=\O'\cap Y\cap F^{-1}[0,u]$. Let $f=F|_{Y'}$
be the restriction. Then $Y'$ is compact and $f$ has a local maximum on $Y'$. This maximum corresponds to a critical point of $f$ of index $n$,
so either a boundary stable critical point of $F$ of index $n+1$, or a boundary unstable critical point of index $n$. But the corresponding
critical value is smaller
than $u$, so smaller than $d$, which contradicts property (TG2).
\end{proof}

\begin{lemma}\label{lem:mtl-lemma2}
For any $x\in[c,1]$ the set $F^{-1}[0,x]$ cannot have a connected component disjoint from $Y$.
\end{lemma}
\begin{proof}
Assume the contrary, and let $\O'$ be a connected component of $F^{-1}[0,x]$ disjoint from $Y$. Let $\O_1$ be the connected
component of $\O$ containing $\O'$. Suppose that $\O_1\cap Y=\emptyset$. Then either $\O_1\cap(\S_0\cup\S_1)=\emptyset$
or $\O_1\cap(\S_0\cup\S_1)\neq\emptyset$. In the first case $\O_1$ is a closed connected component of $\O$, in the second either $\O_1\cap\S_0$,
or $\O_1\cap\S_1$ is not empty, so either $\S_0$ or $\S_1$ has a closed connected component. The contradiction implies that $\O_1\cap Y\neq\emptyset$.

By Lemma~\ref{lem:mtl-lemma1} we have then $F^{-1}(x)\cap\O_1\cap Y\neq\emptyset$, hence $\O'':=(F^{-1}[0,x]\cap\O_1)\setminus\O'$ is not empty and
is disjoint from $\O'$. As $\O'$ and $\O''$ both belong to $\O_1$ which is connected, there must be a critical point $z\in\O_1$
of index $1$, which joins $\O'$ to $\O''$. We have $F(z)>x$ and $\O',\O''$ belong to two different connected components of $F^{-1}[0,F(z))$
and to a single connected component of $F^{-1}[0,F(z)]$. The connected component of $F^{-1}[0,F(z))$ containing $\O'$ has empty intersection with $Y$
(by Lemma~\ref{lem:mtl-lemma1}) hence $z$ must be an interior critical point of index $1$. We also remind the reader that all critical points of $F$ on $\O'$
are interior critical points, because $\O'\cap Y=\emptyset$.

Let $W^s$ be the stable manifold of $z$ of the vector field $\nabla F$. Then $W^s\cap\O_1$ must be a connected curve, with non-empty
intersection with $\O'$.
One of its boundaries is either on $\S_0\cap\O'$ or it is a critical point of $F$ in $\O'$, necessarily interior and by the Morse--Smale condition,
its index is $0$.
In the first case, $\S_0\cap\O'$ is not empty and since it is disjoint from $Y$, $\S_0$ has a closed connected component. In the other case, we have in $\O_1$
a single trajectory between a critical point of index $0$ and a critical point of index $1$. This contradicts (TG3).
\end{proof}

\begin{lemma}\label{lem:mtl-lemma3}
Assume that for some $y\in[c,d]$ $\S_1$ and $\S_2$ are two disjoint connected components of $F^{-1}(y)$.
If $F$ has no interior or boundary unstable
critical points of index $n$ with critical value in $[c,y)$, then $\S_1$ and $\S_2$ belong to two different
connected components of $F^{-1}[0,y]$.
\end{lemma}
\begin{proof}
For $x<y$ and close to $y$ the sets $\S_1$ and $\S_2$ lie in two different connected components of $F^{-1}(x,y]$. Let $x_0\geqslant 0$
be the smallest $x$ with that property. Assume $x_0>0$. Then $x_0$ is a critical value of $F$.
The number of connected components of $F^{-1}(x)$ increases as $x$ crosses $x_0$. Thus the corresponding critical point is either
an interior critical point of index $n$, or a boundary unstable critical point of index $n$.
But then $x_0\geqslant c$ because of (TG2), so we have $x_0\in[c,y]$ which contradicts the assumptions of the lemma.

It follows that $x_0=0$. As $F$ has no critical points on $\S_0$, it follows that $\S_1$ and $\S_2$ belong
to different components of $F^{-1}[0,y]$.
\end{proof}

\begin{lemma}\label{lem:mtl-lemma4}
Let $y\in[c,d]$ be chosen so that there are no interior or boundary unstable critical points of index $n$ with critical values in $[c,y)$.
Then $F^{-1}(y)$ has no closed connected components.
\end{lemma}
\begin{proof}
Assume that $\S'$ is a closed connected component of $F^{-1}(y)$. Let $\S''=F^{-1}(y)\setminus\S'$, it is not empty by Lemma~\ref{lem:mtl-lemma1}
(applied for $x=0$, $y=1$), for
otherwise $\O\cap Y=\emptyset$. Let $\O'$ be the connected component of $F^{-1}[0,y]$ containing $\S'$. By Lemma~\ref{lem:mtl-lemma3}, $\O'$ and $\S''$
are disjoint, in particular $\O'\cap F^{-1}(y)\cap Y=\S'\cap Y=\emptyset$. By Lemma~\ref{lem:mtl-lemma1}, $\O'\cap Y=\emptyset$. But this contradicts
Lemma~\ref{lem:mtl-lemma2}.
\end{proof}

\begin{remark}\label{rem:sym2}
There exists a  symmetric formulation of the last three lemmas, which can be obtained by considering the function $1-F$ instead of $F$. For instance,
in Lemma~\ref{lem:mtl-lemma4}, the symmetric
assumption is that there are no interior or
boundary stable critical points of index $1$ in $(y,d]$. The statement is the same.
\end{remark}

\begin{proof}[Proof of Proposition~\ref{prop:mtl}. Case $n>1$.]
Let $x\in[c,d]$ be a non-critical value such that all the critical points of $F$ with index $n$ have critical value greater than $x$,
and all critical points with index smaller than $n$ have critical values smaller than $x$. Such $x$ exists because of (TG1).
If $y\leqslant x$, then Lemma~\ref{lem:mtl-lemma4} guarantees that $F^{-1}(y)$ has no closed connected components. If $y>x$, then $F^{-1}[y,d]$
has no critical points of index $1$ (as $n>1$), so we apply the symmetric counterpart of
Lemma~\ref{lem:mtl-lemma4}.
\end{proof}

\smallskip
\begin{proof}[Proof of Proposition~\ref{prop:mtl}. Case $n=1$.]
The property (TG2) implies that the only critical points of $F|_{[c,d]}$ are the interior critical points of index $1$. Let
us call them $z_1,\dots,z_m$. Since they are all of the same index, by Proposition~\ref{prop:grt} we are able to rearrange the values $F(z_1),\dots,F(z_m)$
at will. Let us fix $c_1,\dots,c_m$ with the property that $c<c_1<\dots<c_m<d$. Let us first rearrange the points $z_1,\dots,z_m$
so that $F(z_1)=\dots=F(z_m)=c_1$.

The singular level set $F^{-1}(c_1)$ is a singular manifold with $m$ singular points $z_1,\dots,z_m$, which are double points.
By Lemma~\ref{lem:mtl-lemma4}, $F^{-1}(c_1)$ has no closed connected components. In particular each of the points $z_1,\dots,z_m$, can be
connected to $Y$ by a curve lying in $F^{-1}(c_1)$. At least one of those points can be connected to $Y$ by a curve $\gamma$, which misses all the other
critical points. We relabel the critical points so that this point is $z_m$. We rearrange the critical points so that $F(z_m)=c_m$ and the value
$F(z_1)=\dots=F(z_{m-1})=c_1$. By construction, $z_m$ can be connected to $Y$ by a curve lying in $F^{-1}(c_m)$.

The procedure now is repeated, i.e. assume that we have already moved $z_{k+1},\dots,z_m$ to levels $c_{k+1},\dots,c_m$ respectively. Then
$F^{-1}(c_1)$ still has no closed connected components by Lemma~\ref{lem:mtl-lemma4}. We assume that $z_k$ can be connected to $Y$ by a curve in $F^{-1}(c_1)$
omitting all the other critical points. Then we rearrange the critical values so that $F(z_k)=c_k$. The proof is accomplished by an inductive argument.
\end{proof}

\subsection{Splitting of cobordisms}

We have now all the ingredients needed to prove our main theorem about
splitting cobordisms. We slightly change the notation in this
subsection; the cobordism will be between $(\S,M)$ and $(\S',M')$.
\begin{theorem}\label{thm:cobsplit}
Let $(\O,Y)$ be a cobordism between $(\S,M)$ and $(\S',M')$. If the following conditions are satisfied
\begin{itemize}
\item $\S$ and $\S'$ have no closed connected components;
\item $\O$ has no closed connected component;
\end{itemize}
Then the relative cobordism can be expressed as a union
\[\O=\O_{0}\cup \O_{1/2}\cup \O_1\cup \O_{3/2}\cup\dots\cup \O_{n+1/2}\cup \O_{n+1}\]
such that $\p \O_s=\S_s\cup\S_{s+1/2}\cup Y_s$ with $\S_0=\S$, $\S_{n+3/2}=\S'$, $Y=Y_{0}\cup\dots\cup Y_{n+1}$. In other words $(\O_s,Y_s)$
is a cobordism between $(\S_s,M_s)$ and $(\S_{s+1/2},M_{s+1/2})$, where $M_s=\p\S_s=\S_s\cap Y_s$. Furthermore
\begin{itemize}
\item[$\bullet$] $(\O_{0},Y_0)$ is a cobordism given by a
    sequence of index $0$ handle attachments;
\item[$\bullet$] for $k=1,\dots,n+1$, $(\O_{k-1/2},Y_{k-1/2})$ is a left product cobordism, given by a sequence of index $k$ left half-handle attachments;
\item[$\bullet$] for $k=1,\dots,n$, $(\O_k,Y_k)$ is a right product cobordism, given by a sequence of index $k$ right half-handle attachments;
\item[$\bullet$] $(\O_{n+1},Y_{n+1})$ is a cobordism provided by a sequence of index $(n+1)$ handle attachments.
\end{itemize}
\end{theorem}
\begin{proof}
Let us begin with a Morse function $F$ on the cobordism
which has only boundary stable critical points (see
Remark~\ref{rem:allarestab}). Assume that $w_1,\dots,w_m$ are the
interior critical points and $y_1,\dots,y_k$ are the boundary
critical points. After a rearrangement of critical points and the cancellation of
pairs of critical points as in
Lemma~\ref{lem:techgood} we can make $F$ technically good. After
applying Theorem~\ref{thm:fullhandsplit} we get that $F$ can have
only $0$ handles and $n+1$ handles as interior handles.
Let us write $\theta=1/(4n+6)$ and choose
$c_0=\theta$, $c_1^s=3\theta$,
$c_1^u=5\theta$, \dots, $c_k^s=(4k-1)\theta$,
$c_k^u=(4k+1)\theta$,\dots, $c_{n+1}^s=1-3\theta$,
$c_{n+1}=1-\theta$. We rearrange the function $F$ according to
Proposition~\ref{prop:grt}. Then we define for $k=0,1/2,1,\dots,n+1$
the manifold $\O_k=F^{-1}[4k\theta,(4k+2)\theta]$, $Y_k=\O_k\cap Y$
and $\S_k=F^{-1}(4k\theta)$.

By construction, each part $(\O_k,Y_k)$ contains
critical points only of one type: for $k=0$ and $n+1$ they are interior critical points,
for $k=1,\dots,n$ they are boundary unstable of index $k$ and for $k=1/2,\dots,n+1/2$,
they are boundary stable of index $k+1/2$ and we conclude the proof by Proposition~\ref{prop:lpc}.
\end{proof}

\begin{remark}
If the cobordism is a product on the boundary, i.e. $Y=M\times
[0,1]$, we can choose the initial Morse function to have no
critical points on the boundary. Then all the critical points of
$F$ come in pairs, $z_j^s$ and $z_j^u$ with $z_j^s$ boundary
stable, $z_j^u$ boundary unstable and there is a single trajectory
of $\nabla F$ going from $z_j^s$ do $z_j^u$.
\end{remark}

The strength of Theorem~\ref{thm:cobsplit} is that it is much easier to study the difference
between the intersection forms on $(\S_{k},M_k)$ and on $(\S_{k\pm 1/2},M_{k\pm 1/2})$. We refer to \cite{BNR2} for an application of this fact.

\section{The cancellation of boundary handles}\label{S:bhc}

In this section we assume that $F$ is a Morse function on the
cobordism $(\O,Y)$ satisfying the Kronheimer--Mrowka--Morse--Smale
regularity condition (Definition~\ref{def:KMMS}). We assume that $F$ has
precisely two critical points $z$ and $w$, with $\ind z=k$ and
$\ind w=k+1$ and that there exists a single
trajectory $\gamma$ of $\nabla F$ going from $z$ to $w$. If $z$
and $w$ are both interior critical points, then
\cite[Theorem~5.4]{Mi-hcob} implies that $(\O,Y)$ is a product
cobordism. In fact, Milnor's proof modifies $F$ only in a small
neighbourhood of $\gamma$, which avoids the boundary $Y$. Hence
it does not matter that in our case the cobordism has a boundary.

We want to extend this result to the case of boundary critical points. In some cases an analogue of the Milnor's theorem holds, in other cases
we can show that it cannot hold.

\subsection{Elementary cancellation theorems}

The following generalization of Milnor's theorem was first obtained in \cite[Theorem 1]{Haj}.

\begin{theorem}\label{thm:elementarycanc}
Let $z$ and $w$ be a boundary critical points of index $k$ and $k+1$, respectively. Assume that
$\gamma$ is a single trajectory joining $z$ and $w$. Furthermore, assume that both $z$ and $w$ are boundary stable, or both boundary unstable.
Then $(\O,Y)$ is a product cobordism.
\end{theorem}

As usual, it is enough to prove the result for
boundary unstable critical
points, the other case is covered if we change $F$ to $1-F$.
Note also, cf.  Section~\ref{ss:ncr}, the assumption that both
critical points are boundary stable, or both boundary unstable is
essential.

A careful reading of Milnor \cite[pages 46--66]{Mi-hcob} shows
that the proof there applies to this situation with only small
modifications. Below we present only three steps of that
proof, adjusted to our situation.
We refer to \cite{Mi-hcob} for
all the missing details.

Let $\xi$ be the gradient vector field of $F$. The proof relies on
the following proposition (see the Preliminary Hypothesis 5.5 in
\cite{Mi-hcob}, proved on pages 55--66).
\begin{proposition}\label{prop:prelhip}
There exist an open neighbourhood $ U$ of $\gamma$ and a
coordinate map $g\colon U\to\R_{\geqslant 0}\times \R^{n}$ and a
gradient-like vector field $\xi'$ agreeing with $\xi$ away from
$U$ such that
\begin{itemize}
\item $g(Y)\subset\{x_1=0\}$, and $g(U)\subset\{x_1\geqslant
    0\}$;
\item $g(z)=(0,0,\dots,0)$;
\item $g(w)=(0,1,\dots,0)$;
\item $g_*\xi'=\eta=(x_1,v(x_2),-x_3,\dots,-x_k,x_{k+1},\dots,x_{n+1})$, where $v$ is a smooth function positive in $(0,1)$, zero at $0$ and $1$
and negative elsewhere. Moreover $\modfrac{dv}{dx_2}=1$ near $0$ and $1$.
\end{itemize}
Furthermore, $U$ can be made arbitrary small (around $\gamma$).
\end{proposition}
Given the proposition, we argue in the same way as in
the classical case, cf. \cite[pages 50--55]{Mi-hcob}: we improve the vector field $\xi'$ in $U$ so that it becomes a gradient like vector
field of a function $F'$, which has no critical points
at all. Then the cobordism is a product cobordism.

The proof of Proposition~\ref{prop:prelhip} is a natural modification of Milnor's proof. After applying arguments
as in \cite[pages 55-58]{Mi-hcob} the proof boils down to the following result.

\begin{proposition}[compare \expandafter{\cite[Theorem 5.6]{Mi-hcob}}]\label{prop:isot}
Let $a+b={n}$, $a\ge 1$ and $b\ge 0$ and write a point $x\in \R_{\geqslant
0}\times\R^{a-1}\times\R^b$ as $(x_a,x_b)$ with
$x_a\in\R_{\geqslant 0}\times\R^{a-1}$ and $x_b\in\R^b$. Assume
that $h\colon(\R_{\geqslant 0}\times\R^{n-1},\{0\}\times\R^{n-1})\to
(\R_{\geqslant 0}\times\R^{n-1},\{0\}\times\R^{n-1})$ is an orientation
preserving diffeomorphism such that $h(0)=0$. Suppose that
$h(\R_{\geqslant 0}\times\R^{a-1}\times\{0\})$ intersects
$\{0\}\times\{0\}\times\R^b$ only at the origin and the
intersection is transverse and the intersection index is $+1$. Then, given any neighbourhood $N$ of
$0\in\R_{\geqslant 0}\times\R^{n-1}$, there exists a smooth isotopy
$h'_t$ for $t\in[0,1]$ of diffeomorphisms from $(\R_{\geqslant
0}\times\R^{n-1},\{0\}\times\R^{n-1})$ to itself with $h'_0=h$ such that
\begin{itemize}
\item[(I)] $h'_t(x)=h(x)$ away from $N$;
\item[(II)] $h_1'(x)=x$ in some small neighbourhood $N_1$ of $0$ such that $\ol{N_1}\subset N$;
\item[(III)] $h_1'(\R_{\geqslant 0}\times\R^{a-1}\times\{0\})\cap
    \{0\}\times\{0\}\times\R^b=\{0\}\in\R_{\geqslant
    0}\times\R^n$.
\end{itemize}
\end{proposition}
\begin{remark}
The transversality assumption from the assumption of Proposition~\ref{prop:isot} is equivalent to the flow of $\xi$ being Morse--Smale.
\end{remark}
The proof of Proposition~\ref{prop:isot} in Milnor's book is given on pages 59--66. We prove here only the analogue of \cite[Lemma~5.7]{Mi-hcob}.
For all other results we refer to Milnor's book.

\begin{lemma}
Let $h$ be as in the hypothesis of Proposition~\ref{prop:isot}. Then
there exists a smooth isotopy $h_t\colon\R_{\geqslant
0}\times\R^{n-1}\to\R_{\geqslant 0}\times\R^{n-1}$, with $h_0$ the
identity map and $h_1=h$, such that for each $t$ we have
$h_t(\R_{\geqslant 0}\times\R^{a-1})\cap\R^b=0$.
\end{lemma}
\begin{proof}
We follow the proof of \cite[Lemma 5.7]{Mi-hcob}.
We shall construct the required isotopy in two steps. First we isotope $h$ by $h_t(x)=\frac{1}{t}h(tx)$.
Then $h_1=h$ and $h_0$ (defined to be the limit as $t\to 0$) is a linear map, the derivative of $h$ at $0$.
If this linear map is an identity, we are done. Otherwise $h_0$ is just a nondegenerate
linear map and clearly it maps $\R_{\geqslant 0}\times\R^{a-1}\times\R^b$
diffeomorphically onto itself. It means that under the decomposition $\R^{n}=\R\oplus\R^{a-1}\oplus\R^b$, $h_0$ has the
following block structure
\[h_0=\begin{pmatrix}
a_{11}&0&0\\
*&A&B\\
*&C&D
\end{pmatrix},\]
where $a_{11}>0$, and stars denote unimportant terms. As $h_0$ is orientation-preserving, $\det h_0>0$.
We can apply a homotopy of linear maps which changes the first
column of $h_0$ to $(a_{11},0,\dots,0)$ and preserves all the
other entries of $h_0$.
We do not change the
determinant and the condition $h_0(\R_{\geqslant 0}\oplus\R^{a-1})\cap\R^b=0$ is
preserved (it means that $a_{11}\det A>0$). Let
\[h_{00}=\begin{pmatrix} A & B\\ C & D\end{pmatrix}.\]
We have $\det h_{0}=a_{11}\det h_{00}$, so $\det h_{00}>0$.
We use the same reasoning as in Milnor's proof to find a homotopy of $h_{00}$ to the identity matrix, finishing the proof.
\end{proof}

\subsection{Non-cancellation results}\label{ss:ncr}

The two results below have completely obvious proofs, we state them to contrast with Theorem~\ref{thm:elementarycanc}.
\begin{lemma}
Assume that a Morse function $F$ on the cobordism $(\O,Y)$ between $(\S_0,M_0)$ and $(\S_1,M_1)$
has two critical points $z$ and $w$. Suppose $z$ is an interior critical point
and $w$ is a boundary critical point. Then $(\O,Y)$ is not a product cobordism.
\end{lemma}
\begin{proof}
$F$ restricted to $Y$ has a single critical point, so the cobordism between $M_0$ and $M_1$ cannot be trivial.
\end{proof}

\begin{lemma}
Suppose that $F$ has two critical points $z$ and $w$. Assume that $z$ is boundary stable and $w$ is boundary unstable. Then $(\O,Y)$
is not a product.
\end{lemma}
\begin{proof}
If it were a product, we would have $H_*(\O,\S_0)=0$. We shall show that this is not the case.

If $F(z)=F(w)$, then there are no trajectories between $z$ and $w$, so by Proposition~\ref{prop:rearii} we can ensure that $F(z)<F(w)$.
So we can always assume that $F(z)\neq F(w)$. For simplicity assume that $F(z)<F(w)$. Let $c$ be a regular value such that
$F(z)<c<F(w)$.

By Lemma~\ref{lem:leftistrivial} $F^{-1}[0,c]\sim
\S_0\times[0,c]$. Then $H_*(\O,\S_0)\cong H_*(\O,F^{-1}[0,c])$.
Now $\O$ arises from $F^{-1}[0,c]$ by a right half-handle
addition, hence $H_*(\O,F^{-1}[0,c])\cong H_*(H,B)$, where $(H,B)$
is the corresponding right half-handle. But $H_*(H,B)$ is not
trivial by Lemma~\ref{lem:natiso} (or
Lemma~\ref{lem:rightislikenormal}).
\end{proof}


\begin{thebibliography}{piglet}
\bibitem[AGV]{AGV} V.~I.~Arnold, S.~Gusein-Zade, A.~Varchenko, \emph{Singularities of differentiable maps. Vol. I. The classification of critical points,
caustics and wave fronts} Monographs in Mathematics, \textbf{82}. Birkh\"auser Boston, Inc., Boston, 1985.

\bibitem[Blo]{Blo} J.~Bloom, \textit{The combinatorics of Morse theory with boundary}, preprint, arxiv:1212.6467.

\bibitem[BNR1]{BNR2} M.~Borodzik, A.~N\'emethi, A.~Ranicki, \emph{Codimension 2 embeddings, algebraic surgery and Seifert forms}, preprint 2012,
arxiv: 1211.5964.

\bibitem[BNR2]{BNR3} M.~Borodzik, A.~N\'emethi, A.~Ranicki, \emph{On the semicontinuity of the mod 2 spectrum of hypersurface singularities},
arxiv: 1210.0798, to appear in J. of Alg. Geom.

\bibitem[Bo]{Bot} R.~Bott, \textit{Morse theory indomitable},
Inst. Hautes \'Etudes Sci. Publ. Math. No. \textbf{68} (1988), 99--114.

\bibitem[Br]{Br} D.~Braess,
\textit{Morse-Theorie f\"ur berandete Mannigfaltigkeiten},
Math. Ann. \textbf{208} (1974), 133--148.

\bibitem[GM]{GMcP} M.~Goresky, R.~McPherson, \emph{Stratified Morse
    theory}, Ergebnisse der Mathematik und ihrer Grenzgebiete (3),
    \textbf{14}. Springer-Verlag, Berlin,

\bibitem[Haj]{Haj} B.~Hajduk, \emph{Minimal $m$--functions}, Fund. Math. \textbf{111}(1981), 179--200.

\bibitem[JR]{JR} A.~Jankowski, R.~Rubinsztein,
\textit{Functions with non-degenerate critical points on manifolds with boundary},
Comment. Math. Prace Mat. \textbf{16} (1972), 99--112.

\bibitem[KM]{KM} P.~Kronheimer, T.~Mrowka, \emph{Monopoles and three manifolds}, New Mathematical Monographs, \textbf{10}.
Cambridge University Press, Cambridge, 2007.

\bibitem[La]{Lau} F.~Laudenbach, \emph{A Morse complex on manifolds with boundary}, Geom. Dedicata \textbf{153} (2011), 47--57.

\bibitem[Mi1]{Mi-morse} J.~Milnor, \emph{Morse Theory}, Annals of Math. Studies \textbf{51}. Princeton Univ. Press, Princeton (1963)

\bibitem[Mi2]{Mi-hcob} J.~Milnor, \emph{Lectures on the $h$-cobordism theorem}, Princeton University Press, Princeton, N.J. 1965.

\bibitem[Ni]{Nic} L.~Nicolaescu, \emph{An invitation to Morse theory},
Universitext. Springer, New York, 2007.

\bibitem[Ra]{Ra} A.~Ranicki, \emph{High-dimensional knot theory. Algebraic surgery in codimension 2.} Springer-Verlag, New York, 1998.

\bibitem[Sa]{Sa} D.~Salamon, \emph{Morse theory, the Conley index and Floer homology},
Bull. London Math. Soc. \textbf{22} (1990), no. 2, 113--140.

\bibitem[Sm1]{Sm1} S.~Smale, \emph{On gradient dynamical systems}, Ann. of Math. (2) \textbf{74} (1961), 199--206.

\bibitem[Sm2]{Sm2} S.~Smale, \emph{On the structure of manifolds}, Amer. J. Math.  \textbf{84} (1962), 387--399.

\bibitem[Wi]{Wi} E.~Witten, \emph{Supersymmetry and Morse theory},
J. Diff. Geo. \textbf{17} (1982), 661--692.

\end{thebibliography}
\end{document}